\documentclass[12pt]{amsart}

\usepackage{amsmath}
\usepackage{amscd}
\usepackage{amssymb}
\usepackage{enumerate}
\usepackage{amsfonts}
\usepackage{graphicx}
\usepackage[all]{xy}
\usepackage{mathrsfs}
\usepackage[hypertex]{hyperref}

\usepackage{bbm}

\newcommand{\limto}{{\displaystyle\lim_{\longrightarrow}}}
\newcommand{\rightlim}{\mathop{\limto}}

\newcommand{\leftlim}{\mathop{\displaystyle\lim_{\longleftarrow}}}
\newcommand{\limfromn}{\leftlim\limits_{\raise3pt\hbox{$n$}}}
\newcommand{\limton}{\rightlim\limits_{\raise3pt\hbox{$n$}}}

\newcommand{\rightlimit}[1]{\mathop{\lim\limits_{\longrightarrow}}\limits%
                   _{\raise3pt\hbox{$\scriptstyle #1$}}}
\newcommand{\leftlimit}[1]{\mathop{\lim\limits_{\longleftarrow}}\limits%
                   _{\raise3pt\hbox{$\scriptstyle #1$}}}

\numberwithin{equation}{section}

\newcommand{\rar}[1]{\stackrel{#1}{\longrightarrow}}

\newcommand{\xrar}[1]{\xrightarrow{#1}}

\newcommand{\into}{\hookrightarrow}

\newcommand{\al}{\alpha}
\newcommand{\be}{\beta}
\newcommand{\ga}{\gamma}
\newcommand{\Ga}{\Gamma}
\newcommand{\de}{\delta}

\newcommand{\la}{\lambda}

\newcommand{\ze}{\zeta}

\newcommand{\sg}{\sigma}

\newcommand{\te}{\theta}

\newcommand{\vp}{\varphi}

\newcommand{\bA}{{\mathbb A}}
\newcommand{\bB}{{\mathbb B}}

\newcommand{\bF}{{\mathbb F}}
\newcommand{\bG}{{\mathbb G}}

\newcommand{\bN}{{\mathbb N}}
\newcommand{\bO}{{\mathbb O}}

\newcommand{\bQ}{{\mathbb Q}}

\newcommand{\bT}{{\mathbb T}}
\newcommand{\bU}{{\mathbb U}}

\newcommand{\bZ}{{\mathbb Z}}

\newcommand{\cE}{{\mathcal E}}
\newcommand{\cF}{{\mathcal F}}

\newcommand{\cL}{{\mathcal L}}

\newcommand{\cO}{{\mathcal O}}

\newcommand{\cR}{{\mathcal R}}

\newcommand{\sD}{{\mathscr D}}

\newcommand{\sF}{{\mathscr F}}

\newcommand{\fG}{{\mathfrak G}}

\newcommand{\fX}{{\mathfrak X}}
\newcommand{\fY}{{\mathfrak Y}}

\newcommand{\fk}{{\mathfrak k}}

\newcommand{\abs}[1]{\lvert #1\rvert}

\newcommand{\Ker}{\operatorname{Ker}}

\newcommand{\End}{\operatorname{End}}
\newcommand{\Hom}{\operatorname{Hom}}

\newcommand{\Spec}{\operatorname{Spec}}

\newcommand{\id}{\operatorname{id}}

\newcommand{\pr}{\mathrm{pr}}

\newcommand{\Ind}{\operatorname{Ind}}

\newcommand{\Tr}{\operatorname{Tr}}
\newcommand{\Nm}{\operatorname{Nm}}

\newcommand{\tens}{\otimes}

\newcommand{\st}{\,\big\vert\,}

\newcommand{\sbr}{\smallbreak}
\newcommand{\mbr}{\medbreak}

\newcommand{\matr}[4]{\left(\begin{array}{cc} \! #1 & \! #2 \! \\ \! #3 & \! #4 \! \end{array}\right)}

\newtheorem{thm}{Theorem}[section]
\newtheorem{cor}[thm]{Corollary}
\newtheorem{lem}[thm]{Lemma}

\newtheorem{prop}[thm]{Proposition}

\newtheorem{conjecture}[thm]{Conjecture}
\newtheorem{probl}{Problem}

\theoremstyle{remark}
\newtheorem{rem}[thm]{Remark}
\newtheorem{rems}[thm]{Remarks}
\newtheorem{example}[thm]{Example}

\newtheorem{defin}[thm]{Definition}

\newcommand{\Fr}{\operatorname{Fr}}

\newcommand{\Gal}{\operatorname{Gal}}

\newcommand{\diag}{\operatorname{diag}}

\newcommand{\ql}{\overline{\bQ}_\ell}
\newcommand{\qls}{\overline{\bQ}_\ell^\times}

\newcommand{\Xt}{\widetilde{X}}
\newcommand{\Gt}{\widetilde{G}}

\newcommand{\Tt}{\widetilde{T}}
\newcommand{\Ut}{\widetilde{U}}

\newcommand{\at}{\widetilde{a}}

\newcommand{\gat}{\widetilde{\ga}}

\newcommand{\psit}{{\widetilde{\psi}}}

\newcommand{\zeb}{{\bar{\zeta}}}

\newcommand{\Knr}{\widehat{K}^{nr}}
\newcommand{\OKnr}{\widehat{\cO}^{nr}_K}

\newcommand{\bfq}{\overline{\bF}_q}
\newcommand{\fqn}{\bF_{q^n}}
\newcommand{\fqq}{\bF_{q^2}}

\newcommand{\UU}{U^{2,q}_3(\bF_{q^2})}

\begin{document}

\title[Deligne-Lusztig constructions]{Deligne-Lusztig constructions for unipotent and $p$-adic groups}

\author{Mitya Boyarchenko}

\begin{abstract}
In 1979 Lusztig proposed a conjectural construction of supercuspidal representations of reductive $p$-adic groups, which is similar to the well known construction of Deligne and Lusztig in the setting of finite reductive groups. We present a general method for explicitly calculating the representations arising from Lusztig's construction and illustrate it with several examples. The techniques we develop also provide background for the author's joint work with Weinstein on a purely local and explicit proof of the local Langlands correspondence.
\end{abstract}

\maketitle

\setcounter{tocdepth}{1}

\tableofcontents

\section*{Introduction}

Let $\bG$ be a connected reductive group over a locally compact non-Archimedean field $K$. In \cite{lusztig-corvallis} Lusztig suggested a cohomological construction of representations of the totally disconnected group $G:=\bG(K)$ associated to characters of unramified maximal tori in $\bG$. We begin by describing (a minor variation of) this construction.

\mbr

Write $\Knr$ for the completion of the maximal unramified extension of $K$. Let $\bT\subset\bG$ be a maximal torus, defined over $K$, such that $\bT\tens_K\Knr$ is contained in a Borel subgroup $\bB$ of $\bG\tens_K\Knr$. Write $\bU$ for the unipotent radical of $\bB$ and put $\Gt=\bG(\Knr)$, $\Tt=\bT(\Knr)$ and $\Ut=\bU(\Knr)$, where all three are viewed as abstract groups for the time being. Finally, let $F:\Gt\rar{}\Gt$ be the group automorphism induced by the (arithmetic) Frobenius element in $\Gal(\Knr/K)$.

\mbr

Define $X:=(\Ut\cap F^{-1}(\Ut))\setminus\bigl\{g\in\Gt\st F(g)g^{-1}\in\Ut\bigr\}$ (the quotient is taken with respect to the left multiplication action of $\Ut\cap F^{-1}(\Ut)$). The groups $G=\bG(K)=\Gt^F$ and $T=\bT(K)=\Tt^F$ act on $X$ via right and left multiplication; the two actions commute (the analogy with \cite[Def.~1.17(ii)]{deligne-lusztig} is apparent). Lusztig suggested that by interpreting $X$ as an (ind-)scheme over $\overline{\bF}_q$, the algebraic closure of the residue field of $K$, one should be able to define homology groups $H_i(X,\ql)$ giving rise to smooth representations of $G\times T$. In that case, for every smooth homomorphism $\theta:T\rar{}\qls$, the $\theta$-isotypic component $H_i(X,\ql)[\theta]$ is a smooth representation of $G$. One expects that for $\theta$ in sufficiently general position, $H_i(X,\ql)[\theta]$ vanishes for all but one $i$, which gives an irreducible representation of $G$, and that this representation is supercuspidal when $\bT$ is anisotropic modulo the center of $\bG$.

\mbr

Lusztig's construction naturally leads to the following three problems.

\begin{probl}\label{prob:formalization}
Formalize the construction above for a general connected reductive group $\bG$, that is, define $X$ as an ind-scheme over $\bfq$ and define its homology groups $H_i(X,\ql)$, in such a way that the action of $G\times T$ on $X$ yields smooth representations of $G\times T$ in $H_i(X,\ql)$ for all $i\geq 0$.
\end{probl}

\begin{probl}\label{prob:computation}
Explicitly compute the representations of $G$ resulting from Lusztig's construction.
\end{probl}

\begin{probl}\label{prob:generalization}
Find a generalization of Lusztig's construction that allows one to remove the assumption that $\bT\tens_K\Knr$ is contained in a Borel subgroup of $\bG\tens_K\Knr$.
\end{probl}

In this article we set aside Problem \ref{prob:formalization} because in the examples we will consider, it can be solved by an \emph{ad hoc} method found by Lusztig \cite{lusztig-corvallis}. On the other hand, Problem \ref{prob:generalization} is very difficult; we should mention that Stasinski \cite{stasinski-extended-deligne-lusztig} made some progress toward its solution in the closely related setting of groups of the form $\bG(\cO_K)$, where $\bG$ is a reductive group scheme over the ring of integers $\cO_K\subset K$.

\mbr

Our aim is to develop general tools that can be applied to the solution of Problem \ref{prob:computation}. Before moving on, we should explain what we mean by ``explicitly computing'' a representation of the group $G=\bG(K)$. For concreteness, let us suppose that we are in a situation where Problem \ref{prob:formalization} has already been solved, and for some $i_0\geq 0$, the representation $H_{i_0}(X,\ql)[\te]$ of $G$ is irreducible and supercuspidal. Then solving Problem \ref{prob:computation} means exhibiting a pair $(H,\rho)$ consisting of an open subgroup $H\subset G$, which is compact modulo the center of $G$, and a smooth irreducible representation $\rho$ of $H$ such that $H_{i_0}(X,\ql)[\te]\cong \Ind_H^G(\rho)$. Both $H$ and $\rho$ must be described explicitly.

\mbr

As remarked earlier, Lusztig's construction has a natural analogue for groups of the form $\bG(\cO_K)$; to formulate it one merely has to replace $K$ with $\cO_K$ and $\Knr$ with $\OKnr$ in the first part of the introduction. In this setting, some examples of the resulting representations were calculated by Lusztig \cite{lusztig-finite-rings} and Stasinski \cite{stasinski-extended-deligne-lusztig}. In those examples, $\bG$ is either $GL_2$ or $SL_2$, and the resulting representations have ``level $\leq 2$,'' namely, they factor through $\bG(\cO_K/P_K^2)$, where $P_K\subset\cO_K$ is the maximal ideal. In particular, they are related to representations of finite groups of the form $\fG(\bF_q)$, where $\fG$ is an algebraic group over the finite field $\bF_q=\cO_K/P_K$ whose unipotent radical is \emph{commutative}.

\mbr

A general solution of Problem \ref{prob:computation} requires the consideration of groups $\fG$ over $\bF_q$ whose unipotent radical may be very complicated. As the examples in Section \ref{s:examples-division} illustrate, the calculation of the representations obtained from Lusztig's construction is closely related to the analysis of \emph{Deligne-Lusztig patterns}, which is a term we use for the following general method of constructing representations of finite groups.

\mbr

Suppose $\fG$ is a connected algebraic group over $\bfq$ and $F:\fG\rar{}\fG$ is a Frobenius relative to an $\bF_q$-rational structure on $\fG$. By a result of Lang \cite{lang-theorem}, the morphism $L_q:\fG\rar{}\fG$ given by $g\mapsto F(g)g^{-1}$ is a finite \'etale cover. Its Galois group is $\fG(\bF_q)=\fG^F$ acting on $\fG$ via right multiplication. If $\fY\subset\fG$ is a subvariety (not necessarily defined over $\bF_q$), then $\fX:=L_q^{-1}(\fY)$ inherits the right $\fG(\bF_q)$-action, so the compactly supported cohomology groups $H^i_c(\fX,\ql)$ become representations of $\fG(\bF_q)$. The original construction of Deligne and Lusztig \cite{deligne-lusztig} can be interpreted as a special case of this pattern. On the other hand, different special cases (in which $\fG$ is a \emph{unipotent} group defined over $\bF_q$) play an important role in an approach to the study of the cohomology of the Lubin-Tate tower over $K$ initiated by Weinstein in \cite{WeinsteinGoodReduction} and continued in \cite{maximal-varieties-LLC}. Thus the main goal of this work is to present a general technique for studying Deligne-Lusztig patterns and explain its applications to Problem \ref{prob:computation} above as well as to the local Langlands correspondence for $GL_n(K)$.


\subsection*{Organization} The rest of the article is organized as follows. Section \ref{s:notation} clarifies all the notation used throughout the article. The general principles for analyzing Deligne-Lusztig patterns for algebraic groups over $\bF_q$ with a large unipotent radical are explained in Section \ref{s:principles}. In Section \ref{s:examples-unipotent} we give examples of results that can be proved using these methods, which have applications in the author's work with Weinstein on the local Langlands correspondence \cite{maximal-varieties-LLC}. In Section \ref{s:construction-division} we present a detailed solution of Problem \ref{prob:formalization} for one particular family of reductive groups over $p$-adic fields, following an outline provided by Lusztig in \cite{lusztig-corvallis}. Several examples of representations of these groups resulting from Lusztig's construction are calculated in Section \ref{s:examples-division}. The proofs of most of the results appearing in this article are postponed until Section \ref{s:proofs}, so as not to disturb the flow of the exposition. The main results of the article are Theorems \ref{t:lusztig-level-2} and \ref{t:main-example}. 

\section{Notation}\label{s:notation}

For the reader's convenience, we collect in one place all the notation used in this article. For brevity, we will use the phrase ``local field'' to refer to a (nondiscrete) locally compact non-Archimedean field, that is, either a finite extension of $\bQ_p$ or the field $\bF_q((\pi))$ of formal Laurent series in one variable over a finite field.

\mbr

The letter $K$ always denotes a local field. We write $\cO_K\subset K$ for its ring of integers and $\pi\in\cO_K$ for a uniformizer, which is chosen once and for all. We denote by $\Knr$ the completion of the maximal unramified extension of $K$ and by $\OKnr$ its ring of integers. The residue field of $K$ is denoted by $\bF_q=\cO_K/(\pi)$; it has $q$ elements. We write $\bfq$ for a fixed algebraic closure of $\bF_q$, which we identify with the residue field of $\Knr$ whenever convenient. The letter $\vp$ is reserved for the Frobenius automorphism of $\Knr$ over $K$ (it induces the automorphism $x\mapsto x^q$ on the residue field of $\Knr$).

\mbr

We write $p$ for the characteristic of $\bF_q$ and $\ell$ for a chosen prime with $\ell\neq p$. We fix an algebraic closure $\ql$ of the field $\bQ_\ell$ of $\ell$-adic numbers. The field of coefficients of all representations and all sheaves is assumed to be $\ql$.

\subsection{$\ell$-adic cohomology}\label{ss:cohomology-notation} We will freely use the formalism of $\ell$-adic cohomology with compact supports, as explained, e.g., in \cite{sga4.5}. We call a ``$\ql$-local system'' what is called a ``$\ql$-faisceau lisse'' in \emph{op.~cit.} If $X$ is a scheme of finite type over a field $\fk$ and $\pr:X\rar{}\fk$ is the structure morphism, then for any $i\geq 0$ we put $H^i_c(X,\ql)=R^i\pr_!(\ql)$, where $\ql$ is the constant local system of rank $1$ on $X$. We view $H^i_c(X,\ql)$ as an $\ell$-adic sheaf over $\Spec\fk$, that is, a finite dimensional vector space over $\ql$ equipped with a continuous action of $\Gal(\fk^{sep}/\fk)$. Similarly, we write $H^i_c(X,\cF)$ in place of $R^i\pr_!(\cF)$ for any (constructible) $\ell$-adic sheaf $\cF$ on $X$.

\mbr

The only cases that are of interest to us are where $\fk$ is either finite or algebraically closed and where $X$ is separated over $\fk$; these situations can be adequately handled by the formalism developed in \cite{sga4.5} and in SGA 4.

\subsection{Notions of Frobenius and the Lang map}\label{ss:Frobenius-notation} The term \emph{Frobenius} is used in several different ways in the article. First, if $X$ is a scheme over $\bF_q$, then $F:X\rar{}X$ will denote the absolute Frobenius (the identity map on the underlying topological space and the map $f\mapsto f^q$ on the structure sheaf). Moreover, if $\fk$ is any field extension of $\bF_q$, then $F$ will also denote the induced endomorphism of $X\tens_{\bF_q}\fk$. When necessary, we will write $F_q$ in place of $F$.

\mbr

An $\ell$-adic sheaf on $\Spec\bF_q$ gives rise to a finite dimensional vector space $V$ over $\ql$ together with a continuous action of $\Gal(\bfq/\bF_q)$; the action of the geometric Frobenius in $\Gal(\bfq/\bF_q)$ will be denoted by $\Fr_q:V\rar{\simeq}V$. We remark that the Tate twist $\ql(1)$ on $\Spec\bF_q$ corresponds to the vector space $\ql$ on which $\Fr_q$ acts as $q^{-1}$. If $X$ is a scheme over $\bF_q$ and $\cF$ is a (constructible) $\ell$-adic sheaf on $X$, for each $i\geq 0$ we can also view $H^i_c(X,\cF)$ as a finite dimensional vector space over $\ql$ equipped with an automorphism $\Fr_q$. As mentioned earlier, the letter $\vp$ will always denote the automorphism of $\Knr$ induced by the arithmetic Frobenius $x\mapsto x^q$ on its residue field $\bfq$, where $K$ is a local field with residue field $\bF_q$.

\mbr

If $G$ is a group scheme over $\bF_q$ and $\fk$ is any field extension of $\bF_q$, we will write $L_q:G\tens_{\bF_q}\fk\rar{}G\tens_{\bF_q}\fk$ for the morphism $g\mapsto F_q(g)g^{-1}$ and call it the \emph{Lang map}.

\section{General principles}\label{s:principles}

\subsection{Setup}\label{ss:DL-setup} In this section, unlike the introduction, the letter $G$ will denote an algebraic group (i.e., a reduced group scheme of finite type) over a finite field $\bF_q$. Suppose that $Y\subset G$ is a (locally closed) subvariety defined over $\bF_q$ and put $X=L_q^{-1}(Y)$, where $L_q:G\rar{}G$ is the Lang map given by $g\mapsto F(g)g^{-1}$ (cf.~\S\ref{ss:Frobenius-notation}). The finite group $G(\bF_q)$ acts on $X$ by right multiplication, so for each $i\geq 0$, we obtain a representation of $G(\bF_q)$ on the finite dimensional vector space $H^i_c(X,\ql)$ over $\ql$. On the other hand, we also have an automorphism $\Fr_q$ of $H^i_c(X,\ql)$ (cf.~\S\ref{ss:cohomology-notation}); the actions of $G(\bF_q)$ and $\Fr_q$ commute. In this section we present some general tools that can be used to attack the following

\begin{probl}\label{prob:patterns}
Compute each $H^i_c(X,\ql)$ as a representation of $G(\bF_q)\times\Fr_q^{\bZ}$.
\end{probl}

For certain classes of unipotent groups $G$ this problem arose in Weinstein's approach to a purely local proof of the local Langlands correspondence \cite{WeinsteinGoodReduction,maximal-varieties-LLC}. On the other hand, we will see in Section \ref{s:examples-division} that it is also related to Problem \ref{prob:computation} above.

\subsection{Computation of spaces of intertwiners}
In the setting of \S\ref{ss:DL-setup}, suppose that $H\subset G$ is a connected subgroup defined over $\bF_q$ and $\chi:H(\bF_q)\rar{}\qls$ is a character. Write $V_\chi$ for the induced representation $\Ind_{H(\bF_q)}^{G(\bF_q)}(\chi)$ of $G(\bF_q)$. The method presented in this section is based on calculating the space of intertwiners $\Hom_{G(\bF_q)}\bigl(V_\chi,H^i_c(X,\ql)\bigr)$ for each $i\geq 0$. In practice, if the dimension of this space is known for sufficiently many pairs $(H,\chi)$, one can often derive a complete description of $H^i_c(X,\ql)$ as a direct sum of irreducible representations of $G(\bF_q)$.

\mbr

To state the next result we need some auxiliary notation. Consider the right multiplication action of $H(\bF_q)$ on $G$ and form the quotient\footnote{Not to be confused with the set of $\bF_q$-points of the homogeneous space $G/H$!} $Q:=G/(H(\bF_q))$ (it is a smooth affine scheme over $\bF_q$). As the Lang map $L_q:G\rar{}G$ is invariant under right multiplication by $H(\bF_q)$, it factors through a morphism $\al:Q\rar{}G$. On the other hand, the quotient map $G\rar{}Q$ is a right $H(\bF_q)$-torsor, so the character $\chi$ yields a $\ql$-local system $\cE_\chi$ of rank $1$ on $Q$. The next result is proved in \S\ref{ss:proof-l:DL-first}.

\begin{lem}\label{l:DL-first}
There is a natural $\Fr_q$-equivariant vector space isomorphism
\begin{equation}\label{e:key-DL-ingredient}
\Hom_{G(\bF_q)}\bigl(V_\chi,H^i_c(X,\ql)\bigr) \cong H^i_c\bigl(\al^{-1}(Y),\cE_\chi\bigl\lvert_{\al^{-1}(Y)}\bigr) \qquad\forall\,i\geq 0.
\end{equation}
\end{lem}

We now make two further assumptions under which the right hand side of \eqref{e:key-DL-ingredient} can be described much more explicitly. Namely, let us suppose that the quotient morphism $G\rar{}G/H$ admits a section $s:G/H\rar{}G$. Assume also that there is an algebraic group morphism $f:H\rar{}\bG_a$ defined over $\bF_q$ such that $\chi=\psi\circ f$ for an additive character $\psi:\bF_q\rar{}\qls$ and let $\cL_\psi$ be the Artin-Schreier local system on $\bG_a$ defined by $\psi$ (in other words, $\cL_\psi$ is induced by the covering $\bG_a\rar{}\bG_a$, $x\mapsto x^q-x$, whose structure group is $\bF_q=\bG_a(\bF_q)$, via the character $\psi$).

\begin{lem}\label{l:DL-second}
There is an isomorphism $\ga:(G/H)\times H\rar{\simeq}Q$ such that $\ga^*\cE_\chi\cong(f\circ\pr_2)^*\cL_\psi$ and $\al\circ\ga=\be$, where $\pr_2:(G/H)\times H\rar{}H$ is the second projection and $\be:(G/H)\times H\rar{}G$ is given by $\be(x,h)=s(F_q(x))\cdot h\cdot s(x)^{-1}$.
\end{lem}

This result is proved in \S\ref{ss:proof-l:DL-second}. Combining Lemmas \ref{l:DL-first} and \ref{l:DL-second} yields

\begin{prop}\label{p:DL-main}
Assume that we are given the following data:
\begin{itemize}
\item an algebraic group $G$ with a connected subgroup $H\subset G$ over $\bF_q$;
 \sbr
\item a section $s:G/H\rar{}G$ of the quotient morphism $G\rar{}G/H$;
 \sbr
\item an algebraic group homomorphism $f:H\rar{}\bG_a$;
 \sbr
\item an additive character $\psi:\bF_q\rar{}\qls$;
 \sbr
\item a locally closed subvariety $Y\subset G$.
\end{itemize}
Set $X=L_q^{-1}(Y)$, the preimage of $Y$ under the Lang map $L_q(g)=F_q(g)g^{-1}$. The right multiplication action of $G(\bF_q)$ on $X$ induces linear representations of $G(\bF_q)$ on the cohomology $H^i_c(X,\ql)$, and for each $i\geq 0$, we have a vector space isomorphism
\begin{equation}\label{e:DL-main}
\Hom_{G(\bF_q)}\bigl(\Ind_{H(\bF_q)}^{G(\bF_q)}(\psi\circ f),H^i_c(X,\ql)\bigr) \cong H^i_c\bigl(\be^{-1}(Y), P^*\cL_\psi\bigr)
\end{equation}
compatible with the action of $\Fr_q$ on both sides (cf.~\S\ref{ss:Frobenius-notation}),
where $\cL_\psi$ is the Artin-Schreier local system on $\bG_a$ corresponding to $\psi$, the morphism $\be:(G/H)\times H\rar{}G$ is given by $\be(x,h)=s(F_q(x))\cdot h\cdot s(x)^{-1}$, and the morphism $P:\be^{-1}(Y)\rar{}\bG_a$ is the composition $\be^{-1}(Y)\into (G/H)\times H \xrar{\ \ \pr_2\ \ }H\rar{f}\bG_a$.
\end{prop}

\begin{rem}
In the situation of Proposition \ref{p:DL-main}, if $G$ is affine and $H$ is contained in the unipotent radical of $G$, then $H$ is a connected unipotent group and the homogeneous space $G/H$ is affine, so the existence of the section $s$ is automatic.
\end{rem}

\subsection{The $\ell$-adic formalism} Our conventions regarding $\ell$-adic sheaves, cohomology and the sheaves-to-functions correspondence are the same as those of \cite{sga4.5}. If $X$ is a scheme of finite type over $\bF_q$ and $\cF$ is a constructible $\ql$-sheaf on $X$, the function corresponding to $\cF$ will be denoted by $t_{\cF}:X(\bF_q)\rar{}\ql$.

\mbr

We write $D^b_c(X,\ql)$ for the bounded derived category of complexes of constructible $\ql$-sheaves on $X$. (In fact, all complexes appearing in this article are concentrated in a single cohomological degree.) As usual, the square brackets $[]$ denote cohomological shift to the left and the parentheses $()$ denote Tate twists.

\mbr

If $f:X\rar{}Y$ is a morphism of schemes of finite type over $\bF_q$, we have the pullback functor $f^*:D^b_c(Y,\ql)\rar{}D^b_c(X,\ql)$ and the functor of (derived) pushforward with compact supports $Rf_!:D^b_c(X,\ql)\rar{}D^b_c(Y,\ql)$. Both of these are compatible with the sheaves-to-functions correspondence in the natural sense.

\mbr

The applications of the methods developed in this section rely heavily on the proper base change theorem and the projection formula. We recall their statements.

\begin{thm}[Proper base change]
If
\[
\xymatrix{
  X' \ar[rr]^{g'} \ar[d]_{f'} & & X \ar[d]^f \\
  Y' \ar[rr]^g & & Y
   }
\]
is a Cartesian diagram of finite type schemes over $\bF_q$, then there is a canonical isomorphism of functors $g^*\circ(Rf_!)\rar{\simeq}Rf'_!\circ g'^* : D^b_c(X,\ql)\rar{}D^b_c(Y',\ql)$.
\end{thm}

\begin{thm}[Projection formula]
If $f:X\rar{}Y$ is a morphism of finite type schemes over $\bF_q$, then for $M\in D^b_c(X,\ql)$ and $N\in D^b_c(Y,\ql)$, we have a natural isomorphism
\[
Rf_!\bigl((f^*N)\tens M\bigr) \cong N\tens Rf_!(M) \qquad\text{in } D^b_c(Y,\ql).
\]
\end{thm}

The proofs of these theorems can be found in SGA 4 and \cite{sga4.5}.

\subsection{Additive characters of finite fields} Consider $\bG_a$ as the additive group over the prime field $\bF_p$. The corresponding Frobenius $F_p:\bG_a\rar{}\bG_a$ acts on the set of $\bF_{p^k}$-points of $\bG_a$ as the map $x\mapsto x^p$ for any $k\in\bN$.

\mbr

Let us fix, once and for all, a nontrivial character $\psi_0:\bF_p\rar{}\qls$. For each $k\in\bN$, the map $\bF_{p^k}\rar{}\Hom(\bF_{p^k},\qls)$ taking $a\in\bF_{p^k}$ to the homomorphism $x\mapsto \psi_0(\Tr_{\bF_{p^k}/\bF_p}(ax))$ is a group isomorphism, which is compatible with the action of $F_p$ on both sides because $\Tr_{\bF_{p^k}/\bF_p}(a^p x)=\Tr_{\bF_{p^k}/\bF_p}(ax^{1/p})$ for all $a,x\in\bF_{p^k}$.

\begin{rem}
If $k,r\in\bN$ are such that $k$ divides $r$, the following square commutes:
\[
\xymatrix{
  \bF_{p^r} \ar[rr] & & \Hom(\bF_{p^r},\qls) \\
  \bF_{p^k} \ar[rr] \ar[u]^{\text{inclusion}} & & \Hom(\bF_{p^k},\qls) \ar[u]_{\text{composition with }\Tr_{\bF_{p^r}/\bF_{p^k}}}
   }
\]
where the horizontal arrows are the isomorphisms that were constructed in the previous paragraph. This observation justifies the validity of the next
\end{rem}

\begin{defin}\label{d:conductor}
If $q$ is a power of $p$ and $n\in\bN$, then given a character $\psi:\bF_{q^n}\rar{}\qls$, there exists a unique integer $1\leq m\leq n$ (which divides $n$) such that $\psi$ factors through the trace map $\Tr_{\bF_{q^n}/\bF_{q^m}}:\bF_{q^n}\rar{}\bF_{q^m}$ and does not factor through the trace map $\bF_{q^n}\rar{}\bF_{q^r}$ for any $1\leq r<m$. We call $q^m$ the \emph{conductor} of $\psi$.
\end{defin}

\begin{rem}
This notion is relative to the choice of $q$. However, in all the situations where we will use it, the prime power $q$ will be specified explicitly.
\end{rem}

\subsection{Inductive calculation of cohomology} In order to use Proposition \ref{p:DL-main} in practice one must be able to solve special cases of the following

\begin{probl}
Let $S$ be a scheme of finite type over $\bF_q$, let $P:S\rar{}\bG_a$ be a morphism defined over $\bF_q$ and let $\cL_\psi$ be the Artin-Schreier local system on $\bG_a$ corresponding to a character $\psi:\bF_q\rar{}\qls$. For each $i\geq 0$, calculate the cohomology $H^i_c(S,P^*\cL_\psi)$ together with the action of $\Fr_q$ on it (see \S\ref{ss:Frobenius-notation}).
\end{probl}

Even if $S$ is an affine space $\bA^r$, this problem is very difficult for general morphisms $P$, and there is no known collection of methods that can be used to solve the problem in all possible cases. However, experience shows that in the situations arising from applications of Proposition \ref{p:DL-main}, the cohomology of $P^*\cL_\psi$ can be computed via an inductive procedure similar to the one used in \cite{maximal-varieties-LLC}. The essence of the idea on which this inductive procedure is based is contained in the following result.

\begin{prop}\label{p:inductive-idea}
Let $q$ be a power of $p$, let $n\in\bN$ and let $\psi:\bF_{q^n}\rar{}\qls$ be a character that has conductor $q^m$ (see Def.~\ref{d:conductor}). Let $S_2$ be a scheme of finite type over $\fqn$, put $S=S_2\times\bA^1$ and suppose that a morphism $P:S\rar{}\bG_a$ has the form
\[
P(x,y)=f(x)^{q^j}y-f(x)^{q^n} y^{q^{n-j}}+P_2(x)
\]
for some integer $j$ not divisible by $m$, where $f,P_2:S_2\rar{}\bG_a$ are two morphisms. Let $S_3\subset S_2$ be the subscheme defined by $f=0$ and let $P_3=P_2\bigl\lvert_{S_3}:S_3\rar{}\bG_a$. Then for all $i\in\bZ$, we have
\[
H^i_c(S,P^*\cL_\psi) \cong H^{i-2}_c(S_3,P_3^*\cL_\psi)(-1)
\]
as vector spaces equipped with an action of $\Fr_q$, where the Tate twist $(-1)$ means that the action of $\Fr_q$ on $H^{i-2}_c(S_3,P_3^*\cL_\psi)$ is multiplied by $q$.
\end{prop}

The proof is given in \S\ref{ss:proof-p:inductive-idea}.

\subsection{The Deligne-Lusztig fixed point formula} When $G$ is a unipotent group over $\bF_q$, the methods outlined earlier in this section appear to be sufficient to solve Problem \ref{prob:patterns} in practice. However, the applications of these methods to Problem \ref{prob:computation} also require the techniques of Deligne-Lusztig theory for finite reductive groups. We now state the fixed point formula proved in \cite[\S3]{deligne-lusztig}; it will be used in the proofs of the main results of Section \ref{s:examples-division} below.

\begin{thm}\label{t:fixed-point-formula-DL}
Let $\sg$ be a finite order automorphism of a separated scheme $X$ of finite type over an algebraically closed field $\fk$ of characteristic $p>0$. Then
\begin{equation}\label{e:fixed-point-formula-DL}
\sum_i (-1)^i \Tr\bigl(\sg^*;H^i_c(X,\ql)\bigr) = \sum_i (-1)^i \Tr\bigl(u^*;H^i_c(X^s,\ql)\bigr),
\end{equation}
where $\sg=s\circ u$ is an ``abstract Jordan decomposition'' --- $s$ and $u$ are powers of $\sg$ such that the order of $s$ (respectively, $u$) is prime to $p$ (respectively, a power of $p$) --- and $X^s$ is the fixed point scheme of $s$ acting on $X$.
\end{thm}

\subsection{Character computations} If, in the situation of Theorem \ref{t:fixed-point-formula-DL}, the order of $\sg$ is a power of $p$, formula \eqref{e:fixed-point-formula-DL} becomes vacuous. In such situations, one can sometimes apply a different fixed point formula, based on a more elementary idea, which we now describe.

\subsubsection{The setting}\label{sss:fixed-point-setting} Assume that $X$ is a separated scheme of finite type over the finite field $\bF_q$ and we are given an automorphism $\sg$ of $X$ and a right action of a finite group $A$ on $X$ that commutes with $\sg$. For each character $\chi:A\rar{}\qls$ and each $i\geq 0$, we write $H^i_c(X,\ql)[\chi]$ for the subspace of $H^i_c(X,\ql)$ on which $A$ acts via $\chi$. This subspace is invariant under the action of $\sg^*:H^i_c(X,\ql)\rar{\simeq}H^i_c(X,\ql)$.

\begin{lem}\label{l:fixed-point-1}
Let $\chi:A\to\qls$ be a character and let $r\geq 0$ be an integer such that $H^i_c(X,\ql)[\chi]=0$ for all $i\neq r$ and the Frobenius $\Fr_q$ acts on $H^r_c(X,\ql)[\chi]$ by a scalar $\la\in\qls$. Then
\begin{equation}\label{e:fixed-point-1}
\Tr\bigl(\sg^*,H^r_c(X,\ql)[\chi]\bigr) = \frac{(-1)^r}{\la\cdot\abs{A}} \cdot \sum_{a\in A} \chi(a) \cdot \Bigl\lvert \bigl\{ x\in X(\bfq) \st \sg(F_q(x))= x\cdot a \bigr\} \Bigr\rvert
\end{equation}
where $\abs{S}$ denotes the number of elements of a finite set $S$.
\end{lem}

The lemma is proved in \S\ref{ss:proof-l:fixed-point-1}.

\subsubsection{A variant}\label{sss:fixed-point-variant} Let $X$ be a separated scheme of finite type over the finite field $\bF_q$ and let $\Ga$ and $A$ be two finite groups. Assume that we are given a left action of $\Ga$ on $X$, a right action of $A$ on $X$, and these actions commute with each other. Let $\te:\Ga\rar{}\qls$ and $\chi:A\rar{}\qls$ be characters. For each $i\geq 0$, write
\[
H^i_c(X,\ql)_{\te,\chi} \subset H^i_c(X,\ql)
\]
for the subspace on which $\Ga$ acts via $\te$ and $A$ acts via $\chi$.

\begin{lem}\label{l:fixed-point-2}
Suppose that there exists $r\geq 0$ such that $H^i_c(X,\ql)_{\te,\chi}=0$ for all $i\neq r$ and the Frobenius $\Fr_q$ acts on $H^r_c(X,\ql)_{\te,\chi}$ by a scalar $\la\in\qls$. Then
\[
\dim H^r_c(X,\ql)_{\te,\chi} = \frac{(-1)^r}{\la\cdot\abs{\Ga}\cdot\abs{A}} \sum_{\ga\in\Ga, a\in A} \te(\ga)^{-1}\chi(a) \cdot \Bigl\lvert \bigl\{ x\in X(\bfq) \st \ga\cdot F_q(x)= x\cdot a \bigr\} \Bigr\rvert.
\]
\end{lem}

This result is proved in \S\ref{ss:proof-l:fixed-point-2}.

\vfill\newpage

\section{Examples for unipotent groups}\label{s:examples-unipotent}

In this section we state two results that can be proved using the methods described in Section \ref{s:principles}. Both have applications to the local Langlands correspondence.

\subsection{First setting}\label{ss:unipotent-first} Let $q$ be a power of a prime $p$ and let $n\geq 1$ be an integer.

\begin{defin}\label{d:ring-scheme-level-2}
If $A$ is an $\bF_p$-algebra, consider the twisted polynomial ring $A\langle\tau\rangle$ with the commutation relation $\tau\cdot a=a^q\cdot\tau$ for all $a\in A$, and put $\cR_{n,q}(A)=A\langle\tau\rangle/(\tau^{n+1})$. The functor $A\mapsto\cR_{n,q}(A)$ is obviously representable by the affine space $\bA^{n+1}$ over $\bF_p$.
Thus it defines a ring scheme $\cR_{n,q}$ over $\bF_p$. Its elements will be written as $a_0+a_1\tau+\dotsc+a_n\tau^n$. The multiplicative group of $\cR_{n,q}$ will be denoted by $\cR^\times_{n,q}$ and the subgroup of $\cR^\times_{n,q}$ consisting of all elements of the form $1+a_1\tau+\dotsc+a_n\tau^n$ will be denoted by $U^{n,q}\subset\cR_{n,q}^\times$ (it is an $n$-dimensional unipotent group over $\bF_p$).
\end{defin}

\begin{rem}\label{r:Frobenius-R-n-q}
The Frobenius morphism $F_p:\cR_{n,q}\rar{}\cR_{n,q}$ is given by
\[
F_p(a_0+a_1\tau+\dotsc+a_n\tau^n)=a_0^p+a_1^p\tau+\dotsc+a_n^p\tau^n.
\]
If $r\geq 1$ is any integer and we view $\cR_{n,q}$ as a scheme over $\bF_{p^r}$ by base change, the corresponding Frobenius equals $F_{p^r}=F_p^r$.
\end{rem}

\begin{defin}\label{d:X-main}
Fix $n$ and $q$ as above and view $U^{n,q}$ as an algebraic group \emph{over $\bF_{q^n}$} by extension of scalars. Let $Y\subset U^{n,q}$ be the subscheme defined by $a_n=0$, consider the Lang map $L_{q^n}:U^{n,q}\rar{}U^{n,q}$ given by $g\mapsto F_{q^n}(g)g^{-1}$ and put $X=L_{q^n}^{-1}(Y)$.
\end{defin}

\begin{rem}
By construction, $L_{q^n}:X\rar{}Y$ is a torsor with respect to the right multiplication action of the finite group $U^{n,q}(\bF_{q^n})$. In particular, for each $i\geq 0$, we obtain a linear action of $U^{n,q}(\bF_{q^n})$ on $H^i_c(X,\ql)$, which commutes with the action of the Frobenius $\Fr_{q^n}$ (cf.~\S\S\ref{ss:cohomology-notation}, \ref{ss:Frobenius-notation}). It is calculated explicitly in Theorem \ref{t:cohomology-first-variety}.
\end{rem}

\begin{rem}\label{r:central-characters-1}
If $Z\subset U^{n,q}$ consists of expressions of the form $1+a_n\tau^n$, then $Z$ is the center of $U^{n,q}$ and $Z(\bF_{q^n})$ is the center of $U^{n,q}(\bF_{q^n})$. We have $Z\cong\bG_a$, and we often tacitly identify the two groups. In particular, every irreducible representation of $U^{n,q}(\bF_{q^n})$ over $\ql$ has a central character $\bF_{q^n}\rar{}\qls$.
\end{rem}

\begin{rem}\label{r:reduced-norm-1}
Fix a local field $K$ with finite residue field $\bF_q$ and let $L\subset\Knr$ be the unique unramified extension of $K$ of degree $n$. Consider the twisted polynomial ring $L\langle\Pi\rangle$, where the relation is $\Pi\cdot x=\vp(x)\cdot\Pi$ for all $x\in L$ and $\vp\in\Gal(L/K)$ is induced by the Frobenius automorphism of $\Knr$. Choose a uniformizer $\pi\in K$ and let $D_{1/n}$ denote the quotient of $L\langle\Pi\rangle$ by the relation $\Pi^n=\pi$. Then $D_{1/n}$ is a central division algebra over $K$ with invariant $1/n$ and $\Pi$ is a uniformizer in $D_{1/n}$.

\mbr

Let $\cO_D\subset D_{1/n}$ denote the ring of integers (the unique maximal order). Then the quotient group $(1+\Pi\cO_D)/(1+\Pi^{n+1}\cO_D)$ can be naturally identified with $U^{n,q}(\bF_{q^n})$ via the isomorphism
\begin{equation}\label{e:identify-U}
U^{n,q}(\bF_{q^n}) \rar{\simeq} \frac{1+\Pi\cO_D}{1+\Pi^{n+1}\cO_D}, \qquad 1+\sum_{j=1}^n a_j\tau^j \longmapsto 1+\sum_{j=1}^n \iota(a_j)\Pi^j,
\end{equation}
where $\iota$ is the natural inclusion of $\bF_{q^n}=\cO_L/(\pi)$ into $\cO_D/(\pi)=\cO_D/(\Pi^n)$. Moreover, the reduced norm homomorphism $D^\times_{1/n}\rar{}K^\times$ induces a map
\[
(1+\Pi\cO_D)/(1+\Pi^{n+1}\cO_D) \rar{} (1+\pi\cO_K)/(1+\pi^2\cO_K).
\]
The right hand side is naturally isomorphic to the additive group of $\bF_q$, so we obtain a group homomorphism $\Nm^{n,q}:U^{n,q}(\bF_{q^n})\rar{}\bF_q$. 

\mbr

This map plays the following role in the study of representations of $U^{n,q}(\bF_{q^n})$. The restriction of $\Nm^{n,q}$ to $Z(\bF_{q^n})=\bF_{q^n}$ is equal to the trace map $\Tr_{\bF_{q^n}/\bF_q}$. In particular, given a character $\psi:Z(\bF_{q^n})\rar{}\qls$ with conductor\footnote{We recall that the notion of conductor was given in Definition \ref{d:conductor}.} $q$, we obtain a preferred extension of $\psi$ to a character of $U^{n,q}(\bF_{q^n})$. Namely, if $\psi=\psi_1\circ\Tr_{\bF_{q^n}/\bF_q}$, where $\psi_1:\bF_q\rar{}\qls$, then $\psi_1\circ\Nm^{n,q}:U^{n,q}(\bF_{q^n})\rar{}\qls$ extends $\psi$.
\end{rem}

\begin{rem}\label{r:inclusion-1}
Suppose that $n=m\cdot n_1$, where $m,n_1\in\bN$, and put $q_1=q^m$, so that $q_1^{n_1}=q^n$. We can consider the unipotent group $U^{n_1,q_1}$ over $\bF_{q^n}$. To avoid confusion, let us temporarily denote its elements by $1+b_1\tau_1+\dotsb+b_{n_1}\tau_1^{n_1}$. We can naturally embed $U^{n_1,q_1}$ as a subgroup of $U^{n,q}$ via the map
\[
1+b_1\tau_1+b_2\tau_1^2+\dotsb+b_{n_1}\tau_1^{n_1} \longmapsto 1 + b_1 \tau^m + b_2 \tau^{2m} + \dotsb + b_{n_1} \tau^{n}.
\]
From now on we identify $U^{n_1,q_1}$ with its image under this embedding. In particular, we view $U^{n_1,q_1}(\bF_{q^n})$ as the subgroup of $U^{n,q}(\bF_{q^n})$ consisting of all elements of the form $1+\sum_{m\mid j}a_j\tau^j$, where each $a_j\in\bF_{q^n}$.
\end{rem}

\begin{thm}\label{t:cohomology-first-variety}
Fix an arbitrary character $\psi:\bF_{q^n}\rar{}\qls$.
 \sbr
\begin{enumerate}[$($a$)$]
\item There is a unique $($up to isomorphism$)$ irreducible representation $\rho_\psi$ of $U^{n,q}(\bF_{q^n})$ with central character $\psi$ that occurs in $H^\bullet_c(X,\ql)=\bigoplus_{i\in\bZ} H^i_c(X,\ql)$. Moreover, the multiplicity of $\rho_\psi$ in $H^\bullet_c(X,\ql)$ as a representation of $U^{n,q}(\bF_{q^n})$ equals $1$.
 \sbr
\item Let $\psi$ have conductor $q^m$, so that $n=mn_1$ for some $n_1\in\bN$. Then $\rho_\psi$ occurs in $H^{n+n_1-2}_c(X,\ql)$, and $\Fr_{q^n}$ acts on it via the scalar $(-1)^{n-n_1}\cdot q^{n(n+n_1-2)/2}$.
 \sbr
\item The representation $\rho_\psi$ can be constructed as follows. Write $\psi=\psi_1\circ\Tr_{\bF_{q^n}/\bF_{q_1}}$ for a unique character $\psi_1:\bF_{q_1}\rar{}\qls$, where $q_1=q^m$ as in Remark \ref{r:inclusion-1}. Put
    \[
    H_m = \Bigl\{ 1 + \sum_{\substack{j\leq n/2 \\ m\mid j}} a_j\tau^j + \sum_{n/2<j\leq n} a_j\tau^j \Bigr\} \subset U^{n,q},
    \]
    a connected subgroup. The projection $\nu_m:H_m\rar{}U^{n_1,q_1}$ obtained by discarding all summands $a_j\tau^j$ with $m\nmid j$ $($cf.~Remark \ref{r:inclusion-1}$)$ is a group homomorphism, and $\widetilde{\psi}:=\psi_1\circ\Nm^{n_1,q_1}\circ\nu_m$ is a character of $H_m(\bF_{q^n})$ that extends $\psi:Z(\bF_{q^n})\rar{}\qls$ $($see Remark \ref{r:reduced-norm-1}$)$. With this notation:
    \begin{itemize}
    \item if $m$ is odd \underline{or} $n_1$ is even, then $\rho_\psi\cong\Ind_{H_m(\bF_{q^n})}^{U^{n,q}(\bF_{q^n})}(\widetilde{\psi})$;
    \item if $m$ is even \underline{and} $n_1$ is odd, then $\Ind_{H_m(\bF_{q^n})}^{U^{n,q}(\bF_{q^n})}(\widetilde{\psi})$ is isomorphic to a direct sum of $q^{n/2}$ copies of $\rho_\psi$. Moreover, in this case, if $\Ga_m\subset U^{n,q}(\bF_{q^n})$ is the subgroup consisting of all elements of the form $h+a_{n/2}\tau^{n/2}$, where $h\in H_m(\bF_{q^n})$ and $a_{n/2}\in\bF_{q^{n/2}}$, then $\widetilde{\psi}$ can be extended to a character of $\Ga_m$, and if $\chi:\Ga_m\rar{}\qls$ is any such extension, then $\rho_\psi\cong\Ind_{\Ga_m}^{U^{n,q}(\bF_{q^n})}(\chi)$.
    \end{itemize}
\end{enumerate}
\end{thm}

This result appears as Theorem 4.4 of \cite{maximal-varieties-LLC}.

\subsection{Second setting}\label{ss:unipotent-second} The setup and results of this subsection are parallel to those of \S\ref{ss:unipotent-first}. The only difference is that $U^{n,q}$ is replaced with a different unipotent group over $\bF_p$. We denote the new group by $G^{n,q}$, and it is defined as follows. If $A$ is any commutative $\bF_p$-algebra, $G^{n,q}(A)$ consists of expressions of the form $1+a_1\cdot e_1+a_2\cdot e_2+\dotsb+a_n\cdot e_n$, where $a_j\in A$ for all $1\leq j\leq n$ and $e_1,e_2,\dotsc,e_n$ are formal symbols. These expressions are multiplied using distributivity and the rules
\[
e_i\cdot e_j = \begin{cases}
e_n & \text{if } i+j=n, \\
0 & \text{if } i+j\neq n;
\end{cases}
\]
\[
e_j\cdot a=a^{q^j}\cdot e_j \qquad\forall\,1\leq j\leq n,\ a\in A.
\]

\mbr

From now on we view $G^{n,q}$ as an algebraic group over $\bF_{q^n}$. Let $Y'\subset G^{n,q}$ be defined by the equation $a_n=0$. Write $L_{q^n}:G^{n,q}\rar{}G^{n,q}$ for the Lang map $g\mapsto F_{q^n}(g) g^{-1}$ and put $X'=L_{q^n}^{-1}(Y')$. Then $G^{n,q}(\bF_{q^n})$ acts on $X'$ by right multiplication. In Theorem \ref{t:cohomology-second-variety} we compute each cohomology group $H^i_c(X',\ql)$ as a representation of $G^{n,q}(\bF_{q^n})$ together with the action of the Frobenius $\Fr_{q^n}$ on it.

\begin{rem}\label{r:central-characters-2}
If $Z'\subset G^{n,q}$ consists of expressions of the form $1+a_n\cdot e_n$, then $Z'$ is the center of $G^{n,q}$ and $Z'(\bF_{q^n})$ is the center of $G^{n,q}(\bF_{q^n})$. We have $Z'\cong\bG_a$, and we often tacitly identify the two groups. In particular, every irreducible representation of $G^{n,q}(\bF_{q^n})$ over $\ql$ has a central character $\bF_{q^n}\rar{}\qls$.
\end{rem}

\begin{rem}\label{r:reduced-norm-2}
Let $K,L,D_{1/n},\pi$ be as in Remark \ref{r:reduced-norm-1}. Write $\cO_L\subset\cO_D$ for the ring of integers of $E$ and denote by $C$ the orthogonal complement of $\cO_L$ in $\cO_D$ with respect to the reduced trace pairing $\cO_D\times\cO_D\rar{}\cO_K$. Thus
\[
C = \cO_L\cdot\Pi \oplus \cO_L\cdot\Pi^2 \oplus \dotsb \oplus \cO_L\cdot\Pi^{n-1}.
\]
Next fix an arbitrary integer $k\geq 1$. It is easy to check that $1+\pi^{2k+1}\cO_L+\pi^k C$ is a subgroup of $\cO_D^\times$ and $1+\pi^{2k+2}\cO_L+\pi^{k+1}C$ is a normal subgroup of $1+\pi^{2k+1}\cO_L+\pi^k C$. We have an identification
\[
G^{n,q}(\bF_{q^n}) \rar{\simeq} \frac{1+\pi^{2k+1}\cO_L+\pi^k C}{1+\pi^{2k+2}\cO_L+\pi^{k+1}C}
\]
obtained by sending $1+a_1e_1+\dotsb+a_ne_n$ to the image of \[
1+\widetilde{a}_n\cdot\pi^{2k+1}+\pi^k\cdot(\widetilde{a}_1\Pi+\widetilde{a}_2\Pi^2+\dotsb+\widetilde{a}_{n-1}\Pi^{n-1})
\]
in the quotient of the right hand side, where $\widetilde{a}_j$ are arbitrary lifts of the $a_j\in\bF_{q^n}$ to elements of $\cO_L$ (here the residue field of $L$ is identified with $\bF_{q^n}$).

\mbr

Furthermore, it is not hard to see that the reduced norm map $D_{1/n}^\times\rar{}K^\times$ takes $1+\pi^{2k+1}\cO_L+\pi^k C$ into $1+\pi^{2k+1}\cO_K$ and $1+\pi^{2k+2}\cO_L+\pi^{k+1}C$ into $1+\pi^{2k+2}\cO_K$. So after passing to the quotients we obtain a homomorphism
\[
\frac{1+\pi^{2k+1}\cO_L+\pi^k C}{1+\pi^{2k+2}\cO_L+\pi^{k+1}C} \rar{} \frac{1+\pi^{2k+1}\cO_K}{1+\pi^{2k+2}\cO_K}.
\]
Identifying the left hand side with $G^{n,q}(\bF_{q^n})$ as explained above and identifying the right hand side with the additive group of $\bF_q$ in the usual way, we obtain a group homomorphism, which we again denote by $\Nm^{n,q}:G^{n,q}(\bF_{q^n})\rar{}\bF_q$.
\end{rem}

\begin{rem}\label{r:inclusion-2}
Suppose that $n=m\cdot n_1$, where $m,n_1\in\bN$, and put $q_1=q^m$, so that $q_1^{n_1}=q^n$. We can consider the unipotent group $G^{n_1,q_1}$ over $\bF_{q^n}$. To avoid confusion, let us temporarily denote its elements by $1+b_1 \cdot e'_1+\dotsb+b_{n_1}\cdot e'_{n_1}$. We can naturally embed $G^{n_1,q_1}$ as a subgroup of $G^{n,q}$ via the map
\[
1 + b_1 e'_1 + b_2 e'_2 + \dotsb + b_{n_1} e'_{n_1} \longmapsto 1 + b_1 e_m + b_2 e_{2m} + \dotsb + b_{n_1} e_n.
\]
From now on we identify $G^{n_1,q_1}$ with its image under this embedding. In particular, we view $G^{n_1,q_1}(\bF_{q^n})$ as the subgroup of $G^{n,q}(\bF_{q^n})$ consisting of all elements of the form $1+\sum_{m\mid j}a_j e_j$, where each $a_j\in\bF_{q^n}$.
\end{rem}

\begin{thm}\label{t:cohomology-second-variety}
Fix an arbitrary character $\psi:\bF_{q^n}\rar{}\qls$.
 \sbr
\begin{enumerate}[$($a$)$]
\item There is a unique $($up to isomorphism$)$ irreducible representation $\rho'_\psi$ of $G^{n,q}(\bF_{q^n})$ with central character $\psi$ that occurs in $H^\bullet_c(X',\ql)=\bigoplus_{i\in\bZ} H^i_c(X',\ql)$. The multiplicity of $\rho'_\psi$ in $H^\bullet_c(X',\ql)$ as a representation of $G^{n,q}(\bF_{q^n})$ is equal to $1$.
 \sbr
\item Let $\psi$ have conductor $q^m$, so that $n=mn_1$ for some $n_1\in\bN$. Then $\rho'_\psi$ occurs in $H^{n+n_1-2}_c(X',\ql)$, and $\Fr_{q^n}$ acts on it via the scalar $(-1)^{n-n_1}\cdot q^{n(n+n_1-2)/2}$.
 \sbr
\item The representation $\rho'_\psi$ can be constructed as follows. Write $\psi=\psi_1\circ\Tr_{\bF_{q^n}/\bF_{q_1}}$ for a unique character $\psi_1:\bF_{q_1}\rar{}\qls$, where $q_1=q^m$ as in Remark \ref{r:inclusion-2}. Put
    \[
    H'_m = \Bigl\{ 1 + \sum_{\substack{j\leq n/2 \\ m\mid j}} a_j e_j + \sum_{n/2<j\leq n} a_j e_j \Bigr\} \subset G^{n,q},
    \]
    a connected subgroup. The projection $\nu'_m:H'_m\rar{}G^{n_1,q_1}$ obtained by discarding all summands $a_j e_j$ with $m\nmid j$ $($cf.~Remark \ref{r:inclusion-2}$)$ is a group homomorphism, and $\widetilde{\psi}':=\psi_1\circ\Nm^{n_1,q_1}\circ\nu'_m$ is a character of $H_m(\bF_{q^n})$ that extends $\psi:Z(\bF_{q^n})\rar{}\qls$ $($see Remark \ref{r:reduced-norm-2}$)$. With this notation:
    \begin{itemize}
    \item if $m$ is odd \underline{or} $n_1$ is even, then $\rho'_\psi\cong\Ind_{H'_m(\bF_{q^n})}^{G^{n,q}(\bF_{q^n})}(\widetilde{\psi}')$;
    \item if $m$ is even \underline{and} $n_1$ is odd, then $\Ind_{H'_m(\bF_{q^n})}^{G^{n,q}(\bF_{q^n})}(\widetilde{\psi}')$ is isomorphic to a direct sum of $q^{n/2}$ copies of $\rho'_\psi$. Moreover, in this case, if $\Ga'_m\subset G^{n,q}(\bF_{q^n})$ is the subgroup consisting of all elements of the form $h+a_{n/2}e_{n/2}$, where $h\in H'_m(\bF_{q^n})$ and $a_{n/2}\in\bF_{q^{n/2}}$, then $\widetilde{\psi}'$ can be extended to a character of $\Ga'_m$, and if $\chi:\Ga'_m\rar{}\qls$ is any such extension, then $\rho'_\psi\cong\Ind_{\Ga'_m}^{G^{n,q}(\bF_{q^n})}(\chi)$.
    \end{itemize}
\end{enumerate}
\end{thm}

Theorem \ref{t:cohomology-second-variety} can be proved by a method that is essentially identical to the one used in \cite{maximal-varieties-LLC} to prove Theorem \ref{t:cohomology-first-variety}.

\section{Lusztig's construction for division algebras}\label{s:construction-division}

If $K$ is a local field and $n\geq 1$ is an integer, there exist a connected reductive group $\bG$ over $K$ such that $\bG(K)$ is isomorphic to the multiplicative group of the central division algebra with invariant $1/n$ over $K$, and a $K$-rational maximal torus $\bT\subset\bG$ such that $\bT(K)$ is isomorphic to the multiplicative group of the degree $n$ unramified extension of $K$. Lusztig's construction, which we described in the Introduction, should associate a family of smooth representations $H_i(X,\ql)[\te]$ of $\bG(K)$ to every smooth homomorphism $\te:\bT(K)\rar{}\qls$. In this section we explain how to formalize the construction. The ideas are taken from \cite[\S2]{lusztig-corvallis}; however, we provide the details of all the proofs and work in a slightly different setup.

\subsection{Setting}\label{ss:Lusztig-setting} In this section $K$ denotes a local field of arbitrary characteristic, with chosen uniformizer $\pi$ and residue field $\bF_q=\cO_K/(\pi)$. We will write $\Knr$ for the completion of the maximal unramified extension of $K$ and $\OKnr$ for the ring of integers of $\Knr$. The Frobenius automorphism of $\Knr$ is denoted by $\vp$ (it induces the automorphism $x\mapsto x^q$ on the residue field $\overline{\bF}_q$ of $\Knr$).

\mbr

Consider the matrix\footnote{The lower left entry of $\varpi$ is $\pi$; each superdiagonal entry is $1$; all other entries are $0$.}
\begin{equation}\label{e:uniformizer-division}
\varpi=\begin{pmatrix} 0 & 1 & 0 & \dotsc & 0 \\  & 0 & 1 & \dotsc & 0 \\ & & \ddots & \ddots & \\ & & & 0 & 1 \\ \pi & & & & 0 \end{pmatrix}
\end{equation}
The homomorphism $F:GL_n(\Knr)\rar{}GL_n(\Knr)$ given by $F(A)=\varpi^{-1} A^\vp\varpi$ is a Frobenius relative to a $K$-rational structure; we denote the corresponding algebraic group over $K$ by $\bG$ (it is an unramified inner form of $GL_n$). Here, for any matrix $A$ with entries in $\Knr$, we denote by $A^\vp$ the result of applying $\vp$ to each entry of $A$.

\mbr

It is useful to have an explicit formula for $F$: if $A=\bigl(a_{ij}\bigr)_{i,j=1}^n$, then
\begin{equation}\label{e:Frobenius-division-explicit}
F(A) = \begin{pmatrix}
\vp(a_{n,n}) & \pi^{-1}\cdot\vp(a_{n,1}) & \pi^{-1}\cdot\vp(a_{n,2}) & \dotsc & \pi^{-1} \cdot\vp(a_{n,n-1}) \\
\pi\vp(a_{1,n}) & \vp(a_{1,1}) & \vp(a_{1,2}) & \dotsc & \vp(a_{1,n-1}) \\
\pi\vp(a_{2,n}) & \vp(a_{2,1}) & \vp(a_{2,2}) & \dotsc & \vp(a_{2,n-1}) \\
\vdots & \vdots & \vdots & \ddots & \vdots \\
\pi\vp(a_{n-1,n}) & \vp(a_{n-1,1}) & \vp(a_{n-1,2}) & \dotsc & \vp(a_{n-1,n-1})
\end{pmatrix},
\end{equation}
as one can easily check by a direct calculation.

\begin{rem}\label{r:identify-rational-points}
Let $D_{1/n}$ be the central division algebra with invariant $1/n$ over $K$. Then the multiplicative group $D_{1/n}^\times$ can be naturally identified with the group $\bG(K)$ of rational points of $\bG$. Indeed, let us use the presentation $D_{1/n}=L\langle\Pi\rangle\bigl/(\Pi^n-\pi)$ given in Remark \ref{r:reduced-norm-1}.
Given $a\in\Knr$, we write $\at=\diag\bigl(a,\vp(a),\dotsc,\vp^{n-1}(a)\bigr)$, a diagonal matrix with entries in $\Knr$. One can easily check that $\varpi\cdot\at=\widetilde{\vp(a)}\cdot\varpi$ for all $a\in L$, and that there is an embedding of $K$-algebras $D_{1/n}\into Mat_n(\Knr)$ determined by $a\mapsto\at$ for all $a\in L$ and $\Pi\mapsto\varpi$, which induces an isomorphism
\begin{equation}\label{e:identify-rational-points}
D_{1/n}^\times\rar{\simeq}\bigl\{ A\in GL_n(\Knr) \st A^\vp=\varpi A\varpi^{-1} \bigr\} \overset{\text{def}}{=} \bG(K).
\end{equation}
\end{rem}

The subgroup of $GL_n(\Knr)$ consisting of diagonal matrices is $F$-invariant and hence corresponds to a maximal torus $\bT\subset\bG$ defined over $K$. The isomorphism \eqref{e:identify-rational-points} induces an isomorphism $L^\times\rar{\simeq}\bT(K)$.

\subsection{Calculation of the quotient $\Xt$}\label{ss:lusztig-calculation-quotient} By construction, $\Gt=\bG(\Knr)=GL_n(\Knr)$ and the subgroup $\Tt=\bT(\Knr)\subset\Gt$ consists of all diagonal matrices. Let $\bB\subset\bG\tens_K\Knr$ be the Borel subgroup consisting of upper triangular matrices. If $\bU$ is its unipotent radical, then $\Ut=\bU(\Knr)\subset GL_n(\Knr)$ consists of unipotent upper triangular matrices. Further, we let $\Ut^-\subset GL_n(\Knr)$ denote the subgroup consisting of unipotent \emph{lower} triangular matrices.

\begin{lem}\label{l:quotient-calculation}
The map $\bigl(\Ut\cap F^{-1}(\Ut)\bigr)\times\bigl(\Ut\cap F(\Ut^-)\bigr)\rar{}\Ut$ given by $(B,g)\mapsto F(B)gB^{-1}$ is a bijection.
\end{lem}

The proof is given in \S\ref{ss:proof-l:quotient-calculation}. The lemma immediately implies

\begin{cor}
Given $A\in GL_n(\Knr)$ satisfying $F(A)A^{-1}\in\Ut$, there is a unique $B\in\Ut\cap F^{-1}(\Ut)$ such that $F(BA)\cdot(BA)^{-1}\in\Ut\cap F(\Ut^-)$.
\end{cor}

Hence we can identify the quotient
\[
(\Ut\cap F^{-1}(\Ut))\setminus\bigl\{A\in GL_n(\Knr)\st F(A)A^{-1}\in\Ut\bigr\}
\]
appearing in Lusztig's construction with the set
\begin{equation}\label{e:X-tilde}
\Xt:=\bigl\{A\in GL_n(\Knr)\st F(A)A^{-1}\in \Ut\cap F(\Ut^-)\bigr\}.
\end{equation}

In the next three lemmas we explain how to interpret $\Xt$ as the set of $\overline{\bF}_q$-points of an ind-scheme defined over $\overline{\bF}_q$ and how to define $H_*(\Xt,\ql)$.

\begin{lem}\label{l:form-of-matrices}
A matrix $A\in GL_n(\Knr)$ belongs to $\Xt$ if and only if it has the form
\begin{equation}\label{e:form-of-matrices}
A=\begin{pmatrix}
a_0 & a_1 & a_2 & \dotsc & a_{n-1} \\
\pi\vp(a_{n-1}) & \vp(a_0) & \vp(a_1) & \dotsc & \vp(a_{n-2}) \\
\pi\vp^2(a_{n-2}) & \pi\vp^2(a_{n-1}) & \vp^2(a_0) & \dotsc & \vp^2(a_{n-3}) \\
\vdots & \vdots & \ddots & \ddots & \vdots \\
\pi\vp^{n-1}(a_1) & \pi\vp^{n-1}(a_2) & \dotsc & \pi\vp^{n-1}(a_{n-1}) & \vp^{n-1}(a_0)
\end{pmatrix}
\end{equation}
where $a_0,a_1,\dotsc,a_{n-1}\in\Knr$, and, in addition, $\det(A)\in K^\times$.
\end{lem}

The proof of the lemma is given in \S\ref{ss:proof-l:form-of-matrices}. It implies that we can write $\Xt$ as the infinite disjoint union of subsets $\Xt^{(m)}$, where $m\in\bZ$ and $\Xt^{(m)}\subset\Xt$ consists of matrices of the form \eqref{e:form-of-matrices} such that $\det(A)\in\pi^m\cdot\cO_K^\times$. We recall that the group $\bG(K)=GL_n(\Knr)^F\cong D_{1/n}^\times$ (cf.~Remark \ref{r:identify-rational-points}) acts on $\Xt$ via right multiplication. The element \eqref{e:uniformizer-division} belongs to $\bG(K)$ and corresponds to a uniformizer in $D_{1/n}$. The action of $\varpi$ takes each $\Xt^{(m)}$ isomorphically onto $\Xt^{(m+1)}$. Thus it suffices to show that $\Xt^{(0)}$ can be naturally identified with the set of rational points of an $\overline{\bF}_q$-scheme.

\begin{lem}\label{l:aux}
If a matrix of the form \eqref{e:form-of-matrices} satisfies $\det(A)\in\cO_K^\times$, then $a_j\in\OKnr$ for $0\leq j\leq n-1$ and $a_0\not\in\pi\OKnr$.
\end{lem}

This lemma is proved in \S\ref{ss:proof-l:aux}.

\subsection{The scheme structure}\label{ss:lusztig-quotient-scheme-structure}
Let $h\geq 1$ be an integer and consider a matrix of the form \eqref{e:form-of-matrices} with $a_0\in(\OKnr/\pi^h\OKnr)^\times$ and $a_j\in\OKnr/\pi^{h-1}\OKnr$ for $1\leq j\leq n-1$. The entries below the diagonal are viewed as elements of $\pi\OKnr/\pi^h\OKnr$. The determinant of any such matrix can be naturally defined as an element of $(\OKnr/\pi^h\OKnr)^\times$.

\mbr

By Lemma \ref{l:aux}, we have $\Xt^{(0)}=\varprojlim_h \Xt^{(0)}_h$, where $\Xt^{(0)}_h$ consists of matrices of the form \eqref{e:form-of-matrices} with $a_0\in(\OKnr/\pi^h\OKnr)^\times$ and $a_j\in\OKnr/\pi^{h-1}\OKnr$ for $1\leq j\leq n-1$, whose determinant belongs to $(\cO_K/\pi^h\cO_K)^\times\subset(\OKnr/\pi^h\OKnr)^\times$. Now $\Xt^{(0)}_h$ can be naturally viewed as the set of $\bfq$-points of a scheme of finite type over $\bF_q$. When $\operatorname{char}(K)=p>0$, the construction of this scheme is clear (see also Remark \ref{r:equal-characteristic-construction}). In the case where $\operatorname{char}(K)=0$, we sketch it in \S\ref{ss:greenberg-construction} below.

\mbr

Similarly, $\Xt^{(m)}=\varprojlim_h \Xt^{(m)}_h$ for every $m\in\bZ$, where each $\Xt^{(m)}_h$ can be regarded as an algebraic variety over $\overline{\bF}_q$. In particular, $\Xt^{(m)}$ can be viewed as an $\bfq$-scheme.

\begin{example}\label{ex:lusztig-level-1}
Taking $h=1$ in the construction above, we see that $\Xt^{(0)}_1$ is the finite discrete set of diagonal matrices of the form $\operatorname{diag}\bigl(a,a^q,a^{q^2},\dotsc,a^{q^{n-1}}\bigr)$ with $a\in\bfq^\times$, whose determinant belongs to $\bF_q\subset\bfq$. Thus $\Xt^{(0)}_1$ is naturally identified with $\bF_{q^n}^\times$, viewed as a discrete set. The actions of $\cO_D^\times\subset D_{1/n}^\times$ and $\cO_L^\times\subset L^\times$ on $\Xt$ induce actions on $\Xt^{(0)}_1$ that factor through the quotients $\cO_D^\times/(1+\Pi\cO_D)$ and $\cO_L^\times/(1+\pi\cO_L)$, respectively. Both of these quotients can be naturally identified with $\bF_{q^n}^\times$, and both actions become the multiplication action of $\bF_{q^n}^\times$ on itself.
\end{example}

\subsection{Definition of homology groups}\label{ss:lusztig-homology-definition}
Recall that the subgroup $\bT(K)\cong L^\times$ acts on $\Xt$ via left multiplication. The action of $\cO^\times_L\subset\bT(K)$ preserves each $\Xt^{(m)}$, and if we write $\cO^\times_L=\varprojlim Z_h$, where $Z_h=(\cO_L/\pi^h\cO_L)^\times$ for $h\geq 1$, then the action of $\cO_L^\times$ on $\Xt^{(m)}_h$ factors through $Z_h$. The next result is proved in \S\ref{ss:proof-l:cohomology-X-h}.

\begin{lem}\label{l:cohomology-X-h}
For each $h\geq 2$, the action of $W_h=\Ker(Z_h\to Z_{h-1})$ on $\Xt^{(m)}_h$ preserves every fiber of the natural map $\Xt^{(m)}_h\rar{}\Xt^{(m)}_{h-1}$, the induced morphism $W_h\setminus\Xt^{(m)}_h\rar{}\Xt^{(m)}_{h-1}$ is smooth, and each of its fibers is isomorphic to the affine space $\bA^{n-1}$ over ${\overline{\bF}_q}$.
\end{lem}

Following \cite{lusztig-corvallis}, we put $H_i(S,\ql):= H^{2d-i}_c(S,\ql(d))$ for any smooth ${\overline{\bF}_q}$-scheme $S$ of pure dimension $d$, where $H^{2d-i}_c$ denotes $\ell$-adic cohomology with compact supports and $(d)$ is the Tate twist. Lemma \ref{l:cohomology-X-h} yields an isomorphism \[H_i(\Xt^{(m)}_{h-1},\ql)\rar{\simeq}H_i(\Xt^{(m)}_h,\ql)^{W_h};\] in particular, we have a natural embedding $H_i(\Xt^{(m)}_{h-1},\ql)\into H_i(\Xt^{(m)}_h,\ql)$. We set
\[
H_i(\Xt^{(m)},\ql):=\varinjlim_h H_i(\Xt^{(m)}_h,\ql), \qquad H_i(\Xt,\ql):=\bigoplus_m H_i(\Xt^{(m)},\ql).
\]
For each $i\geq 0$, the vector space $H_i(\Xt,\ql)$ inherits commuting smooth actions of the groups $\bG(K)\cong D_{1/n}^\times$ and $\bT(K)\cong L^\times$. Given a smooth character $\te:L^\times\rar{}\qls$, we write $H_i(\Xt,\ql)[\te]$ for the subspace of $H_i(\Xt,\ql)$ on which $L^\times$ acts via $\te$.

\begin{defin}\label{d:lusztig-construction}
We will denote by $R^i_{\bT,\te}$ the resulting smooth representation of the group $\bG(K)=D^\times_{1/n}$ in the vector space $H_i(\Xt,\ql)[\te]$.
\end{defin}

\begin{rem}
The setup considered in this section is essentially equivalent to that of \cite[\S2]{lusztig-corvallis}, modulo the fact that Lusztig worked with the subgroup $SL_n(\Knr)\subset GL_n(\Knr)$ rather than the full general linear group (in particular, Lusztig's $\bG(K)$ is isomorphic to the subgroup of elements in $D_{1/n}^\times$ that have reduced norm $1$).
\end{rem}

\subsection{Appendix: the mixed characteristic case}\label{ss:greenberg-construction} In this subsection we assume that $K$ is a finite extension of $\bQ_p$. For each integer $h\geq 1$, we described in \S\ref{ss:lusztig-quotient-scheme-structure} a certain set of matrices $\Xt^{(0)}_h$. Our goal is to explain how $\Xt^{(0)}_h$ can be naturally identified with the set of $\bfq$-points of a scheme of finite type over $\bF_q$. The method we use is due to Greenberg \cite{greenberg-1,greenberg-2}; however, we follow the more modern interpretation of this method presented in \cite[\S2]{stasinski-formalism}.

\mbr

Let $K_0$ be the maximal unramified extension of $\bQ_p$ inside $K$, so that $K_0$ also has residue field $\bF_q$ and $K$ is totally ramified over $K_0$. Let $r\geq 1$ be such that the kernel of the natural homomorphism $\cO_{K_0}\rar{}\cO_K/(\pi^h)$ is generated by $p^r$. Then $\cO_K/(\pi^h)$ can be viewed as an algebra over $\cO_{K_0}/(p^r)$. For each $\bF_p$-algebra $R$ write $W_r(R)$ for the ring of ($p$-typical) truncated Witt vectors of length $r$ over $R$. We have a natural identification $\cO_{K_0}/(p^r)\cong W_r(\bF_q)$, so if $R$ is an $\bF_q$-algebra, then $W_r(R)$ can be viewed as an algebra over $\cO_{K_0}/(p^r)$.

\mbr

We obtain a functor $\bO_{K,h}$ from $\bF_q$-algebras to commutative rings defined by
\[
\bO_{K,h}(R) = \cO_K/(\pi^h) \underset{\cO_{K_0}/(p^r)}{\otimes} W_r(R).
\]
(Since $r$ is determined uniquely by $h$, we omit it from the notation.) Thanks to \cite[\S2]{stasinski-formalism}, the functor $\bO_{K,h}$ is representable by an affine scheme of finite type over $\bF_q$. This fact formally implies that the functor $\bO_{K,h}^\times$ from $\bF_q$-algebras to commutative groups, which sends $R$ to the group of units $\bO_{K,h}(R)^\times$ of the ring $\bO_{K,h}(R)$, is also representable by an affine scheme of finite type over $\bF_q$. By construction, we have a natural identification $\bO_{K,h}(\bfq)=\OKnr/(\pi^h)$.

\begin{rem}
In fact, working with the quotient $\cO_{K_0}/(p^r)$ is unnecessary. Let $W(R)$ denote the full ring of $p$-typical Witt vectors over $R$. We can identify $\cO_{K_0}$ with $W(\bF_q)$, and if $R$ is an $\bF_q$-algebra, then $W(R)$ can be viewed as an $\cO_{K_0}$-algebra. We have a canonical isomorphism $\cO_K/(\pi^h)\tens_{\cO_{K_0}} W(R)\rar{\simeq}\bO_{K,h}(R)$ for each $h\geq 1$, which provides a more ``uniform'' (with respect to $h$) description of the functor $\bO_{K,h}$.
\end{rem}

\mbr

We will now explain how $\Xt^{(0)}_h$ can be viewed as the set of $\bfq$-points of a closed $\bF_q$-subscheme of the product $\bO_{K,h}^\times\times\underbrace{\bO_{K,h-1}\times\dotsc\times\bO_{K,h-1}}_{n-1}$. If $R$ is any $\bF_q$-algebra, the endomorphism $a\mapsto a^q$ of $R$ induces an $\cO_{K_0}$-algebra endomorphism of $W(R)$ and hence of $\bO_{K,h}(R)$. We obtain an endomorphism of $\bO_{K,h}$ as a functor from $\bF_q$-algebras to commutative rings, which we denote simply by $\vp:\bO_{K,h}\rar{}\bO_{K,h}$. (The induced map $\bO_{K,h}(\bfq)\rar{}\bO_{K,h}(\bfq)$ is the same as the map induced by the Frobenius automorphism of $\OKnr$ used in \S\ref{ss:Lusztig-setting}, so the notation should not cause any confusion.) On the other hand, multiplication by $\pi$ induces a morphism of functors $\bO_{K,h-1}\rar{}\bO_{K,h}$ whose image will be denoted by $\pi\bO_{K,h-1}\subset\bO_{K,h}$.

\subsubsection{The main construction}\label{sss:greenberg-main}
Let $R$ be any $\bF_q$-algebra. Given $a_0\in\bO_{K,h}(R)^\times$ and $a_1,\dotsc,a_{n-1}\in\bO_{K,h-1}(R)$, form the matrix
\[
A_h(a_0,a_1,\dotsc,a_{n-1})=\begin{pmatrix}
a_0 & a_1 & a_2 & \dotsc & a_{n-1} \\
\pi\vp(a_{n-1}) & \vp(a_0) & \vp(a_1) & \dotsc & \vp(a_{n-2}) \\
\pi\vp^2(a_{n-2}) & \pi\vp^2(a_{n-1}) & \vp^2(a_0) & \dotsc & \vp^2(a_{n-3}) \\
\vdots & \vdots & \ddots & \ddots & \vdots \\
\pi\vp^{n-1}(a_1) & \pi\vp^{n-1}(a_2) & \dotsc & \pi\vp^{n-1}(a_{n-1}) & \vp^{n-1}(a_0)
\end{pmatrix}
\]
where the entries below the diagonal are viewed as elements of the additive subgroup $\pi\bO_{K,h-1}(R)\subset\bO_{K,h}(R)$. The determinant of such a matrix can be viewed as an element of $\bO_{K,h}(R)^\times$, and the map
\[
(a_0,a_1,\dotsc,a_{n-1}) \mapsto \det A_h(a_0,a_1,\dotsc,a_{n-1})
\]
is a morphism of functors $\bO_{K,h}^\times\times\bO_{K,h-1}^{n-1}\rar{}\bO_{K,h}^\times$. With this notation, $\Xt^{(0)}_h$ is the set of $\bfq$-points of the closed $\bF_q$-subscheme of $\bO_{K,h}^\times\times\bO_{K,h-1}^{n-1}$ defined as the fiber of
\[
\bO_{K,h}^\times\times\bO_{K,h-1}^{n-1}\rar{}\bO_{K,h}^\times, \qquad (a_0,a_1,\dotsc,a_{n-1}) \mapsto \frac{\vp\bigl(\det A_h(a_0,a_1,\dotsc,a_{n-1})\bigr)}{\det A_h(a_0,a_1,\dotsc,a_{n-1})},
\]
over the identity element of $\bO_{K,h}^\times$.

\begin{rem}\label{r:equal-characteristic-construction}
Sometimes (for instance, in the proof of Lemma \ref{l:cohomology-X-h} and in \S\ref{ss:division-level-2}) it is convenient to be able to treat the equal characteristics and mixed characteristics cases simultaneously without having to use two different sets of notations. To that end, we adopt the following conventions. Suppose that $K=\bF_q((\pi))$ is the field of formal Laurent series in one variable over $\bF_q$. We then put $K_0=K$, and if $R$ is any $\bF_q$-algebra, we set $W(R):=R[[\pi]]$, which is viewed as an algebra over $\cO_{K_0}=\bF_q[[\pi]]$ in the obvious way. We again obtain a functor $\bO_{K,h}:R\mapsto\cO_K/(\pi^h)\tens_{\cO_{K_0}}W(R)$. Since in this setting we have $\bO_{K,h}(R)=R[[\pi]]/(\pi^h)$, it is clear that $\bO_{K,h}$ is representable by the affine space $\bA^h$ over $\bF_q$. All the other constructions described earlier in this subsections go through without any changes.
\end{rem}

\section{Examples for division algebras}\label{s:examples-division}

We keep the notation of Section \ref{s:construction-division}. In particular, $K$ is a local field, $n$ is an integer, $L\subset\Knr$ is the unramified extension of $K$ of degree $n$ and $D_{1/n}$ is the central division algebra with invariant $1/n$ over $K$. We use the presentation $D_{1/n}=L\langle\Pi\rangle/(\Pi^n-\pi)$ given in Remark \ref{r:reduced-norm-1}, where $\pi\in K$ is a uniformizer, and write $\cO_D=\cO_L\langle\Pi\rangle/(\Pi^n-\pi)$ for the ring of integers of $D_{1/n}$. Given a smooth character $\te:L^\times\rar{}\qls$, we define the \emph{level} of $\te$ to be the smallest integer $r$ such that $\te$ is trivial on $1+\pi^r\cO_L\subset\cO_L^\times$. As before, we write $\bF_q$ and $\bF_{q^n}$ for the residue fields of $K$ and $L$, respectively.

\subsection{Representations of level $\leq 2$}\label{ss:division-level-2} In \cite{corwin-division-Advances-1974} Corwin described a procedure that associates a smooth irreducible representation of $D_{1/n}^\times$ to any smooth character $\te:L^\times\rar{}\qls$. In this subsection we explain a related construction in the case where the level of $\te$ is at most $2$ (cf.~Remark \ref{r:eta-theta-irreducible}) and compare it with the cohomological construction suggested by Lusztig, which was described in detail in Section \ref{s:construction-division}.

\mbr

Assume that $\te:L^\times\rar{}\qls$ is a character that is trivial on $1+\pi^2\cO_L$. The restriction of $\te$ to $1+\pi\cO_L$ can then be viewed as a character of $(1+\pi\cO_L)/(1+\pi^2\cO_L)$, which we identify with the additive group of $\bF_{q^n}$. Let $\psi:\bF_{q^n}\rar{}\qls$ be the (possibly trivial) character induced by $\te$. In Theorem \ref{t:cohomology-first-variety} we explicitly constructed an irreducible representation $\rho_\psi$ of the finite group $U^{n,q}(\bF_{q^n})$. We identify $U^{n,q}(\bF_{q^n})$ with the quotient $(1+\Pi\cO_D)/(1+\Pi^{n+1}\cO_D)$ via the isomorphism \eqref{e:identify-U} and let $\widetilde{\rho}_\psi$ denote the inflation of $\rho_\psi$ to an irreducible representation of $1+\Pi\cO_D$.

\begin{rem}
Write $U^j_D=1+\Pi^j\cO_D$ for all $j\geq 1$. Then the center of $U^1_D/U^{n+1}_D$ equals $U^n_D/U^{n+1}_D$ and the central character of $\rho_\psi$ corresponds to $\psi$ under the natural identification $\bF_{q^n}\rar{\simeq}U^n_D/U^{n+1}_D$ given by $a\mapsto 1+a\pi=1+a\Pi^n \mod U^{n+1}_D$.
\end{rem}

Next let $\zeta\in\cO_L^\times$ denote a primitive root of unity of order $q^n-1$; the image of $\zeta$ modulo $\pi$ generates the cyclic group $(\cO_L/(\pi))^\times=\bF_{q^n}^\times$. We have a natural embedding $\cO_L\into\cO_D$, so $\ze$ can also be viewed as an element of $\cO_D^\times$. The group $\cO_D^\times$ is the semidirect product $\langle\zeta\rangle\ltimes(1+\Pi\cO_D)$, where $\langle\zeta\rangle\subset\cO_D^\times$ is the cyclic subgroup generated by $\zeta$. Hence if $\widetilde{\rho}_\psi$ can be extended to a representation of $\cO_D^\times$, then that extension is unique up to a $1$-dimensional character of the cyclic group $\langle\ze\rangle$.

\begin{lem}\label{l:extend-to-O-D-times}
There exists a (necessarily unique and irreducible) representation $\eta^\circ_\te$ of $\cO_D^\times$ such that $\eta^\circ_\te\bigl\lvert_{1+\Pi\cO_D}=\widetilde{\rho}_\psi$ and $\Tr(\eta^\circ_\te(\ze))=(-1)^{n+n/m}\te(\ze)$, where $m$ is such that $q^m$ is the conductor of $\psi$ (see Definition \ref{d:conductor}).
\end{lem}

The lemma is proved in \S\ref{ss:proof-l:extend-to-O-D-times}.

\mbr

Finally, observe that $D_{1/n}^\times$ is the semidirect product of the normal subgroup $\cO_D^\times$ and the cyclic subgroup $\Pi^\bZ$ generated by the uniformizer $\Pi$. It contains the product $\pi^{\bZ}\cdot\cO_D^\times$ as an open subgroup of index $n$. We extend the representation $\eta^\circ_\te$ described in Lemma \ref{l:extend-to-O-D-times} to a representation $\eta'_\te$ of $\pi^{\bZ}\cdot\cO_D^\times\cong\pi^{\bZ}\times\cO_D^\times$ by letting $\pi$ act via the scalar $\te(\pi)$ and form the induced representation
\begin{equation}\label{e:eta-theta}
\eta_\te := \Ind_{\pi^\bZ\cdot\cO_D^\times}^{D_{1/n}^\times} \eta'_\te.
\end{equation}

\begin{thm}\label{t:lusztig-level-2}
With the notation above, if $q^m$ is the conductor of $\psi$, then $R^{n-n/m}_{\bT,\te}\cong\eta_\te$ and $R^i_{\bT,\te}=0$ for all $i\neq n-n/m$, where $R^i_{\bT,\te}$ is the smooth representation of $D^\times_{1/n}$ constructed in Definition \ref{d:lusztig-construction}.
\end{thm}

The proof of the theorem is given in \S\ref{ss:proof-t:lusztig-level-2}. The heart of the proof lies in finding a relation between the ind-scheme $\Xt$ constructed in Section \ref{s:construction-division} and the variety $X$ appearing in Definition \ref{d:X-main}. In fact, we will see that the scheme $\Xt^{(0)}_2$ (see \S\ref{ss:lusztig-quotient-scheme-structure}) is naturally isomorphic to a disjoint union of $q^n-1$ copies of $X$.

\begin{rem}\label{r:eta-theta-irreducible}
If $\psi$ has conductor $q^n$ (in other words, $\psi$ does not factor through the trace map $\Tr_{\bF_{q^n}/\bF_{q^r}}:\bF_{q^n}\rar{}\bF_{q^r}$ for any $1\leq r<n$), then the representation $\eta_\te$ is irreducible and coincides with the representation of $D_{1/n}^\times$ associated to $\te$ via the construction of \cite{corwin-division-Advances-1974}. To check irreducibility, observe that the subgroup $U^n_D/U^{n+1}_D$ is central in $\pi^\bZ\cdot\cO_D^\times/U^{n+1}_D$, and if we identify $U^n_D/U^{n+1}_D$ with the additive group $\bF_{q^n}$ in the natural way, then the conjugation action of $\Pi\in D_{1/n}^\times/U^{n+1}_D$ becomes identified with the action of the Frobenius $F_q(x)=x^q$ on $\bF_{q^n}$. If $\psi$ has conductor $q^n$, then $\psi$ is not invariant under the action of $F_q^j$ for any $1\leq j\leq n-1$. This means that the normalizer of the irreducible representation $\eta'_\te$ in $D_{1/n}^\times$ is equal to $\pi^\bZ\cdot\cO_D^\times$, so by Mackey's criterion, $\eta_\te$ is irreducible.
\end{rem}

\subsection{Higher levels: general strategy}\label{ss:higher-levels-strategy} We now suggest a strategy for analyzing the representations $R^i_{\bT,\te}$ for characters $\te:L^\times\rar{}\qls$ that have level $>2$.

\mbr

\textbf{From now on we assume that} $\operatorname{char}(K)=p$; in other words, $K=\bF_q((\pi))$. It is likely that the methods of Section \ref{s:principles} can also be used to handle the mixed characteristics case; however, the calculations will be rather more complicated. The idea is to relate the schemes $\Xt^{(0)}_h$ defined in Section \ref{s:construction-division} with certain higher-dimensional analogues of the unipotent group $U^{n,q}$ and the scheme $X$ constructed in Definitions \ref{d:ring-scheme-level-2} and \ref{d:X-main}. We keep the notation used in the first part of the section and identify $L$ with the field of Laurent series $\bF_{q^n}((\pi))$. Elements of $\cO_D$ will often be written as formal power series $\sum_{j\geq 0}a_j\Pi^j$, where $a_j\in\bF_{q^n}$; the commutation relation defining the product of two such power series is $\Pi\cdot a=a^q\cdot\Pi$ for $a\in\bF_{q^n}$.

\mbr

From now on we fix an integer $h\geq 2$.

\begin{defin}\label{d:ring-scheme-level-h}
If $A$ is an $\bF_p$-algebra, consider the twisted polynomial ring $A\langle\tau\rangle$ with the commutation relation $\tau\cdot a=a^q\cdot\tau$ ($a\in A$) and put $\cR_{h,n,q}(A)=A\langle\tau\rangle/(\tau^{n(h-1)+1})$. The functor $A\mapsto\cR_{h,n,q}(A)$ is obviously representable by the affine space $\bA^{n(h-1)+1}$ over $\bF_p$.
Thus it defines a ring scheme $\cR_{h,n,q}$ over $\bF_p$. Its elements will be written as $a_0+a_1\tau+\dotsc+a_{n(h-1)}\tau^{n(h-1)}$. The multiplicative group of $\cR_{h,n,q}$ will be denoted by $\cR^\times_{h,n,q}$ and the subgroup of $\cR^\times_{h,n,q}$ of elements of the form $1+a_1\tau+\dotsc+a_{n(h-1)}\tau^{n(h-1)}$ will be denoted by $U_h^{n,q}\subset\cR_{h,n,q}^\times$ (it is a unipotent group over $\bF_p$).
\end{defin}

When $h=2$, this construction specializes to that of Definition \ref{d:ring-scheme-level-2}. We have an obvious analogue of Remark \ref{r:Frobenius-R-n-q}. However, the analogue of Definition \ref{d:X-main} is more subtle. In what follows we will work with a subscheme of $U^{n,q}_h$ of the form $X_h=L_{q^n}^{-1}(Y_h)$, but the construction of $Y_h$ is more complicated (in particular, it is no longer defined by linear equations; see Remark \ref{r:Y-3-is-complicated} for an example).

\begin{rem}
With the notation of Definition \ref{d:ring-scheme-level-h}, we have a ring isomorphism
\begin{equation}\label{e:identify-rings}
\cR_{h,n,q}(\fqn)\rar{\simeq}\frac{\cO_D}{\Pi^{n(h-1)+1}\cO_D}, \qquad \sum_{j=0}^{n(h-1)}a_j\tau^j \longmapsto \sum_{j=0}^{n(h-1)}a_j\Pi^j,
\end{equation}
that restricts to group isomorphisms
\begin{equation}\label{e:identify-U-h}
\cR_{h,n,q}^\times(\fqn)\rar{\simeq}\frac{\cO_D^\times}{1+\Pi^{n(h-1)+1}\cO_D}, \qquad U^{n,q}_h(\fqn)\rar{\simeq}\frac{1+\Pi\cO_D}{1+\Pi^{n(h-1)+1}\cO_D}.
\end{equation}
\end{rem}

\subsubsection{The scheme $\Xt^{(0)}_h$ revisited} In \S\ref{ss:greenberg-construction} we explained how to define $\Xt^{(0)}_h$ as an (affine) scheme of finite type over $\bF_q$. In the equal characteristics case $K=\bF_q((\pi))$ this definition can be rewritten more explicitly, and we proceed to do so.

\begin{defin}\label{d:big-matrices}
If $A$ is an $\bF_p$-algebra we let $M_h(A)$ denote the ring of all $n$-by-$n$ matrices $B=(b_{ij})_{i,j=1}^n$ with $b_{ii}\in A[\pi]/(\pi^h)$ for all $i$, $b_{ij}\in A[\pi]/(\pi^{h-1})$ for all $i<j$ and $b_{ij}\in\pi A[\pi]/(\pi^h)$ for all $i>j$. The ring operations are given by the usual addition and multiplication of matrices. The determinant can be viewed as a multiplicative map $\det:M_h(A)\rar{}A[\pi]/(\pi^h)$.
\end{defin}

\begin{defin}\label{d:embedding-level-h}
If $A$ is an $\bF_p$-algebra and $a_0,a_1,\dotsc,a_{n(h-1)}\in A$, we will write $\iota_h(a_0,a_1,\dotsc,a_{n(h-1)})\in M_h(A)$ for the matrix (cf.~Definition \ref{d:big-matrices})
{\footnotesize
\[
\begin{pmatrix}
\sum\limits_{j=0}^{h-1} a_{nj}\pi^j & \sum\limits_{j=0}^{h-1} a_{nj+1}\pi^j &  \sum\limits_{j=0}^{h-1} a_{nj+2}\pi^j & \dotsc & \sum\limits_{j=0}^{h-1} a_{nj+n-1}\pi^j \\
\sum\limits_{j=0}^{h-1} a_{nj+n-1}^q\pi^{j+1} & \sum\limits_{j=0}^{h-1} a_{nj}^q\pi^j & \sum\limits_{j=0}^{h-1} a^q_{nj+1}\pi^j & \dotsc & \sum\limits_{j=0}^{h-1} a_{nj+n-2}^q\pi^j \\
\sum\limits_{j=0}^{h-1} a^{q^2}_{nj+n-2}\pi^{j+1} & \sum\limits_{j=0}^{h-1} a_{nj+n-1}^{q^2}\pi^{j+1} & \sum\limits_{j=0}^{h-1} a^{q^2}_{nj}\pi^j & \dotsc & \sum\limits_{j=0}^{h-1} a_{nj+n-3}^{q^2}\pi^j \\
\vdots & \vdots & \ddots & \ddots & \vdots \\
\sum\limits_{j=0}^{h-1} a_{nj+1}^{q^{n-1}}\pi^{j+1} & \sum\limits_{j=0}^{h-1} a_{nj+2}^{q^{n-1}}\pi^{j+1} & \dotsc & \sum\limits_{j=0}^{h-1} a_{nj+n-1}^{q^{n-1}}\pi^{j+1} & \sum\limits_{j=0}^{h-1} a_{nj}^{q^{n-1}}\pi^j
\end{pmatrix}
\]
}
\end{defin}

\begin{rem}\label{r:embedding-level-h}
If $\varpi$ is the matrix given by \eqref{e:uniformizer-division}, we can also write
\begin{equation}\label{e:iota-h}
\iota_h(a_0,a_1,\dotsc,a_{n(h-1)}) = \sum_{j=0}^{n(h-1)} \at_j \varpi^j,
\end{equation}
where $\at$ denotes the diagonal matrix $\operatorname{diag}(a,a^q,\dotsc,a^{q^{n-1}})$ for any $a\in A$.
\end{rem}

If $A$ is an $\bF_q$-algebra, then the set of $A$-valued points $\Xt^{(0)}_h(A)$ of the scheme $\Xt^{(0)}_h$ can be identified with the set of all matrices of the form $\iota_h(a_0,a_1,\dotsc,a_{n(h-1)})$ with $a_0\in A^\times$ and $a_1,\dotsc,a_{n(h-1)}\in A$, whose determinant lies in $\bF_q[\pi]/(\pi^h)\subset A[\pi]/(\pi^h)$.

\begin{example}\label{ex:n-2-h-3}
Let $n=2$ and $h=3$. We find that when $n=2$, the set $\Xt^{(0)}_3(A)$ can be identified with the set of all matrices of the form
\[
\matr{a_0+a_2\pi+a_4\pi^2}{a_1+a_3\pi}{a_1^q\pi+a_3^q\pi^2}{a_0^q+a_2^q\pi+a_4^q\pi^2}
\]
such that $a_0\in A^\times$, $a_1,a_2,a_3,a_4\in A$ and
\[
(a_0+a_2\pi+a_4\pi^2)(a_0^q+a_2^q\pi+a_4^q\pi^2)-(a_1+a_3\pi)(a_1^q\pi+a_3^q\pi^2)\in\bF_q[\pi]/(\pi^2).
\]
\end{example}

\begin{rem}\label{r:n-2-h-3}
The last requirement can be rewritten more explicitly as follows:
\[
a_0^{q+1}\in\bF_q^\times, \quad a_0 a_2^q+a_2 a_0^q-a_1^{q+1}\in\bF_q, \quad a_0 a_4^q+a_4 a_0^q+a_2^{q+1}-a_1a_3^q-a_3a_1^q\in\bF_q.
\]
\end{rem}

\subsubsection{Conventions} From now on all geometric objects we consider will be defined over $\bF_{q^n}$. In particular, we view $\Xt^{(0)}_h$ as a scheme over $\bF_{q^n}$ by base change from $\bF_q$ (the actions of $\cO_L^\times$ and $\cO_D^\times$ on $\Xt^{(0)}_h$ are only defined over $\bF_{q^n}$) and we view the unipotent group $U^{n,q}_h$ constructed in Definition \ref{d:ring-scheme-level-h} as an algebraic group over $\bF_{q^n}$.

\subsubsection{Group actions} Recall that we have a left action of $\cO_L^\times$ and a right action of $\cO_D^\times$ on $\Xt^{(0)}_h$, and these actions commute with each other. The action of $\cO_L^\times$ factors through $\cO_L^\times/U^h_L$ and the action of $\cO_D^\times$ factors through $\cO_D^\times/U^{n(h-1)+1}_D$, where $U^j_L=1+\pi^j\cO_L$ and $U^j_D=1+\Pi^j\cO_D$ for all $j\geq 1$. We will now explicitly describe the actions of $\cO_L^\times/U^h_L$ and $\cO_D^\times/U^{n(h-1)+1}_D$ on $\Xt^{(0)}_h$.

\mbr

The quotient $\cO_L^\times/U^h_L$ can be identified with the group of truncated polynomials $b_0+b_1\pi+\dotsc+b_{h-1}\pi^{h-1}$, where $b_0\in\bF_{q^n}^\times$ and $b_1,\dotsc,b_{h-1}\in\fqn$, which are multiplied using distributivity and the identity $\pi^h=0$. With this notation, for any $\fqn$-algebra $A$, the left action of $\cO_L^\times/U^h_L$ on $\Xt^{(0)}_h(A)$ comes from the left multiplication action of $\bigl(\bF_{q^n}[\pi]/(\pi^h)\bigr)^\times$ on the set of all matrices $M_h(A)$ described in Definition \ref{d:big-matrices}.

\mbr

On the other hand, with the notation of Remark \ref{r:embedding-level-h}, we have $\varpi\cdot\at=\widetilde{a^q}\cdot\varpi$ for all $a\in\bF_{q^n}$, and formula \eqref{e:iota-h} shows that the quotient ring $\cO_D/(\Pi^{n(h-1)+1})$ can be identified with the subring of $M_h(\bF_{q^n})$ (cf.~Def.~\ref{d:big-matrices}) consisting of all matrices of the form $\iota_h(a_0,a_1,\dotsc,a_{n(h-1)})$, where $a_0,a_1,\dotsc,a_{n(h-1)}\in\bF_{q^n}$ (cf.~Def.~\ref{d:embedding-level-h}). Under this identification, for any $\fqn$-algebra $A$, the right action of $\cO_D^\times/U^{n(h-1)+1}_D$ on $\Xt^{(0)}_h(A)$ comes from the action of $\cO_D^\times/U^{n(h-1)+1}_D$ on $M_h(A)$ by right multiplication.

\begin{rems}\label{rems:crucial}
\begin{enumerate}[(1)]
\item Let $A$ be an $\bF_p$-algebra and define $\iota'_h:\cR_{h,n,q}(A)\into M_h(A)$ by
    \[
    \iota'_h\left( \sum_{j=0}^{n(h-1)} a_j\tau^j \right) = \iota_h\bigl(a_0,a_1,\dotsc,a_{n(h-1)}\bigr),
    \]
    where $\cR_{h,n,q}(A)$ is the ring introduced in Definition \ref{d:ring-scheme-level-h}. Then $\iota'_h$ is additive but not multiplicative in general. However, it enjoys the following weakening of the multiplicativity property: if $A$ is an $\fqn$-algebra, then $\iota'_h(xy)=\iota'_h(x)\iota'_h(y)$ for all $x\in\cR_{h,n,q}(A)$ and all $y\in\cR_{h,n,q}(\fqn)$.
 \sbr
\item If $y\in\cR_{h,n,q}(\fqn)$, then $\det\iota'_h(y)$ is automatically contained in the subring $\bF_q[\pi]/(\pi^h)\subset\fqn[\pi]/(\pi^h)$.
\end{enumerate}
 \sbr
Thus the subset $\Xt^{(0)}_h(A)\subset M_h(A)$ is indeed stable under right multiplication by matrices of the form $\iota_h(a_0,a_1,\dotsc,a_{n(h-1)})$ with $a_0\in\fqn^\times$ and $a_1,\dotsc,a_{n(h-1)}\in\bF_{q^n}$.
\end{rems}

\subsubsection{The main construction}\label{sss:higher-levels-main-construction} Recall that we now view $\cR_{h,n,q}$ (Def.~\ref{d:ring-scheme-level-h}) and $M_h$ (Def.~\ref{d:big-matrices}) as ring schemes over $\fqn$. Using the morphism $\iota'_h:\cR_{h,n,q}\into M_h$ constructed in Remark \ref{rems:crucial}(1), we can identify $\Xt^{(0)}_h$ with the closed subscheme $\iota'^{-1}_h(\Xt^{(0)}_h)$ of the multiplicative group $\cR_{h,n,q}^\times$ of $\cR_{h,n,q}$.

\begin{defin}
We define $X_h\subset U^{n,q}_h$ to be the intersection of $\iota'^{-1}_h(\Xt^{(0)}_h)$ with the subgroup $U^{n,q}_h\subset\cR_{h,n,q}^\times$. Equivalently, we see that for any $\fqn$-algebra $A$, we have
\[
X_h(A) = \bigl\{ g\in U^{n,q}_h(A) \st \det\iota'_h(g)\in\bF_q[\pi]/(\pi^h) \bigr\}.
\]
\end{defin}

\begin{rem}\label{r:X-h-is-a-preimage}
Remarks \ref{rems:crucial} imply that the subscheme $X_h\subset U^{n,q}_h$ is closed under right multiplication by $U^{n,q}_h(\fqn)$. Therefore $X_h=L_{q^n}^{-1}(Y_h)$ for some closed subscheme $Y_h\subset U^{n,q}_h$, where $L_{q^n}:U^{n,q}_h\rar{}U^{n,q}_h$ is the Lang map $g\mapsto F_{q^n}(g)g^{-1}$.
\end{rem}

\begin{defin}\label{d:left-action}
The image of the embedding $\iota'_h:X_h\into\Xt^{(0)}_h$ is invariant under the left action of the subgroup $U^1_L/U^h_L\subset\cO_L^\times/U^h_L$. Hence there is a left action of $U^1_L/U^h_L$ on $X_h$ that makes $\iota'_h$ equivariant under $U^1_L/U^h_L$. We will denote this action by
\[
\ga:x\longmapsto\ga*x, \qquad \ga\in U^1_L/U^h_L, \ x\in X_h.
\]
This action commutes with the right multiplication action of $U^{n,q}_h(\fqn)$ on $X_h$.
\end{defin}

\subsubsection{Some conjectures} In \S\ref{sss:higher-levels-main-construction} we constructed a closed subscheme $X_h\subset U^{n,q}_h$ defined over $\fqn$ which is stable under right multiplication by elements of $U^{n,q}_h(\fqn)$ and a left action ``$*$'' of $U^1_L/U^h_L$ on $X_h$ that commutes with the right action of $U^{n,q}_h(\fqn)$. Given an integer $i\geq 0$ and a character $\chi:U^1_L/U^h_L\rar{}\qls$, we write
\[
H^i_c(X_h,\ql)[\chi]\subset H^i_c(X_h,\ql)
\]
for the subspace on which $U^1_L/U^h_L$ acts via $\chi$. This subspace is a representation of the finite group $U^{n,q}_h(\fqn)$; it also carries an action of the Frobenius $\Fr_{q^n}$ that commutes with the action of $U^{n,q}_h(\fqn)$.

\begin{conjecture}
For each $i\geq 0$, we have $H^i_c(X_h,\ql)=0$ unless $i$ or $n$ is even, and $\Fr_{q^n}$ acts on $H^i_c(X_h,\ql)$ by the scalar $(-1)^i q^{ni/2}$.
\end{conjecture}

\begin{rem}
This conjecture amounts to the assertion that $X_h$ is a ``maximal variety'' over $\fqn$; in other words, the size of $X_h(\fqn)$ is as large as possible subject to the constraints imposed by the dimensions of $H^i_c(X_h,\ql)$. When $h=2$, the conjecture follows from Theorem \ref{t:cohomology-first-variety} because $X_2=X$ by Corollary \ref{c:embed-X}.
\end{rem}

\begin{conjecture}\label{conj:reduction-to-X_h}
Given a character $\chi:U^1_L/U^h_L\rar{}\qls$, there exists $r\geq 0$ such that $H^i_c(X_h,\ql)[\chi]=0$ for all $i\neq r$. Moreover, $H^r_c(X_h,\ql)[\chi]$ is an irreducible representation of $U^{n,q}_h(\fqn)$.
\end{conjecture}

If $h=2$, we can identify $\bF_{q^n}$ both with $U^1_L/U^2_L$ and with the center of $U^{n,q}_2(\fqn)$, and under these identifications, the left action of $U^1_L/U^2_L$ is the same as the right multiplication action of the center of $U^{n,q}_2(\fqn)$ on $X_2$. So for $h=2$, the last conjecture also follows from Theorem \ref{t:cohomology-first-variety}.

\subsubsection{Calculation of $R^i_{\bT,\te}$} We now explain how the calculation of the representations $R^i_{\bT,\te}$ arising from Lusztig's construction (Definition \ref{d:lusztig-construction}) can be reduced to the calculation of $H^i_c(X_h,\ql)$ as a representation of $(U^1_L/U^{h}_L)\times U^{n,q}_h(\fqn)$ (modulo Conjecture \ref{conj:reduction-to-X_h}). We will identify $U^{n,q}_h(\fqn)$ with $U^1_D/U^{n(h-1)+1}_D$ via \eqref{e:identify-U-h}. We also choose a generator $\ze$ of the cyclic group $\fqn^\times$; we can view $\ze$ both as an element of $\cO_L^\times$ and as an element of $\cO_D^\times$. Note that $\cO_D^\times/U^{n(h-1)+1}_D$ is the semidirect product of the normal subgroup $U^1_D/U^{n(h-1)+1}_D$ and the cyclic subgroup $\langle\ze\rangle$ generated by $\ze$.

\begin{prop}[see \S\ref{ss:proof-p:main-reduction}]\label{p:main-reduction}
Let $\te:L^\times\to\qls$ be a character of level $h$ and write $\chi:U^1_L/U^h_L\to\qls$ for the character induced by $\te$. Assume that Conjecture \ref{conj:reduction-to-X_h} holds for the character $\chi$, let $r\geq 0$ be as in the formulation of the conjecture and let $\rho_\chi$ denote the representation of $U^1_D/U^{n(h-1)+1}_D\cong U^{n,q}_h(\fqn)$ in $H^r_c(X_h,\ql)[\chi]$. Then
 \sbr
\begin{enumerate}[$($a$)$]
\item $\rho_\chi$ extends to a representation $\eta^\circ_\te$ of $\cO_D^\times/U^{n(h-1)+1}_D$ with $\Tr(\eta^\circ_\te(\ze))=(-1)^r\te(\ze)$.
 \sbr
\item Let $\widetilde{\eta}^\circ_\te$ denote the inflation of $\eta^\circ_\te$ to a representation of $\cO_D^\times$ and extend $\widetilde{\eta}^\circ_\te$ to a representation $\eta'_\te$ of $\pi^\bZ\cdot\cO_D^\times$ by setting $\eta'_\te(\pi)=\te(\pi)$. Finally, set
    \[
    \eta_\te := \Ind_{\pi^\bZ\cdot\cO_D^\times}^{D_{1/n}^\times} \eta'_\te.
    \]
    Then $R^{2(h-1)(n-1)-r}_{\bT,\te}\cong\eta_\te$ and $R^i_{\bT,\te}=0$ for all $i\neq 2(h-1)(n-1)-r$.
 \sbr
\item If the restriction of $\chi$ to $U^{h-1}_L/U^h_L\cong\fqn$ has conductor $q^n$ (Def.~\ref{d:conductor}), then $\eta_\te$ is irreducible.
\end{enumerate}
\end{prop}

\subsection{Higher levels: an example} We continue working in the setup of \S\ref{ss:higher-levels-strategy}.

\begin{thm}[see \S\ref{ss:proof-t:main-example}]\label{t:main-example} Let $n=2$ and $h=3$. Let $\te:L^\times\to\qls$ be a character of level $3$ and write $\chi:U^1_L/U^3_L\to\qls$ for the character induced by $\te$. Assume that the restriction of $\chi$ to $U^2_L/U^3_L\cong\fqq$ has conductor $q^2$ (Def.~\ref{d:conductor}). Then
 \sbr
\begin{enumerate}[$($a$)$]
\item Conjecture \ref{conj:reduction-to-X_h} holds for the character $\chi$ with $r=2$.
 \sbr
\item We have $R^i_{\bT,\te}=0$ for all $i\neq 2$ and $R^2_{\bT,\te}\cong\Ind_{\pi^\bZ\cdot\langle\ze\rangle\cdot U^2_D}^{D_{1/n}^\times}\te'$, where $\te':\pi^\bZ\cdot\langle\ze\rangle\cdot U^2_D\to\qls$ is the character given by $\pi^s\ze^k\cdot\bigl(1+\sum_{j\geq 2}a_j\Pi^j\bigr)\mapsto\te\bigl(\pi^s\ze^k\cdot(1+a_2\pi+a_4\pi^4)\bigr)$.
\end{enumerate}
\end{thm}

\section{Proofs}\label{s:proofs}

\subsection{Proof of Lemma \ref{l:DL-first}}\label{ss:proof-l:DL-first}
Let us view $H^i_c(X,\ql)$ as a representation of the subgroup $H(\bF_q)\subset G(\bF_q)$ by restricting the right multiplication action of $G(\bF_q)$ on $X$ to an action of $H(\bF_q)$. Using Frobenius reciprocity, the assertion of the lemma can be rewritten as follows: for each $i\geq 0$, there is a linear isomorphism
\[
\Hom_{H(\bF_q)}\bigl(\chi,H^i_c(X,\ql)\bigr) \cong H^i_c\bigl(\al^{-1}(Y),\cE_\chi\bigl\lvert_{\al^{-1}(Y)}\bigr)
\]
compatible with the action of $\Fr_q$ on both sides. (The action of $\Fr_q$ on the left hand side is induced by its action on $H^i_c(X,\ql)$.) Let $\nu:G\rar{}Q:=G/(H(\bF_q))$ denote the quotient map, which is a torsor with respect to the left multiplication action of $H(\bF_q)$. If we let $\ql$ denote the constant rank $1$ $\ql$-local system on $G$, then
\begin{equation}\label{e:pushforward-decomposition}
\nu_!(\ql) \cong \bigoplus_{\rho\in\widehat{H(\bF_q)}} \rho\tens\cE_\rho
\end{equation}
as local systems with an action of $H(\bF_q)$, where $\widehat{H(\bF_q)}$ is a set of representatives of isomorphism classes of irreducible representations of $H(\bF_q)$ over $\ql$, and if $\rho$ is any such representation, then $\cE_\rho$ denotes the local system on $Q$ obtained from the torsor $\nu$ via $\rho$. The action of $H(\bF_q)$ on $\nu_!(\ql)$ comes from its action on $G$ and the action of $H(\bF_q)$ on each $\rho\tens\cE_\rho$ comes from its action on $\rho$.

\mbr

Formula \eqref{e:pushforward-decomposition} implies that for each $\rho\in\widehat{H(\bF_q)}$ and each $i\geq 0$ we have a $\Fr_q$-equivariant isomorphism $\Hom_{H(\bF_q)}\bigl(\rho,H^i_c(\al^{-1}(Y),\nu_!\ql)\bigr)\cong H^i_c\bigl(\al^{-1}(Y),\cE_\rho\bigl\lvert_{\al^{-1}(Y)}\bigr)$. Since $X=L_q^{-1}(Y)=\nu^{-1}(\al^{-1}(Y))$, the proper base change theorem allows us to canonically identify $H^i_c(\al^{-1}(Y),\nu_!\ql)$ with $H^i_c(X,\ql)$. Taking $\rho=\chi$ completes the proof of the lemma. \qed

\subsection{Proof of Lemma \ref{l:DL-second}}\label{ss:proof-l:DL-second}
By construction, the morphism $(G/H)\times H\rar{\gat}G$ given by $(x,h)\mapsto s(x)h$ is an isomorphism that is compatible with the right multiplication of $H(\bF_q)$ on both sides. Hence there is a unique isomorphism $\ga:(G/H)\times H\rar{\simeq}Q$ making the following square commute:
\[
\xymatrix{
  (G/H)\times H \ar[rr]^{\ \ \ \ \ \gat} \ar[d]_{\id\times L_q} & & G \ar[d] \\
  (G/H)\times H \ar[rr]^{\ \ \ \ \ \ga} & & Q
   }
\]
where $L_q:H\rar{}H$ is the Lang map $h\mapsto F_q(h)h^{-1}$ for the group $H$, which identifies $H$ with the quotient $H/(H(\bF_q))$, and the right vertical arrow is the quotient map $G\rar{}Q=G/(H(\bF_q))$. Thus $\ga^*(\cE_\chi)\cong\pr_2^*(\cL_\chi)$, where $\pr_2:(G/H)\times H\rar{}H$ is the second projection and $\cL_\chi$ is the local system on $H$ induced by the $H(\bF_q)$-torsor $L_q:H\rar{}H$ via $\chi$. Since $\chi=\psi\circ f$, we have $\cL_\chi\cong f^*\cL_\psi$. Finally, if we let $\be=\al\circ\ga:(G/H)\times H\rar{}G$, then the identity
\[
L_q(s(x)h)=F_q(s(x)h)\cdot(s(x)h)^{-1}=s(F_q(x))F_q(h)h^{-1}s(x)^{-1}=s(F_q(x))L_q(h)s(x)^{-1}
\]
implies that $\be(x,h)=s(F_q(x))\cdot h\cdot s(x)^{-1}$, completing the proof. \qed

\subsection{Multiplicative local systems}\label{ss:multiplicative-local-systems} This subsection and the next one contain some preliminaries for the proof of Proposition \ref{p:inductive-idea}.

\mbr

Let $G$ be an algebraic group over $\bF_q$. A rank $1$ $\ql$-local system $\cL$ on $G$ is said to be \emph{multiplicative} if $\mu^*(\cL)\cong\pr_1^*(\cL)\tens\pr_2^*(\cL)$, where $\mu:G\times G\rar{}G$ is the group operation and $\pr_1,\pr_2:G\times G\rar{}G$ are the two projections. In this case the corresponding trace-of-Frobenius function takes values in $\qls$ and is a character $t_{\cL}:G(\bF_q)\rar{}\qls$.

\begin{lem}\label{l:characters-and-multiplicative-local-systems}
If $G$ is a connected commutative algebraic group over $\bF_q$, the map $[\cL]\mapsto t_{\cL}$ is a bijection between the set of isomorphism classes of multiplicative local systems on $G$ and $\Hom\bigl(G(\bF_q),\qls\bigr)$.
\end{lem}

This follows from Lemmas 1.10 and 1.11 in \cite{intro}.

\begin{lem}\label{l:cohomology-vanishing}
Let $G$ be a connected algebraic group over $\bF_q$. If $\cL$ is any nontrivial multiplicative local system on $G$, then $H^i_c(G,\cL)=0$ for all $i\geq 0$.
\end{lem}

For the proof, see \cite[Lem.~9.4]{characters}.

\subsection{Multiplicative local systems on $\bG_a$}\label{ss:mult-loc-sys-G_a} In this subsection we view $\bG_a$ as an algebraic group over the prime field $\bF_p$. Recall that we have a nontrivial character $\psi_0:\bF_p\rar{}\qls$, which is fixed once an for all. The corresponding Artin-Schreier local system $\cL_{\psi_0}$ on $\bG_a$ is multiplicative because the Lang map $L_p:\bG_a\rar{}\bG_a$ is a group homomorphism, and the trace-of-Frobenius function $t_{\cL_{\psi_0}}:\bF_p\rar{}\qls$ is equal to $\psi_0$ by \cite[(1.5.1), \emph{Sommes Trig.}]{sga4.5}. We can also consider the base change of $\cL_{\psi_0}$ to the algebraic closure $\overline{\bF}_p$ of $\bF_p$ in $\bF$. For every $x\in\overline{\bF}_p$, let $\cL_x$ denote the pullback of $\cL_{\psi_0}$ by the morphism $y\mapsto xy$. Then $\cL_x$ is a  local system on $\bG_a\tens\overline{\bF}_p$, and if $x\in\bF_{p^k}$ for some integer $k\geq 1$, then $\cL_x$ is also defined over $\bF_{p^k}$.

\begin{lem}
For each integer $k\geq 1$, the map $x\mapsto\cL_x$ is an isomorphism between $\bF_{p^k}$ and the group\footnote{The group operation is given by the tensor product of local systems.} of (isomorphism classes of) multiplicative local systems on $\bG_a\tens\bF_{p^k}$. The character $\bF_{p^k}\to\qls$ corresponding to $\cL_x$ is given by $y\mapsto\psi_0(\Tr_{\bF_{p^k}/\bF_p}(xy))$.
\end{lem}

This result follows from the previous remarks together with Lemma \ref{l:characters-and-multiplicative-local-systems}.

\mbr

Our next goal is to compute the action of the endomorphism ring $\End_{\overline{\bF}_p}(\bG_a\tens\overline{\bF}_p)$ on the group of isomorphism classes of multiplicative local systems on $\bG_a\tens\overline{\bF}_p$. Let us make this problem more precise. The last lemma yields a natural identification of this group with the additive group of $\overline{\bF}_p$.
For each $f\in\End_{\overline{\bF}_p}(\bG_a\tens\overline{\bF}_p)$ and each $x\in\overline{\bF}_p$, the pullback $f^*(\cL_x)$ is another multiplicative local system on $\bG_a\tens\overline{\bF}_p$, whence $f^*(\cL_x)\cong\cL_{f^*(x)}$ for a unique element $x\in\overline{\bF}_p$, and $f^*:\overline{\bF}_p\rar{}\overline{\bF}_p$ is an additive homomorphism. We would like to calculate the map $f\mapsto f^*$ explicitly.

\mbr

To this end, we observe that every element of $\End_{\overline{\bF}_p}(\bG_a\tens\overline{\bF}_p)$ can be written uniquely as $a_0+a_1 F_p+\dotsb+a_d F_p^d$, where $a_j\in\overline{\bF}_p$ and $F_p:\bG_a\rar{}\bG_a$ is the Frobenius $x\mapsto x^p$. Furthermore, the map $f\mapsto f^*$ is an antihomomorphism, that is, it is additive and $(f\circ g)^*=g^*\circ f^*$. So it suffices to calculate $f^*$ when $f$ is a scalar in $\overline{\bF}_p$ and when $f=F_p$. The answer is provided by the following lemma.

\begin{lem}
If $f\in\End_{\overline{\bF}_p}(\bG_a\tens\overline{\bF}_p)$ is the endomorphism of multiplication by some element of $\overline{\bF}_p$, then $f^*=f$. Moreover, $F_p^*=F_p^{-1}$ is the map $x\mapsto x^{1/p}$.
\end{lem}

\begin{proof}
If $z\in\overline{\bF}_p$ is such that $f(x)=xz$ for all $x$, then $f^*(\cL_x)$ is the pullback of $\cL_{\psi_0}$ via the map $y\mapsto zxy$, which implies that $f^*(x)=zx$, proving the first assertion of the lemma. On the other hand, if $f(x)=x^p$, then $f^*(\cL_x)$ is the pullback of $\cL_{\psi_0}$ by the map $y\mapsto x\cdot y^p = (x^{1/p}\cdot y)^p$. Since $\cL_{\psi_0}$ is defined over $\bF_p$, it is invariant under pullback via $y\mapsto y^p$, which means that $f^*(\cL_x)$ is isomorphic to the pullback of $\cL_{\psi_0}$ via $y\mapsto x^{1/p}y$. So $f^*(\cL_x)\cong\cL_{x^{1/p}}$, which proves the lemma.
\end{proof}

\begin{cor}\label{c:dual-endomorphism}
We have $(a_0+a_1 F_p+\dotsb+a_d F_p^d)^*=a_0+a_1^{1/p} F_p^{-1}+\dotsb+a_d^{1/p^d} F_p^{-d}$.
\end{cor}

\subsection{Proof of Proposition \ref{p:inductive-idea}}\label{ss:proof-p:inductive-idea} Let us introduce some additional notation. Write $\pr:S=S_2\times\bA^1\rar{}S_2$ for the first projection and $\iota:S_3\into S_2$ for the inclusion, and let $\eta:S\rar{}\bG_a$ be given by $\eta(x,y)=f(x)^{q^j}y-f(x)^{q^n} y^{q^{n-j}}$. Then $P_1=\eta+P_2\circ\pr$ as morphisms $S\rar{}\bG_a$, where $+$ denotes addition in $\bG_a$. Since $\cL_\psi$ is a multiplicative local system on $\bG_a$, we have $P_1^*\cL_\psi\cong(\eta^*\cL_\psi)\tens\pr^*(P_2^*\cL_\psi)$. The projection formula implies that
\[
R\pr_!(P_1^*\cL_\psi)\cong P_2^*\cL_\psi \tens R\pr_!(\eta^*\cL_\psi) \qquad\text{in } D^b_c(S_2,\ql).
\]
To complete the proof it therefore suffices to check that
\begin{equation}\label{e:need-p:inductive-idea}
R\pr_!(\eta^*\cL_\psi)\cong \iota_!(\ql)[2](1) \qquad\text{in } D^b_c(S_2,\ql)
\end{equation}
where $\ql$ denotes the constant rank $1$ local system on $S_2$. The restriction of $\eta$ to $\pr^{-1}(S_3)\subset S$ is constant, so the restriction of $\eta^*\cL_\psi$ to $\pr^{-1}(S_3)$ is the constant local system of rank $1$. The restriction of $\pr$ to $\pr^{-1}(S_3)$ can be identified with the first projection $S_3\times\bA^1\rar{}S_3$, so the proper base change theorem implies that $\iota^* R\pr_!(\eta^*\cL_\psi)\cong \ql[2](1)$ in $D^b_c(S_3,\ql)$. Hence to prove \eqref{e:need-p:inductive-idea} it remains to show that each of the sheaves $R^i\pr_!(\eta^*\cL_\psi)$ vanishes outside the subscheme $S_3\subset S_2$.

\mbr

To that end, consider any point $x\in S_2(\bfq)$ that does not belong to $S_3$. The stalk of $R^i\pr_!(\eta^*\cL_\psi)$ at $x$ can be calculated via the proper base change theorem:
\[
R^i\pr_!(\eta^*\cL_\psi)_x \cong H^i_c(\bG_a,f_x^*\cL_\psi)
\]
where $f_x:\bG_a\rar{}\bG_a$ is given by $y\mapsto f(x)^{q^j}y-f(x)^{q^n} y^{q^{n-j}}$. (Here, by a slight abuse of notation, we view $\bG_a$ as an algebraic group over $\bfq$.) With the notation of \S\ref{ss:mult-loc-sys-G_a} we can write $\cL_\psi=\cL_z$ for some $z\in\bF_{q^n}$; the assumption on the conductor of $\psi$ made in the statement of the proposition means that $z\in\bF_{q^m}$ but $z\not\in\bF_{q^j}$, since $m$ does not divide $j$. By Corollary \ref{c:dual-endomorphism}, we have $f_x^*\cL_\psi\cong\cL_{f_x^*(z)}$, where
\[
f_x^*(z)=f(x)^{q^j}z-f(x)^{q^n/(q^{n-j})}z^{1/(q^{n-j})}=f(x)^{q^j}\cdot(z-z^{q^{j-n}})=f(x)^{q^j}\cdot(z-z^{q^{j}}).
\]
But $z\neq z^{q^j}$ and $f(x)\neq 0$, so $f_x^*\cL_\psi$ is a nontrivial multiplicative local system on $\bG_a$. Thus $H^i_c(\bG_a,f_x^*\cL_\psi)=0$ for all $i\geq 0$ by Lemma \ref{l:cohomology-vanishing}, as required. \qed

\subsection{Proof of Lemma \ref{l:fixed-point-1}}\label{ss:proof-l:fixed-point-1} For each $a\in A$, let $\rho_a:X\rar{}X$ denote the automorphism $x\mapsto x\cdot a$ given by the action of $a$ on $X$ and write $\sg_a=\sg\circ\rho_a$. Then $\sg_a$ is a finite order automorphism of $X$, so as in the proof of \cite[Prop.~3.3]{deligne-lusztig},
\[
\sum_{i} (-1)^i \Tr\bigl( \Fr_q\circ\sg_a^*; H^i_c(X,\ql) \bigr) = \Bigl\lvert \bigl\{ x\in X(\bfq) \st \sg(F_q(x))= x\cdot a^{-1} \bigr\} \Bigr\rvert
\]
because the right hand side equals the number of fixed points of $F_q\circ\sg_a$ acting on $X(\bfq)$. Now observe that for each $i\geq 0$, the operator
\[
\abs{A}^{-1}\sum_{a\in A}\chi(a^{-1})\rho_a^* : H^i_c(X,\ql)\rar{}H^i_c(X,\ql)
\]
is a projector onto the subspace $H^i_c(X,\ql)[\chi]$, commuting with the actions of both $\sg^*$ and $\Fr_q$.
It follows that
\[
\abs{A}^{-1}\sum_{a\in A} \chi(a)\cdot \Bigl\lvert \bigl\{ x\in X(\bfq) \st \sg(F_q(x))= x\cdot a \bigr\} \Bigr\rvert = \sum_i (-1)^i \Tr\bigl( \Fr_q\sg^*; H^i_c(X,\ql)[\chi] \bigr).
\]
Now formula \eqref{e:fixed-point-1} follows at once from the hypotheses of Lemma \ref{l:fixed-point-1}. \qed

\subsection{Proof of Lemma \ref{l:fixed-point-2}}\label{ss:proof-l:fixed-point-2} For each $\ga\in\Ga$, write $\sg_\ga:X\rar{}X$ for the automorphism $x\mapsto\ga\cdot x$ given by the action of $\Ga$. The argument of \S\ref{ss:proof-l:fixed-point-1} yields
\[
\abs{A}^{-1}\sum_{a\in A} \chi(a)\cdot \Bigl\lvert \bigl\{ x\in X(\bfq) \st \ga\cdot F_q(x)= x\cdot a \bigr\} \Bigr\rvert = \sum_i (-1)^i \Tr\bigl( \Fr_q\sg_\ga^*; H^i_c(X,\ql)[\chi] \bigr).
\]
It remains to observe that for each $i\geq 0$, the operator
\[
\frac{1}{\abs{\Ga}} \sum_{\ga\in\Ga} \te(\ga)^{-1}\sg_\ga^* : H^i_c(X,\ql)[\chi]\rar{}H^i_c(X,\ql)[\chi]
\]
is a projector onto the subspace $H^i_c(X,\ql)_{\te,\chi}\subset H^i_c(X,\ql)[\chi]$ that commutes with the action of $\Fr_q$, and to apply the hypotheses of Lemma \ref{l:fixed-point-2}. \qed

\subsection{Proof of Lemma \ref{l:quotient-calculation}}\label{ss:proof-l:quotient-calculation}
We reformulate the assertion of the lemma as follows: given $h\in\Ut$, there exist unique $g\in\Ut\cap F(\Ut^-)$ and $B\in\Ut\cap F^{-1}(\Ut)$ such that $F(B)g=hB$. To prove the last statement, observe first that the explicit formula \eqref{e:Frobenius-division-explicit} implies that $\Ut\cap F(\Ut^-)$ is the subgroup of $\Ut$ consisting of matrices $(c_{ij})$ such that $c_{ij}=0$ unless $i=1$ or $i=j$, while $\Ut\cap F^{-1}(\Ut)$ is the subgroup of $\Ut$ consisting of matrices $(b_{ij})$ such that $b_{i,n}=0$ for all $1\leq i\leq n-1$.

\mbr

Choose an arbitrary $h\in\Ut$ and write $h=(a_{ij})_{i,j=1}^n$, so that $a_{ij}=0$ for $i>j$ and $a_{ii}=1$ for all $i$. Further, write
\[
B=
\begin{pmatrix}
1 & b_{1,2} & b_{1,3} & \dotsc & b_{1,n-1} & 0 \\
 & 1 & b_{2,3} & \dotsc & b_{2,n-1} & 0 \\
 & & 1 & & & \vdots \\
 & & & \ddots & b_{n-2,n-1} & 0 \\
 & & & & 1 & 0 \\
 & & & & & 1
\end{pmatrix}, \qquad
g=
\begin{pmatrix}
1 & c_2 & c_3 & \dotsc & c_n \\
 & 1 & 0 & \dotsc & 0 \\
 & & 1 &  & \\
 & & & \ddots & \\
 & & & & 1
\end{pmatrix}
\]
and treat the identity $F(B)g=hB$ as a system of equations in the unknowns $b_{ij}$ and $c_k$. We must prove that this system has a unique solution in $\Knr$.

\mbr

In fact, we will show that by choosing a specific order of these equations, the system can be solved recursively. As the first step, consider the entries of both sides of the equation $F(B)g=hB$ in the following positions: $(n-1,n)$, $(n-2,n-1)$, $(n-3,n-2)$, $\dotsc$, $(2,3)$, $(1,2)$. Using \eqref{e:Frobenius-division-explicit}, we obtain the following equations:
\begin{eqnarray*}
\vp(b_{n-2,n-1}) &=& a_{n-1,n}, \\
\vp(b_{n-3,n-2}) &=& b_{n-2,n-1}+a_{n-2,n-1}, \\
\vp(b_{n-4,n-3}) &=& b_{n-3,n-2}+a_{n-3,n-2}, \\
&& \dotsc \\
\vp(b_{1,2}) &=& b_{2,3}+a_{2,3}, \\
c_2 &=& b_{1,2}+a_{1,2}.
\end{eqnarray*}
These equations can be solved uniquely for $b_{n-2,n-1}$, $b_{n-3,n-2}$, $\dotsc$, $b_{2,3}$, $b_{1,2}$ and $c_2$.

\mbr

We proceed by induction. Namely, suppose that $t\geq 2$ and that $b_{ij}$ and $c_k$ were already found for all $j\leq i+t-1$ and $k\leq t$. Consider the entries of both sides of the equation $F(B)g=hB$ in the following positions: $(n-t,n)$, $(n-t-1,n-1)$, $\dotsc$, $(2,t+2)$, $(1,t+1)$. Using \eqref{e:Frobenius-division-explicit}, we obtain the following equations:
\begin{eqnarray*}
\vp(b_{n-t-1,n-1}) &=& a_{n-t,n}, \\
\vp(b_{n-t-2,n-2}) &=& b_{n-t-1,n-1}+???, \\
\vp(b_{n-t-3,n-3}) &=& b_{n-t-2,n-2}+???, \\
&& \dotsc \\
\vp(b_{1,t+1}) &=& b_{2,t+2}+???, \\
c_{t+1} &=& b_{1,t+1}+???,
\end{eqnarray*}
where in each case the question marks ``???'' denote a polynomial expression that involves only $b_{ij}$ and $c_k$ for $j\leq i+t-1$ and $k\leq t$, as well as the entries of $h$. The equations above can be solved uniquely for $b_{n-t-1,n-1}$, $b_{n-t-2,n-2}$, $\dotsc$, $b_{1,t+1}$, $c_{t+1}$, which completes the induction step in the proof of Lemma \ref{l:quotient-calculation}. \qed

\subsection{Proof of Lemma \ref{l:form-of-matrices}}\label{ss:proof-l:form-of-matrices} Recall that $\Ut\cap F(\Ut^-)$ is the subgroup of $\Ut$ consisting of matrices $(c_{ij})$ such that $c_{ij}=0$ unless $i=1$ or $i=j$. So if $A\in GL_n(\Knr)$, then $A\in\Xt$ if and only if there exist $c_2,c_3,\dotsc,c_n\in\Knr$ such that
\begin{equation}\label{e:proof-l:form-of-matrices-1}
F(A) = \begin{pmatrix}
1 & c_2 & c_3 & \dotsc & c_n \\
 & 1 & 0 & \dotsc & 0 \\
 & & 1 &  & \\
 & & & \ddots & \\
 & & & & 1
\end{pmatrix} \cdot A.
\end{equation}
Let us compare the entries on the two sides of \eqref{e:proof-l:form-of-matrices-1} in the second, third, $\dotsc$, $n$-th rows. Using \eqref{e:Frobenius-division-explicit}, one easily finds that if \eqref{e:proof-l:form-of-matrices-1} holds for some $c_2,c_3,\dotsc,c_n\in\Knr$, then $A$ necessarily has the form \eqref{e:form-of-matrices} for some $a_0,a_1,\dotsc,a_{n-1}\in\Knr$. Reversing this argument, we see that if $A\in GL_n(\Knr)$ has the form \eqref{e:form-of-matrices}, then
\begin{equation}\label{e:proof-l:form-of-matrices-2}
F(A) = \begin{pmatrix}
c_1 & c_2 & c_3 & \dotsc & c_n \\
 & 1 & 0 & \dotsc & 0 \\
 & & 1 &  & \\
 & & & \ddots & \\
 & & & & 1
\end{pmatrix} \cdot A
\end{equation}
for uniquely determined $c_1,c_2,\dotsc,c_n\in\Knr$ such that $c_1\neq 0$. It remains to observe that if \eqref{e:proof-l:form-of-matrices-2} holds, then $c_1=1$ if and only if $\det(F(A))=\det(A)$, which is equivalent to the requirement that $\det(A)\in K^\times$. \qed

\subsection{Proof of Lemma \ref{l:aux}}\label{ss:proof-l:aux} We will use the standard formula for $\det(A)$ as an alternating sum of $n!$ terms labeled by the elements of the symmetric group $S_n$:
\begin{equation}\label{e:expansion-of-determinant}
\det(A)=\sum_{\sg\in S_n} (-1)^\sg \left( \prod_{i=1}^n A_{i,\sg(i)} \right),
\end{equation}
where $A_{i,j}$ is the $(i,j)$-th entry of $A$. Let us denote by $v_0,v_1,\dotsc,v_{n-1}\in\bZ$ the normalized valuations of the elements $a_0,a_1,\dotsc,a_{n-1}\in\Knr$. We must prove that if $\det(A)$ has normalized valuation $0$, then $v_0=0$ and $v_1,\dotsc,v_{n-1}\geq 0$.

\mbr

Let $\tau\in S_n$ be the $n$-cycle defined by $\tau(i)=i+1$ for $1\leq i\leq n-1$ and $\tau(n)=1$. If $\sg=\tau^j$ for some $0\leq j\leq n-1$, the corresponding term in \eqref{e:expansion-of-determinant} only involves the variable $a_j$ and has normalized valuation $n\cdot v_j+j$. If $\sg$ is not a power of $\tau$, the corresponding term in \eqref{e:expansion-of-determinant} involves two or more of the variables $a_0,a_1,\dotsc,a_{n-1}$.

\mbr

We will show that $\det(A)$ has normalized valuation
\[
\min\bigl\{ n\cdot v_0,\ n\cdot v_1+1,\ n\cdot v_2+2, \dotsc,\ n\cdot v_{n-1}+n-1 \bigr\},
\]
which will complete the proof of the lemma. Since the integers
\begin{equation}\label{e:proof-l:aux-1}
n\cdot v_0,\ n\cdot v_1+1,\ n\cdot v_2+2, \dotsc,\ n\cdot v_{n-1}+n-1
\end{equation}
are pairwise distinct, it suffices to check that if $\sg\in S_n$ is not a power of $\tau$, then the normalized valuation of the corresponding term in \eqref{e:expansion-of-determinant} is a nontrivial weighted average of the numbers \eqref{e:proof-l:aux-1}. To this end, we fix $\sg\in S_n$ and write
\[
\al(i):=\begin{cases}
\sg(i)-i & \text{if } i\leq \sg(i) \\
\sg(i)-i+n & \text{if } i>\sg(i)
\end{cases}
\quad\text{and}\quad
\be(i):=\begin{cases}
0 & \text{if } i\leq \sg(i) \\
1 & \text{if } i>\sg(i)
\end{cases}
\]
for $1\leq i\leq n$. The product $\prod_{i=1}^n A_{i,\sg(i)}$ has normalized valuation $\sum_{i=1}^n (v_{\al(i)}+\be(i))$. Since $\sum_{i=1}^n (\sg(i)-i)=0$, we have $\frac{1}{n}\sum_{i=1}^n \al(i)=\sum_{i=1}^n\be(i)$, which implies that
\[
\sum_{i=1}^n (v_{\al(i)}+\be(i)) = \frac{1}{n} \cdot \sum_{i=1}^n (n\cdot v_{\al(i)} + \al(i) ),
\]
where the right hand side is manifestly a weighted average of the numbers \eqref{e:proof-l:aux-1}. It remains to observe that $\al(i)$ is independent of $i$ if and only if $\sg$ is a power of $\tau$. \qed

\subsection{Proof of Lemma \ref{l:cohomology-X-h}}\label{ss:proof-l:cohomology-X-h} Without loss of generality, we may assume that $m=0$. We will give an argument that is valid both in the equal characteristic case and in the mixed characteristic case using the conventions of \S\ref{ss:greenberg-construction} and Remark \ref{r:equal-characteristic-construction}.

\subsubsection{}
Let us introduce some auxiliary notation. For every integer $r\geq 1$, reduction modulo $\pi$ induces a morphism of functors $\bO_{K,r}\rar{}\bO_{K,1}$ which will be denoted by $x\mapsto\bar{x}$ (in each case, it will be clear what $r$ is from the context). We remark also that $\bO_{K,1}$ is the functor that sends every $\bF_q$-algebra $R$ to the underlying commutative ring $R$. Next, consider the morphism $\bO_{K,r}\rar{}\bO_{K,r-1}$ of reduction modulo $\pi^{r-1}$. If we view it as a morphism of commutative group schemes under addition, its kernel is naturally identified with the additive group $\bG_a$ over $\bF_q$. Since $\bO_{K,r-1}$ is an affine scheme, it follows that the morphism $\bO_{K,r}\rar{}\bO_{K,r-1}$ admits a scheme-theoretic section. For each $r$, we choose such a section once and for all, and denote it by $x\mapsto\hat{x}$ (again, in each case it will be clear what $r$ is from the context). As a final remark, observe that if $r\geq 2$, then the kernel of the reduction morphism of multiplicative groups $\bO_{K,r}^\times\rar{}\bO_{K,r-1}^\times$ is once again naturally identified with $\bG_a$, and the map $x\mapsto\hat{x}$ can also be viewed as a morphism of schemes $\bO_{K,r-1}^\times\rar{}\bO_{K,r}^\times$.

\subsubsection{}\label{sss:identification}
We now proceed with the proof of Lemma \ref{l:cohomology-X-h}. Let $h\geq 2$ be given and consider the product $\bO_{K,h-1}^\times\times\bO_{K,h-2}^{n-1}\times\bG_a^n$. We have an isomorphism of schemes
\[
\bO_{K,h-1}^\times\times\bO_{K,h-2}^{n-1}\times\bG_a^n \rar{\simeq} \bO_{K,h}^\times\times\bO_{K,h-1}^{n-1}
\]
given by
\[
(a_0,a_1,\dotsc,a_{n-1},b_0,b_1,\dotsc,b_{n-1})\mapsto (\hat{a}_0+b_0\pi^h,\hat{a}_1+b_1\pi^h,\dotsc,\hat{a}_{n-1}+b_{n-1}\pi^h).
\]
Under this isomorphism, the subscheme $\Xt^{(0)}_h\subset\bO_{K,h}^\times\times\bO_{K,h-1}^{n-1}$ is identified with a subscheme of the product $\Xt^{(0)}_{h-1}\times\bG_a^n\subset\bO_{K,h-1}^\times\times\bO_{K,h-2}^{n-1}\times\bG_a^n$. Let us describe this subscheme explicitly. Given an $\bF_q$-algebra $R$ and a point
\[
x=(a_0,a_1,\dotsc,a_{n-1},b_0,b_1,\dotsc,b_{n-1})\in \Xt^{(0)}_{h-1}(R)\times R^n,
\]
the difference\footnote{We recall that the notation $A_h$ was introduced in \S\ref{ss:greenberg-construction}.}
\begin{eqnarray*}
\sD(x) &:=& \vp\bigl(\det A_h(\hat{a}_0+b_0\pi^h,\hat{a}_1+b_1\pi^h,\dotsc,\hat{a}_{n-1}+b_{n-1}\pi^h)\bigr) \\
&& -\det A_h(\hat{a}_0+b_0\pi^h,\hat{a}_1+b_1\pi^h,\dotsc,\hat{a}_{n-1}+b_{n-1}\pi^h)
\end{eqnarray*}
automatically belongs to $\pi^h R\subset\bO_{K,h}(R)^\times$. We can thus view $\sD$ as a morphism of schemes $\Xt^{(0)}_{h-1}\times\bG_a^n\rar{}\bG_a$, and $\Xt^{(0)}_h$ is identified with the fiber of $\sD$ over $0\in\bG_a$.

\subsubsection{} For the remainder of the proof, we extend scalars from $\bF_q$ to $\bfq$ and view each $\Xt^{(0)}_h$ as a scheme over $\bfq$. In particular, for each $h\geq 1$, the finite group $Z_h=(\cO_L/(\pi^h))^\times$ acts on $\Xt^{(0)}_h$ by left multiplication, where $L$ is the degree $n$ unramified extension of $K$ inside $\Knr$.

\subsubsection{}
Observe that the group $W_h=\Ker(Z_h\to Z_{h-1})$ appearing in the formulation of Lemma \ref{l:cohomology-X-h} can be naturally identified with the additive group of $\bF_{q^n}$ via the map $\la\mapsto 1+\la\pi^h$. Under this isomorphism and the identification of $\Xt^{(0)}_h$ with a subscheme of $\Xt^{(0)}_{h-1}\times\bG_a^n$ described in \S\ref{sss:identification}, the action of $W_h$ on $\Xt^{(0)}_h$ is given by
\begin{equation}\label{e:proof-l:cohomology-X-h-1}
\la:(a_0,\dotsc,a_{n-1},b_0,\dotsc,b_{n-1})\mapsto(a_0,\dotsc,a_{n-1},b_0+\la \bar{a}_0,b_1,\dotsc,b_{n-1}).
\end{equation}
This implies the first assertion of the lemma, namely, that the natural projection $\Xt^{(0)}_h\rar{}\Xt^{(0)}_{h-1}$ is $W_h$-equivariant. To prove the other two assertions, we will show that the quotient $W_h\setminus\Xt^{(0)}_h$ is in fact isomorphic to the product $\Xt^{(0)}_{h-1}\times\bA^{n-1}$.

\subsubsection{}\label{sss:quotient-identification} If we view $\bG_a$ as a scheme over $\bfq$ on which the finite additive group $\bF_{q^n}$ acts via addition, then the map $\bG_a\rar{}\bG_a$ given by $x\mapsto x^{q^n}-x$ identifies the quotient $\bF_{q^n}\setminus\bG_a$ with $\bG_a$. This implies that if we consider \eqref{e:proof-l:cohomology-X-h-1} as defining a left action of $\bF_{q^n}$ on $\Xt^{(0)}_{h-1}\times\bG_a^n$, then the map $\Xt^{(0)}_{h-1}\times\bG_a^n\rar{}\Xt^{(0)}_{h-1}\times\bG_a^n$ given by
\[
(a_0,\dotsc,a_{n-1},b_0,\dotsc,b_{n-1})\mapsto \Bigl( a_0,\dotsc,a_{n-1},\frac{b_0^{q^n}}{\bar {a}_0^{q^n}}-\frac{b_0}{\bar{a}_0},b_1,\dotsc,b_{n-1} \Bigr)
\]
identifies the quotient $\bF_{q^n}\setminus\bigl(\Xt^{(0)}_{h-1}\times\bG_a^n\bigr)$ with $\Xt^{(0)}_{h-1}\times\bG_a^n$.

\subsubsection{} With the notation of \S\ref{sss:identification}, let us analyze the determinant
\begin{equation}\label{e:proof-l:cohomology-X-h-2}
\det A_h(\hat{a}_0+b_0\pi^h,\hat{a}_1+b_1\pi^h,\dotsc,\hat{a}_{n-1}+b_{n-1}\pi^h)\in\bO_{K,h}^\times.
\end{equation}
Using the standard expansion of the determinant as an alternating sum of $n!$ terms, we can view as a certain polynomial expression in $\hat{a}_0,\dotsc,\hat{a}_{n-1},b_0,\dotsc,b_{n-1}$. The only summand that involves $b_0$ is the one coming from the coefficient of $\pi^h$ in the product of the diagonal entries, namely,
\[
(\hat{a}_0+b_0\pi^h)\cdot(\widehat{\vp(a_0)}+b_0^q\pi^h)\cdot\dotsc\cdot(\widehat{\vp^{n-1}(a_0)}+b_0^{q^{n-1}}\pi^h).
\]
Thus we see that the contribution of $b_0$ to the determinant \eqref{e:proof-l:cohomology-X-h-2} equals
\[
\left( \frac{b_0}{\bar{a}_0} + \frac{b_0^q}{\bar{a}_0^q} + \dotsc + \frac{b_0^{q^{n-1}}}{\bar{a}_0^{q^{n-1}}} \right) \cdot \bar{a}_0^{1+q+q^2+\dotsc+q^{n-1}} \cdot \pi^h.
\]
It follows that $\Xt^{(0)}_h$ is identified with the subscheme of $\Xt^{(0)}_{h-1}\times\bG_a^n$ defined by an equation of the form $\sF_1(a_0,b_0)^q-\sF_1(a_0,b_0)=\sF_2^q-\sF_2$, where
\[
\sF_1(a_0,b_0) = \left( \frac{b_0}{\bar{a}_0} + \frac{b_0^q}{\bar{a}_0^q} + \dotsc + \frac{b_0^{q^{n-1}}}{\bar{a}_0^{q^{n-1}}} \right) \cdot \bar{a}_0^{1+q+q^2+\dotsc+q^{n-1}}
\]
and $\sF_2:\Xt^{(0)}_{h-1}\times\bG_a^n\rar{}\bG_a$
is a certain morphism that is \emph{independent of $b_0$}.

\subsubsection{} It remains to observe that the determinant $\det A_{h-1}(a_0,a_1,\dotsc,a_{n-1})$ is equal to $\bar{a}_0^{1+q+\dotsc+q^{n-1}}$ modulo $\pi$, which implies that if $(a_0,a_1,\dotsc,a_{n-1})\in\Xt^{(0)}_{h-1}$, then $\bar{a}_0^{1+q+\dotsc+q^{n-1}}\in\bF_{q}$. Hence the equation $\sF_1(a_0,b_0)^q-\sF_1(a_0,b_0)=\sF_2^q-\sF_2$ defining $\Xt^{(0)}_h$ as a subscheme of $\Xt^{(0)}_{h-1}\times\bG_a^n$ can be rewritten as
\[
\frac{b_0^{q^n}}{\bar {a}_0^{q^n}}-\frac{b_0}{\bar{a}_0} = \bar{a}_0^{-1-q-\dotsc-q^{n-1}}\cdot(\sF_2^q-\sF_2).
\]
In view of the remarks in \S\ref{sss:quotient-identification}, the projection $\Xt^{(0)}_{h-1}\times\bG_a^n\rar{}\Xt^{(0)}_{h-1}\times\bA^{n-1}$ obtained by discarding $b_0$ yields an isomorphism $W_h\setminus\Xt^{(0)}_h\rar{\simeq}\Xt^{(0)}_{h-1}\times\bA^{n-1}$. \qed

\subsection{An auxiliary calculation}\label{ss:auxiliary-traces} This subsection contains some preliminaries for the proof of Lemma \ref{l:extend-to-O-D-times} and Theorem \ref{t:lusztig-level-2}. Recall that in Definition \ref{d:ring-scheme-level-2} we introduced a ring scheme $\cR_{n,q}$ over $\bF_p$ and a unipotent group $U^{n,q}$, which coincides with the unipotent radical of the multiplicative group $\cR_{n,q}$. In Definition \ref{d:X-main} we constructed a scheme $X$ over $\bF_{q^n}$ equipped with an action of $U^{n,q}(\bF_{q^n})$ by right multiplication.

\mbr

The subgroup of $\cR_{n,q}^\times$ consisting of all elements of the form $\al+0\cdot\tau+\dotsc+0\cdot\tau^n$ is naturally isomorphic to the multiplicative group $\bG_m$, and $\cR_{n,q}^\times$ is equal to the semidirect product $\bG_m\ltimes U^{n,q}$. If we let $\bG_m$ act on $U^{n,q}$ by conjugation, then the action of $\bG_m(\bF_{q^n})=\bF_{q^n}^\times$ preserves the subvariety $X\subset U^{n,q}$, and thus we obtain an action of $\bF_{q^n}^\times\ltimes U^{n,q}(\bF_{q^n})$ on $X$, where $\bF_{q^n}^\times$ acts by conjugation and $U^{n,q}(\bF_{q^n})$ acts by right multiplication, as before.

\mbr

Choose any character $\psi:\bF_{q^n}\rar{}\qls$ and view it as a character of the center of $U^{n,q}(\bF_{q^n})$. Write $H^i_c(X,\ql)[\psi]$ for the subspace of $H^i_c(X,\ql)$ on which the center of $U^{n,q}(\bF_{q^n})$ acts via $\psi$. Parts (a) and (b) of Theorem \ref{t:cohomology-first-variety} can be rephrased as follows.

\begin{cor}\label{c:main-reformulation}
If $q^m$ is the conductor of $\psi$, we have $H^i_c(X,\ql)[\psi]=0$ for all $i\neq n+n/m-2$, and $H^{n+n/m-2}_c(X,\ql)[\psi]$ is an irreducible representation of $U^{n,q}(\bF_{q^n})$.
\end{cor}

As the center $Z(\bF_{q^n})$ of $U^{n,q}(\bF_{q^n})$ commutes with all of $\cR_{n,q}^\times(\bF_{q^n})$, we see that the space $H^{n+n/m-2}_c(X,\ql)[\psi]$ is preserved by the larger group $\bF_{q^n}^\times\ltimes U^{n,q}(\bF_{q^n})$, whose action on $X$ was described earlier. The main result of this subsection is

\begin{prop}\label{p:trace-zeta-bar}
Let $\zeb$ be a generator of the cyclic group $\bF_{q^n}^\times$. Then
\[
\Tr\bigl( \zeb^*, H^{n+n/m-2}_c(X,\ql)[\psi] \bigr) = (-1)^{n+n/m},
\]
where $\zeb^*$ denotes the automorphism of $H^i_c(X,\ql)$ induced by the action of $\zeb$ on $X$.
\end{prop}

\begin{proof}
Since $H^i_c(X,\ql)[\psi]=0$ for all $i\neq n+n/m-2$, the assertion of the proposition can be reformulated as follows: $\sum_i (-1)^i \Tr(\zeb^*,H^i_c(X,\ql)[\psi])=1$. The operator $q^{-n}\cdot\sum_{z\in Z(\bF_{q^n})}\psi(z)^{-1}z^*$ is the projection of $H^i_c(X,\ql)$ onto the subspace $H^i_c(X,\ql)[\psi]$, so we are reduced to showing that
\begin{equation}\label{e:need-auxiliary}
\sum_{z\in Z(\bF_{q^n})} \psi(z)^{-1}\cdot \sum_i (-1)^i \Tr\bigl((\zeb,z)^*,H^i_c(X,\ql)\bigr) = q^n.
\end{equation}
To prove \eqref{e:need-auxiliary} we apply Theorem \ref{t:fixed-point-formula-DL} to the action of $(\zeb,z)\in\bF_{q^n}^\times\ltimes U^{n,q}(\bF_{q^n})$ on $X$. Its ``abstract Jordan decomposition'' is $(\zeb,z)=\zeb\circ z$. Since $\zeb^{q^r}\neq \zeb$ for any $1\leq r<n$, the subvariety $X^\zeb\subset X$ of fixed points of $\zeb$ is the finite set $Z(\bF_{q^n})\subset X$ on which $Z(\bF_{q^n})$ acts by translation. By Theorem \ref{t:fixed-point-formula-DL}, $\sum_i (-1)^i \Tr\bigl((\zeb,z)^*,H^i_c(X,\ql)\bigr)$ equals the number of fixed points of $z$ acting on $X^\zeb$, which is $q^n$ if $z=0$ and $0$ otherwise. The identity \eqref{e:need-auxiliary} follows at once.
\end{proof}

\subsection{Proof of Lemma \ref{l:extend-to-O-D-times}}\label{ss:proof-l:extend-to-O-D-times}
Recall that $\psi:\bF_{q^n}\rar{}\qls$ is the character induced by the restriction of $\te$ to $1+\pi\cO_L\subset\cO_L^\times$. In \S\ref{ss:auxiliary-traces} we constructed a representation $\xi_\psi$ of the semidirect product $\bF_{q^n}^\times\ltimes U^{n,q}(\bF_{q^n})$ on $H^{n+n/m-2}_c(X,\ql)[\psi]$, which has the following two properties. First, the restriction of $\xi_\psi$ to $U^{n,q}(\bF_{q^n})$ is the representation $\rho_\psi$ constructed in Theorem \ref{t:cohomology-first-variety}. Second, if $\zeb$ is a generator of $\bF_{q^n}^\times$, then Proposition \ref{p:trace-zeta-bar} yields $\Tr(\xi_\psi(\zeb))=(-1)^{n+n/m}$, where $q^m$ is the conductor of $\psi$.

\mbr

Recall also that $\ze\in\cO_L^\times$ denotes a primitive root of unity of order $q^n-1$. If we view $\ze$ as an element of $\cO_D^\times/U^{n+1}_D$, we obtain a semidirect product decomposition $\cO_D^\times/U^{n+1}_D=\langle\ze\rangle\ltimes (U^1_D/U^{n+1}_D)$, where $\langle\ze\rangle$ is the cyclic group generated by $\ze$. Let $\zeb\in\bF_{q^n}^\times$ denote the image of $\ze$ under the reduction modulo $\Pi$ homomorphism $\cO_D^\times\rar{}\cO_D^\times/U^1_D\cong\bF_{q^n}^\times$. It is clear that the isomorphism $U^{n,q}(\bF_{q^n})\rar{\simeq}U^1_D/U^{n+1}_D$ given by \eqref{e:identify-U} is compatible with the conjugation actions of $\zeb$ and $\ze$ on $U^{n,q}(\bF_{q^n})$ and $U^1_D/U^{n+1}_D$, respectively. Thus we obtain an isomorphism $\bF_{q^n}^\times\ltimes U^{n,q}(\bF_{q^n})\rar{\simeq}\cO_D^\times/U^{n+1}_D$, which we use to view $\xi_\psi$ as a representation of $\cO_D^\times/U^{n+1}_D$. Finally, let $\widetilde{\te}$ be the character of $\cO_D^\times/U^{n+1}_D$ defined by $\widetilde{\te}(\ze)=\te(\ze)$ and $\widetilde{\te}(g)=1$ for all $g\in U^1_D/U^{n+1}_D$. Form the tensor product $\xi_\psi\tens\widetilde{\te}$ and let $\eta_\te^\circ$ be its inflation to a representation of $\cO_D^\times$. By construction, $\eta_\te^\circ$ satisfies all the requirements of Lemma \ref{l:extend-to-O-D-times}.

\subsection{Proof of Theorem \ref{t:lusztig-level-2}}\label{ss:proof-t:lusztig-level-2} The argument consists of several steps. We begin by recalling that with the notation of \S\ref{ss:lusztig-homology-definition}, we have
\[
H_*(\Xt,\ql)=\bigoplus_{m\in\bZ}H_*(\Xt^{(m)},\ql)=\bigoplus_{m\in\bZ} \varinjlim_h H_*(\Xt^{(m)}_h,\ql).
\]
Since the character $\te:L^\times\rar{}\qls$ is trivial on $1+\pi^2\cO_L$, Lemma \ref{l:cohomology-X-h} and the definition of $H_*(\Xt^{(m)},\ql)$ imply that the subspace $H_*(\Xt,\ql)[\te]\subset H_*(\Xt,\ql)$ is contained in the image of the natural inclusion $\oplus_m H_*(\Xt^{(m)}_2,\ql)\into H_*(\Xt,\ql)$.

\subsubsection{First reduction}\label{sss:first-reduction} Let us write $\Xt_2$ for the disjoint union of the schemes $\Xt^{(m)}_2$, $m\in\bZ$, and $H_*(\Xt_2,\ql)=\oplus_m H_*(\Xt^{(m)}_2,\ql)$. The action of $D_{1/n}^\times\times L^\times$ on $\Xt$ (where $D_{1/n}^\times$ acts via right multiplication and $L^\times$ acts via left multiplication, as before) induces an action of the discrete group $(D_{1/n}^\times/U^{n+1}_D)\times(L^\times/U^2_L)$ on $\Xt_2$, where for each $j\geq 1$, we write $U^j_D=1+\Pi^j\cO_D$ and $U^j_L=1+\pi^j\cO_L$. In order to prove Theorem \ref{t:lusztig-level-2} we must calculate each $H_i(\Xt_2,\ql)$ as a representation of $(D_{1/n}^\times/U^{n+1}_D)\times(L^\times/U^2_L)$. This goal will be accomplished in Corollary \ref{c:representation-of-Gamma} below.

\subsubsection{Relation with \S\ref{ss:unipotent-first}}\label{sss:embedding-level-2} Recall that in Definitions \ref{d:ring-scheme-level-2} and \ref{d:X-main} we constructed a unipotent group $U^{n,q}$ over $\bF_p$ and a subscheme $X$ of $U^{n,q}$ defined over $\bF_{q^n}$ such that $X$ is stable under right multiplication by $U^{n,q}(\bF_{q^n})$. From now on we extend the base to $\bfq\supset\bF_{q^n}$ and view $X$ as a scheme over $\bfq$, since we will only be interested in the cohomology $H^*_c(X,\ql)$ as a representation of $U^{n,q}(\bF_{q^n})$ and the action of $\Fr_{q^n}$ will be irrelevant. We next construct a certain embedding $X\into\Xt^{(0)}_2$.

\mbr

To this end, recall that in \S\ref{ss:greenberg-construction} we explicitly described $\Xt^{(0)}_2$ as a subscheme of the product $\bO_{K,2}^\times\times\bO_{K,1}^{n-1}$. Now if $R$ is any $\bfq$-algebra, then $\bO_{K,1}(R)$ is equal to the underlying ring of $R$ and we have a natural embedding $R\into\bO_{K,2}^\times(R)$ given by $a\mapsto 1+\pi a$. In particular, there is a well defined embedding of $\bfq$-schemes
\[
\iota:U^{n,q}\rar{}\bO_{K,2}^\times\times\bO_{K,1}^{n-1}, \qquad 1+a_1\tau+\dotsc+a_n\tau^n \longmapsto (1+a_n\pi,a_1,\dotsc,a_{n-1}).
\]
Now form the $n$-by-$n$ matrix $A_2(1+a_n\pi,a_1,\dotsc,a_{n-1})$, where the notation $A_2$ was introduced in \S\ref{sss:greenberg-main}. It is clear that we can write
\[
\det A_2(1+a_n\pi,a_1,\dotsc,a_{n-1}) = 1+\pi\cdot N_2(a_1,a_2,\dotsc,a_n)
\]
for a certain morphism of schemes $N_2:U^{n,q}\rar{}\bG_a$.

\mbr

The following crucial result is proved in \cite[Prop.~6.2]{maximal-varieties-LLC}.

\begin{prop}
With the notation above, $\pr_n(L_{q^n}(g))=N_2(g)^q-N_2(g)$ for all $g\in U^{n,q}$, where $L_{q^n}:U^{n,q}\rar{}U^{n,q}$ is the Lang map given by $g\mapsto F_{q^n}(g)g^{-1}$ and $\pr_n:U^{n,q}\rar{}\bG_a$ is the projection given by $1+a_1\tau+\dotsc+a_n\tau^n\mapsto a_n$.
\end{prop}

Recalling that the subscheme $\Xt^{(0)}_2\subset \bO_{K,2}^\times\times\bO_{K,1}^{n-1}$ is defined by the equation
\[
\vp(\det A_2(a_0,a_1,\dotsc,a_{n-1}))=\det A_2(a_0,a_1,\dotsc,a_{n-1})
\]
and that $X=L_{q^n}^{-1}(\pr_n^{-1}(0))\subset U^{n,q}$, we immediately obtain

\begin{cor}\label{c:embed-X}
The morphism $\iota$ restricts to an embedding $\iota:X\into\Xt^{(0)}_2$, whose image is equal to the intersection of $\Xt^{(0)}_2$ with $(1+\pi\bO_{K,1})\times\bO_{K,1}^{n-1}\subset\bO_{K,2}^\times\times\bO_{K,1}^{n-1}$.
\end{cor}

\subsubsection{Comparison of group actions} Recall that the group $U^{n,q}(\bF_{q^n})$ is naturally identified with $U^1_D/U^{n+1}_D$ via the isomorphism \eqref{e:identify-U}. Thus $U^{n,q}(\bF_{q^n})$ can be thought of as acting on the right both on $X$ and on $\Xt^{(0)}_2$. On the other hand, recall that if $Z=\{1+a_n\tau^n\}$ denotes the center of $U^{n,q}$, then $Z(\bF_{q^n})$ is the center of $U^{n,q}(\bF_{q^n})$ and acts on $X$ by left multiplication. We can naturally identify $Z(\bF_{q^n})$ with the additive group $\bF_{q^n}$, which in turn is identified with $U^1_L/U^2_L$. Thus $Z(\bF_{q^n})$ can be thought of as acting on $\Xt^{(0)}_2$ on the left.

\begin{lem}\label{l:embedding-actions-compatible}
The embedding $\iota:X\into\Xt^{(0)}_2$ constructed in \S\ref{sss:embedding-level-2} (Cor.~\ref{c:embed-X}) is equivariant with respect to the left action of $Z(\bF_{q^n})$ and the right action of $U^{n,q}(\bF_{q^n})$.
\end{lem}

\begin{proof}
For simplicity, let us check this statement at the level of $\bfq$-points (this is enough because $X$ is reduced). We return to the original definition of $\Xt^{(0)}_2(\bfq)$ as the set of matrices of the form \eqref{e:form-of-matrices} with $a_0\in(\OKnr/\pi^2\OKnr)^\times$ and $a_j\in\OKnr/\pi\OKnr$ for $1\leq j\leq n-1$, whose determinant belongs to $(\cO_K/\pi^2\cO_K)^\times\subset(\OKnr/\pi^2\OKnr)^\times$. If $a\in\OKnr$, let us write $\at$ for the diagonal matrix $\operatorname{diag}(a,\vp(a),\dotsc,\vp^{n-1}(a))$. Then the formula for the embedding $\iota:X(\bfq)\into\Xt^{(0)}_2(\bfq)$ can be rewritten as
\begin{equation}\label{e:embedding-explicit}
\iota : 1+a_1\tau+a_2\tau^2+\dotsc+a_n\tau^n \longmapsto 1+\at_1\varpi+\at_2\varpi^2+\dotsc+\at_n\varpi^n,
\end{equation}
where the matrix $\varpi$ is given by \eqref{e:uniformizer-division} and $1$ denotes the $n$-by-$n$ identity matrix on the right hand side. We now observe that $\varpi\cdot\at=\widetilde{\vp(a)}\cdot\varpi$ for all $a\in\cO_L$ and that $\varpi^{n+1}=0$ when viewed as an element of $\Xt_2$. In view of the relations defining the multiplication in the group $U^{n,q}$, these observations imply at once that $\iota$ commutes with the right multiplication action of $U^{n,q}(\bF_{q^n})$. The fact that it commutes with the left multiplication action of $Z(\bF_{q^n})$ follows from a direct calculation: if $z\in\bF_{q^n}$, then $z$ corresponds to $1+z\tau^n$ in $Z(\bF_{q^n})$ and to $1+z\pi$ in $U^1_L/U^2_L$, and we have
\begin{eqnarray*}
\iota\bigl((1+z\tau^n)\cdot(1+a_1\tau+\dotsc+a_n\tau^n)\bigr) &=& \iota(1+a_1\tau+\dotsc+(a_n+z)\tau^n) \\
&=& 1+\at_1\varpi+\dotsc+(\at_n+\widetilde{z})\varpi^n \\
&=& (1+\pi\widetilde{z})\cdot(1+\at_1\varpi+\dotsc+\at_n\varpi^n),
\end{eqnarray*}
where the last step uses the observation that $\varpi^n$ is the scalar matrix $\pi$.
\end{proof}

\subsubsection{}\label{sss:realize-as-induced-rep} Recall that $\ze\in\cO_L^\times$ denotes a primitive root of unity of order $q^n-1$ and that $\pi\in\cO_K$ is a chosen uniformizer. We will view $\ze$ and $\pi$ both as elements of $D_{1/n}^\times$ and as elements of $L^\times$. Recall also that the product $(D_{1/n}^\times/U^{n+1}_D)\times(L^\times/U^2_L)$ acts on $\Xt_2$. We will now define a subgroup $\Ga$ of this product and an action of $\Ga$ on $X$.

\mbr

View $(\ze,\ze^{-1})$ and $(\pi,\pi^{-1})$ as elements of $(D_{1/n}^\times/U^{n+1}_D)\times(L^\times/U^2_L)$. Then $(\ze,\ze^{-1})$ normalizes $(U^1_D/U^{n+1}_D)\times(U^1_L/U^2_L)$ and $(\pi,\pi^{-1})$ is central. Define
\[
\Ga=\langle(\pi,\pi^{-1})\rangle \cdot \langle(\ze,\ze^{-1})\rangle \cdot \bigl((U^1_D/U^{n+1}_D)\times(U^1_L/U^2_L)\bigr)
\]
and let $\Ga$ act on $X$ in the following way:

\begin{itemize}
\item the action of $U^1_D/U^{n+1}_D$ comes from the right multiplication action of $U^{n,q}(\bF_{q^n})$ via the isomorphism \eqref{e:identify-U};
 \sbr
\item the action of $U^1_L/U^2_L$ comes from the left multiplication action of $Z(\bF_{q^n})$, the center of $U^{n,q}(\bF_{q^n})$, via the natural identifications $U^1_L/U^2_L\cong\bF_{q^n}\cong Z(\bF_{q^n})$;
 \sbr
\item the element $(\ze,\ze^{-1})$ acts as conjugation by $\zeb$, the image of $\ze$ in $\bF_{q^n}^\times=\cO_D^\times/U^1_D$ (cf.~\S\S\ref{ss:auxiliary-traces}, \ref{ss:proof-l:extend-to-O-D-times});
 \sbr
\item the element $(\pi,\pi^{-1})$ acts trivially.
\end{itemize}

\begin{lem}\label{l:action-of-Gamma}
\begin{enumerate}[$($a$)$]
\item Consider the action of $\Ga$ on $\Xt_2$ obtained by restricting the action of $(D_{1/n}^\times/U^{n+1}_D)\times(L^\times/U^2_L)$ on $\Xt_2$. The embedding $\iota:X\into\Xt^{(0)}_2$ is $\Ga$-equivariant.
 \sbr
\item If $g\in(D_{1/n}^\times/U^{n+1}_D)\times(L^\times/U^2_L)$ is such that $g\not\in\Ga$, then $g(\iota(X))\cap\iota(X)=\varnothing$. In particular, $\Ga$ is the stabilizer of $\iota(X)$ for the action of $(D_{1/n}^\times/U^{n+1}_D)\times(L^\times/U^2_L)$.
 \sbr
\item $\Xt_2$ is equal to the union of $\bigl((D_{1/n}^\times/U^{n+1}_D)\times(L^\times/U^2_L)\bigr)$-translates of $\iota(X)$.
\end{enumerate}
\end{lem}

\begin{proof}
(a) The fact that $\iota$ commutes with the action of $\bigl((U^1_D/U^{n+1}_D)\times(U^1_L/U^2_L)\bigr)\subset\Ga$ is a reformulation of Lemma \ref{l:embedding-actions-compatible}. The compatibility of $\iota$ with the actions of $(\pi,\pi^{-1})$ and $(\ze,\ze^{-1})$ follows from a direct calculation.

\mbr

(b) Recall from Example \ref{ex:lusztig-level-1} that the scheme $\Xt^{(0)}_1$ can be naturally identified with the finite discrete set $\bF_{q^n}^\times$ so that the induced actions of $\cO_D^\times/U^1_D\cong\bF_{q^n}^\times$ and $\cO_L^\times/U^1_L\cong\bF_{q^n}^\times$ on $\Xt^{(0)}_1$ become identified with the multiplication action of $\bF_{q^n}^\times$ on itself. Corollary \ref{c:embed-X} means that the image of $\iota:X\into\Xt^{(0)}_2$ is equal to the fiber of the natural quotient map $\Xt^{(0)}_2\rar{}\Xt^{(0)}_1$ over $1\in\bF_{q^n}^\times$.
Since $(\cO_D^\times/U^{n+1}_D)\times(\cO_L^\times/U^2_L)$ is equal to the product of $\langle(\ze,\ze^{-1})\rangle \cdot \bigl((U^1_D/U^{n+1}_D)\times(U^1_L/U^2_L)\bigr)$ and the cyclic subgroup $\langle(1,\ze)\rangle$, we obtain the assertion of (b) when $g\in(\cO_D^\times/U^{n+1}_D)\times(\cO_L^\times/U^2_L)$. The general case follows immediately by observing that the action of $\Pi^m\in D_{1/n}^\times$ takes $\Xt^{(0)}_2$ onto $\Xt^{(m)}_2$ and the action of $\pi^j\in L^\times$ takes $\Xt^{(0)}_2$ onto $\Xt^{(nj)}_2$ for all $m,j\in\bZ$.

\mbr

(c) The argument used to prove (b) shows that $\Xt^{(0)}_2$ equals the union of translates of $\iota(X)$ under the subgroup $\langle(1,\ze)\rangle$. Since the action of $(\Pi^m,1)$ takes $\Xt^{(0)}_2$ isomorphically onto $\Xt^{(m)}_2$ for each $m\in\bZ$, the proof is complete.
\end{proof}

\begin{cor}\label{c:representation-of-Gamma}
For each $i\geq 0$, let us view\footnote{Recall that $X$ has dimension $n-1$, and the Tate twist is irrelevant for us at this point.} $H_i(X,\ql)=H^{2n-2-i}_c(X,\ql)$ as a representation of $\Ga$ using the action of $\Ga$ on $X$ constructed above.
 \mbr
\begin{enumerate}[$($a$)$]
\item There is a natural isomorphism
\[
H_i(\Xt_2,\ql) \cong \Ind_{\Ga}^{(D_{1/n}^\times/U^{n+1}_D)\times(L^\times/U^2_L)} \bigl(H_i(X,\ql)\bigr)
\]
of representations of $(D_{1/n}^\times/U^{n+1}_D)\times(L^\times/U^2_L)$.
 \sbr
\item Let $\psi:\bF_{q^n}\rar{}\qls$ be a character with conductor $q^m$, view $\psi$ as a character of $\{1\}\times(U^1_L/U^2_L)\subset\Ga$ and let $H_i(X,\ql)[\psi]\subset H_i(X,\ql)$ denote the $\psi$-isotypic component of $H_i(X,\ql)$. Then $H_i(X,\ql)[\psi]=0$ unless $i=n-n/m$, and if we let $\xi_\psi$ denote the representation of $\Ga$ in $H_{n-n/m}(X,\ql)[\psi]$, then $\xi'_\psi$ has the following properties:
 \sbr
\begin{itemize}
\item if we identify $U^{n,q}(\bF_{q^n})$ with the subgroup $(U^1_D/U^{n+1}_D)\times\{1\}$ of $\Ga$ via the isomorphism \eqref{e:identify-U}, then the restriction of $\xi'_\psi$ to $U^{n,q}(\bF_{q^n})$ is the representation $\rho_\psi$ constructed in Theorem \ref{t:cohomology-first-variety};
 \sbr
\item $\Tr(\xi'_\psi(\ze,\ze^{-1}))=(-1)^{n+n/m}$;
 \sbr
\item $\xi'_\psi(\pi,\pi^{-1})=1$.
\end{itemize}
\end{enumerate}
\end{cor}

\begin{proof}
(a) follows at once from parts (b) and (c) of Lemma \ref{l:action-of-Gamma}. The first two assertions of part (b) are reformulations of Corollary \ref{c:main-reformulation} and Proposition \ref{p:trace-zeta-bar}, while the last assertion of (b) is clear from the definition.
\end{proof}

\subsubsection{}\label{sss:end-proof-level-2} We now complete the proof of Theorem \ref{t:lusztig-level-2}. As before, we choose a level $2$ character $\te:L^\times\rar{}\qls$ and view it as a character of $L^\times/U^2_L$. We also let $\psi:\bF_{q^n}\rar{}\qls$ be the character corresponding to the restriction of $\te$ to $U^1_L/U^2_L$ via the natural identification $\bF_{q^n}\cong U^1_L/U^2_L$ and write $q^m$ for the conductor of $\psi$.

\mbr

As mentioned in \S\ref{sss:first-reduction}, the product $(D_{1/n}^\times/U^{n+1}_D)\times(L^\times/U^2_L)$ acts on $\Xt_2$, and for each $i\geq 0$, the $\te$-isotypic component $H_i(\Xt_2,\ql)[\te]\subset H_i(\Xt_2,\ql)$ is a representation of $D_{1/n}^\times/U^{n+1}_D$. Corollary \ref{c:representation-of-Gamma} implies that $H_i(\Xt_2,\ql)[\te]=0$ when $i\neq n-n/m$. To finish the proof we must show that $H_{n-n/m}(\Xt_2,\ql)[\te]\cong\eta_\te$ as representations of $D_{1/n}^\times$, where $\eta_\te$ is constructed in \S\ref{ss:division-level-2}, cf.~\eqref{e:eta-theta}.

\mbr

Since $\{1\}\times(L^\times/U^2_L)$ is central in $(D_{1/n}^\times/U^{n+1}_D)\times(L^\times/U^2_L)$, we can in fact view $H_{n-n/m}(\Xt_2,\ql)[\te]$ as a representation of $(D_{1/n}^\times/U^{n+1}_D)\times(L^\times/U^2_L)$. The structure of that representation is readily determined from Corollary \ref{c:representation-of-Gamma}. Namely, let $\widetilde{\Ga}$ denote the subgroup of $(D_{1/n}^\times/U^{n+1}_D)\times(L^\times/U^2_L)$ generated by $\Ga$ and $\{1\}\times(L^\times/U^2_L)$. Then the representation $\xi'_\psi$ described in Corollary \ref{c:representation-of-Gamma}(b) extends to a representation $\widetilde{\xi}'_\te$ of $\widetilde{\Ga}$ whose restriction to $\{1\}\times(L^\times/U^2_L)$ is given by the scalar $\te$, and we have
\[
H_{n-n/m}(\Xt_2,\ql)[\te] \cong \Ind_{\widetilde{\Ga}}^{(D_{1/n}^\times/U^{n+1}_D)\times(L^\times/U^2_L)} (\widetilde{\xi}'_\te)
\]
as representations of $(D_{1/n}^\times/U^{n+1}_D)\times(L^\times/U^2_L)$.

\mbr

It remains to observe that $\widetilde{\Ga}$ is equal to the product $(\pi^\bZ\cdot\cO_D^\times)\times(L^\times/U^2_L)$ and that $\widetilde{\xi}'_\te$ is equal to the tensor product $\eta'_\te\tens\te$, where $\eta'_\te$ is the representation of $\pi^\bZ\cdot\cO_D^\times$ constructed in \S\ref{ss:division-level-2}, which is seen immediately by writing $(\pi,1)=(\pi,\pi^{-1})\cdot(1,\pi)$ and $(\ze,1)=(\ze,\ze^{-1})\cdot(1,\ze)$ as elements of $(D_{1/n}^\times/U^{n+1}_D)\times(L^\times/U^2_L)$. Therefore $H_{n-n/m}(\Xt_2,\ql)[\te]\cong\eta_\te\tens\te$ as representations of $(D_{1/n}^\times/U^{n+1}_D)\times(L^\times/U^2_L)$, so \emph{a fortiori}, $H_{n-n/m}(\Xt_2,\ql)[\te]\cong\eta_\te$ as a representation of $D_{1/n}^\times/U^{n+1}_D$. \qed

\subsection{Proof of Proposition \ref{p:main-reduction}}\label{ss:proof-p:main-reduction}

\subsubsection*{Step 1} To prove part (a) of Proposition \ref{p:main-reduction}, it suffices to show that $\rho_\chi$ can be extended to a representation $\rho'_\chi$ of $\cO_D^\times/U^{n(h-1)+1}_D$ satisfying $\Tr(\rho'_\chi(\ze))=(-1)^r$. (Since $\rho_\chi$ is irreducible by assumption, such an extension is automatically unique.)

\mbr

Let us identify $\cO_D^\times/U^{n(h-1)+1}_D$ with $\cR_{h,n,q}^\times(\fqn)$ via \eqref{e:identify-U-h}, and let $\zeb\in\cR_{h,n,q}^\times(\fqn)$ be the element corresponding to $\ze$ under this identification. Then $\cR_{h,n,q}^\times(\fqn)$ can be viewed as the semidirect product $\langle\zeb\rangle\ltimes U^{n,q}_h(\fqn)$, where $\langle\zeb\rangle$ is the cyclic subgroup of $\cR_{h,n,q}^\times(\fqn)$ generated by $\zeb$. The subscheme $X_h\subset U^{n,q}_h\subset\cR_{h,n,q}^\times$ is stable under conjugation by $\zeb$, so we obtain an action of $\cR_{h,n,q}^\times(\fqn)$ on $X_h$ determined by the requirements that $\zeb$ acts by conjugation and $U^{n,q}_h(\fqn)$ acts by right multiplication.

\mbr

This action commutes with the left action of $U^1_L/U^h_L$ on $X_h$, so we obtain a representation $\rho'_\chi$ of $\cR_{h,n,q}^\times(\fqn)$ in $H^r_c(X_h,\ql)[\chi]$. By construction, $\rho'_\chi$ is an extension of $\rho_\chi$. It remains to check that $\Tr(\rho'_\chi(\zeb))=(-1)^r$.

\mbr

As in the proof of Proposition \ref{p:trace-zeta-bar}, we can rewrite the desired identity as
\begin{equation}\label{e:proof-p:main-reduction-1}
\sum_i (-1)^i \Tr\bigl(\zeb^*,H^i_c(X_h,\ql)[\chi]\bigr) = 1
\end{equation}
(because $H^i_c(X_h,\ql)[\chi]=0$ for all $i\neq r$ by assumption) where $\zeb^*$ is the automorphism of $H^i_c(X_h,\ql)$ induced by conjugation by $\zeb$. For each $\ga\in U^1_L/U^h_L$ write $(\zeb,\ga)$ for the automorphism of $X_h$ given by $x\mapsto \zeb(\ga*x)\zeb^{-1}$ (we recall that ``$*$'' denotes the left action of $U^1_L/U^h_L$ on $X_h$ constructed in Definition \ref{d:left-action}). As the $\zeb$-conjugation action commutes with the left action of $U^1_L/U^h_L$, we see that \eqref{e:proof-p:main-reduction-1} is equivalent to
\begin{equation}\label{e:proof-p:main-reduction-2}
\sum_{\ga\in U^1_L/U^h_L} \chi(\ga)^{-1} \sum_i (-1)^i \Tr\bigl((\zeb,\ga)^*,H^i_c(X_h,\ql)\bigr) = \abs{U^1_L/U^h_L}=q^{n(h-1)}.
\end{equation}
Theorem \ref{t:fixed-point-formula-DL} implies that for each $\ga\in U^1_L/U^h_L$,
\[
\sum_i (-1)^i \Tr\bigl((\zeb,\ga)^*,H^i_c(X_h,\ql)\bigr) = \sum_i (-1)^i \Tr\bigl(\ga^*,H^i_c(X^\zeb_h,\ql)\bigr)
\]
where $X^\zeb_h\subset X_h$ is the subvariety of points invariant under $\zeb$-conjugation.

\mbr

Let us calculate $X^\zeb_h$. Since $\zeb$ is a generator of $\fqn^\times$, we have $\zeb^{q^j}=\zeb$ if and only if $n$ divides $j$. Therefore $X^\zeb_h$ can be identified with the subvariety of all elements of $U^{n,q}_h$ of the form $1+a_n\tau^n+a_{2n}\tau^{2n}+\dotsc+a_{n(h-1)}\tau^{n(h-1)}$ for which the determinant of the diagonal matrix $\operatorname{diag}(x,\vp(x),\dotsc,\vp^{n-1}(x))$ belongs to $\bF_q[\pi]/(\pi^h)$, where $x:=1+a_n\pi+a_{2n}\pi^2+\dotsc+a_{n(h-1)}\pi^{h-1}$. The last requirement is equivalent to $\vp^n(x)=x$, so we see that $X^\zeb_h=\bigl\{1+a_n\tau^n+a_{2n}\tau^{2n}+\dotsc+a_{n(h-1)}\tau^{n(h-1)}\st a_j\in\fqn\ \forall\,j\bigr\}$ can be naturally identified with the finite discrete set $U^1_L/U^h_L$, and the left action ``$*$'' of $U^1_L/U^h_L$ on $X^\zeb_h$ becomes identified with the left multiplication action of $U^1_L/U^h_L$ on itself. These observations immediately yield \eqref{e:proof-p:main-reduction-2}.

\subsubsection*{Step 2} The rest of the proof is very similar to the argument we used in \S\ref{sss:realize-as-induced-rep}. For each $i\geq 0$ we need to calculate the representation $R^i_{\bT,\te}=H_i(\Xt,\ql)[\te]$ of $D_{1/n}^\times$, where $H_i(\Xt,\ql)[\te]\subset H_i(\Xt,\ql)$ is the subspace on which $L^\times$ acts via $\te$.

\mbr

Since $\te$ has level $h$, we can identify $R^i_{\bT,\te}$ with $H_i(\Xt_h,\ql)$, where $\Xt_h$ is the disjoint union of the schemes $\Xt^{(m)}_h$ for $m\in\bZ$. The action of $D_{1/n}^\times\times L^\times$ on $\Xt_h$ factors through the discrete group $(D_{1/n}^\times/U^{n(h-1)+1}_D)\times(L^\times/U^h_L)$. In \S\ref{sss:higher-levels-main-construction} we described an embedding $\iota'_h:X_h\into\Xt^{(0)}_h$. It has the following properties, whose proofs are essentially identical to the proofs of the corresponding assertions of Lemma \ref{l:action-of-Gamma}.

\begin{lem}\label{l:action-of-Gamma-h}
Let $\Ga_h\subset(D_{1/n}^\times/U^{n(h-1)+1}_D)\times(L^\times/U^h_L)$ denote the normalizer of the image of $\iota'_h$. Then
 \sbr
\begin{enumerate}[$($a$)$]
\item $\Ga_h$ is the subgroup of $(D_{1/n}^\times/U^{n(h-1)+1}_D)\times(L^\times/U^h_L)$ generated by the subgroup $(U^1_D/U^{n(h-1)+1}_D)\times(U^1_L/U^h_L)$ and the elements $(\ze,\ze^{-1})$ and $(\pi,\pi^{-1})$.
 \sbr
\item If $g\in(D_{1/n}^\times/U^{n(h-1)+1}_D)\times(L^\times/U^h_L)$ and $g\not\in\Ga_h$, then $g(\iota'_h(X_h))\cap\iota'_h(X_h)=\varnothing$.
 \sbr
\item $\Xt_h$ is equal to the union of $(D_{1/n}^\times/U^{n(h-1)+1}_D)\times(L^\times/U^h_L)$-translates of $\iota'_h(X_h)$.
 \sbr
\item The action of $\Ga_h$ on $X_h$ that makes $\iota'_h$ equivariant is determined as follows:
     \sbr
    \begin{itemize}
    \item the action of the subgroup $(U^1_D/U^{n(h-1)+1}_D)\times(U^1_L/U^h_L)\subset\Ga_h$ on $X_h$ comes from the right multiplication action of $U^1_D/U^{n(h-1)+1}_D$ and the action ``$*$'' of $U^1_L/U^h_L$ on $X_h$ described in Definition \ref{d:left-action};
     \sbr
    \item the action of $(\ze,\ze^{-1})\in\Ga_h$ on $X_h$ corresponds to the $\zeb$-conjugation action considered in the first part of this proof;
     \sbr
    \item the action of $(\pi,\pi^{-1})\in\Ga_h$ on $X_h$ is trivial.
    \end{itemize}
\end{enumerate}
\end{lem}

\subsubsection*{Step 3} Lemma \ref{l:action-of-Gamma-h} implies that for each $i\geq 0$ we have a natural isomorphism
\[
H_i(\Xt_h,\ql) \cong \Ind_{\Ga_h}^{(D_{1/n}^\times/U^{n(h-1)+1}_D)\times(L^\times/U^h_L)} H_i(X_h,\ql).
\]
Observe also that $\dim X_h=\dim\Xt_h=(h-1)(n-1)$ by Lemma \ref{l:cohomology-X-h}, so that $H_i(X_h,\ql)\cong H_c^{2(h-1)(n-1)-i}(X_h,\ql)$ for all $i\geq 0$ (up to a Tate twist). Assertion (b) of Proposition \ref{p:main-reduction} follows from the previous calculations by an argument that is essentially identical to the one used in \S\ref{sss:end-proof-level-2}, so we skip the details.

\subsubsection*{Step 4} Assertion (c) of Proposition \ref{p:main-reduction} is proved by the same argument as the one used in Remark \ref{r:eta-theta-irreducible} in the special case $h=2$.

\subsection{Proof of Theorem \ref{t:main-example}}\label{ss:proof-t:main-example} The argument consists of several steps.

\subsubsection{Setup} We begin by specializing all the constructions of \S\ref{ss:higher-levels-strategy} to the case where $n=2$ and $h=3$. The unipotent group $U^{n,q}_h=U^{2,q}_3$ is the algebraic group over $\fqq$ defined as follows: for every $\fqq$-algebra $A$, the group $U^{2,q}_3(A)$ consists of expressions of the form $1+a_1\tau+a_2\tau^2+a_3\tau^3+a_4\tau^4$, where $a_j\in A$ and $\tau$ is a formal symbol, which are multiplied using distributivity and the rules $\tau^5=0$ and $\tau\cdot a=a^q\cdot\tau$ for all $a\in A$. The subscheme $X_3\subset U^{2,q}_3$ is defined by the equations
\begin{equation}\label{e:equations-defining-X-3}
a_2^q+a_2-a_1^{q+1}\in\bF_q \qquad\text{and}\qquad a_4^q+a_4 +a_2^{q+1}-a_1a_3^q-a_3a_1^q\in\bF_q,
\end{equation}
which follows from Example \ref{ex:n-2-h-3} and Remark \ref{r:n-2-h-3}. The group $U^{2,q}_3(\fqq)$ acts on $X_3$ by right multiplication, as always. The left action ``$*$'' of $U^1_L/U^3_L$ on $X_3$ described in Definition \ref{d:left-action} is given by the formula
\[
(1+\la\pi+\mu\pi^2)*(1+a_1\tau+a_2\tau^2+a_3\tau^3+a_4\tau^4) =
\]
\[
=
1+a_1\tau+(\la+a_2)\tau^2+(a_3+\la a_1)\tau^3+(\mu+a_4+\la a_2)\tau^4.
\]
for all $\la,\mu\in\fqq$.

\begin{rems}\label{rems:properties}
\begin{enumerate}[(1)]
\item The center of the group $\UU$ consists of all elements of the form $1+a_2\tau^2+a_4\tau^4$, where $a_2,a_4\in\fqq$.
 \sbr
\item In what follows we will identify the additive group $\fqq$ both with the subgroup $U^2_L/U^3_L\subset U^1_L/U^3_L$ and with the subgroup $\{1+a_4\tau^4\}\subset\UU$.
 \sbr
\item Under these identifications, the actions of $\fqq$ on $X_3$ induced by the left action of $U^1_L/U^3_L$ and by the right multiplication action of $\UU$ coincide.
\end{enumerate}
\end{rems}

\subsubsection{Filtration} The algebraic group $U^{2,q}_3$ has a natural filtration
\begin{equation}\label{e:filtration}
\{1\}\subset H_4\subset H_3\subset H_2\subset H_1=U^{2,q}_3
\end{equation}
where $H_j=\{1+a_j\tau^j+\dotsc+a_4\tau^4\}\subset U^{2,q}_3$. It will be used in what follows.

\subsubsection{Representations}\label{sss:step-1} Choose a character $\chi:U^1_L/U^3_L\rar{}\qls$ and let $\psi$ denote its restriction to $U^2_L/U^3_L\cong\fqq$. We will also identify $\fqq$ with the central subgroup $H_4(\fqq)$ of $U^{2,q}_3(\fqq)$. Define a character
\begin{equation}\label{e:psi-tilde}
\psit:H_3(\fqq)\rar{}\qls, \qquad 1+a_3\tau^3+a_4\tau^4\longmapsto\psi(a_4).
\end{equation}

\begin{lem}\label{l:analysis-irreps}
If $\psi$ has conductor $q^2$, then for every irreducible representation $\rho$ of $U^{2,q}_3(\fqq)$ on which $H_4(\fqq)$ acts via $\psi$, the restriction of $\rho$ to $H_3(\fqq)$ contains $\psit$ as a direct summand.
\end{lem}

\begin{proof}
It is enough to check that the conjugation action of $\UU$ on the set of all possible extensions of $\psi$ to a character $H_3(\fqq)\to\qls$ is transitive. To this end, let $\be\in\fqq$, put $g=1-\be\tau\in\UU$ and consider $h=1+a_3\tau^3+a_4\tau^4\in H_3(\fqq)$. A direct calculation shows that
\[
ghg^{-1} = 1+a_3\tau^3 + (a_4+\be^q a_3-\be a_3^q)\tau^4,
\]
and therefore
\[
\psit(ghg^{-1})=\psit(h)\psi(\be^q a_3-\be a_3^q).
\]
A character of $H_3(\fqq)$ is an extension of $\psi$ if and only if it has the form \[h=1+a_3\tau^3+a_4\tau^4\mapsto\psit(h)\nu(a_3)\] for some character $\nu:\fqq\to\qls$. It remains to show that every character of $\fqq$ can be written as $a\mapsto\psi(\be^q a-\be a^q)$ for some choice of $\be\in\fqq$. To see this, fix a nontrivial character $\psi_0:\bF_q\to\qls$. Then every character of $\fqq$ has the form $a\mapsto\psi_0\bigl(\Tr_{\fqq/\bF_q}(xa)\bigr)$ for some $x\in\fqq$. In particular, we can write $\psi(a)=\psi_0\bigl(\Tr_{\fqq/\bF_q}(xa)\bigr)$, where $x\in\fqq$ and $x\not\in\bF_q$ because $\psi$ has conductor $q^2$ by assumption. Then for any $\be\in\fqq$, we have
\begin{eqnarray*}
\psi(\be^q a-\be a^q) &=& \psi_0\bigl(\Tr_{\fqq/\bF_q}(x\be^q a-x\be a^q)\bigr) \\
&=& \psi_0\bigl(\Tr_{\fqq/\bF_q}(x\be^q a-x^{1/q}\be^{1/q} a)\bigr) \\
&=& \psi_0\bigl(\Tr_{\fqq/\bF_q}((x-x^{1/q})\be^q a)\bigr).
\end{eqnarray*}
Since $x\in\fqq\setminus\bF_q$, the map $\be\mapsto(x-x^{1/q})\be^q$ is a bijection of $\fqq$ onto itself, which completes the proof of Lemma \ref{l:analysis-irreps}.
\end{proof}

\begin{defin}\label{d:chi-sharp}
Given a character $\chi:U^1_L/U^3_L\rar{}\qls$, let $\chi^\sharp:H_2(\fqq)\rar{}\qls$ be the character given by $\chi^\sharp(1+a_2\tau^2+a_3\tau^3+a_4\tau^4)=\chi(1+a_2\pi+a_4\pi^2)$.
\end{defin}

\begin{cor}\label{c:structure-irreps}
Fix a character $\psi:\fqq\rar{}\qls$ of conductor $q^2$. The map
\[
\chi \longmapsto \Ind_{H_2(\fqq)}^{U^{2,q}_3(\fqq)} (\chi^\sharp)
\]
is a bijection between the set of characters $U^1_L/U^3_L\rar{}\qls$ whose restriction to $\fqq\cong U^2_L/U^3_L$ equals $\psi$ and the set of isomorphism classes of irreducible representations of $U^{2,q}_3(\fqq)$ on which $H_4(\fqq)\cong\fqq$ acts via $\psi$.
\end{cor}

\begin{proof}
If $\psit$ is given by \eqref{e:psi-tilde}, one can easily check that the normalizer of $\psit$ with respect to the conjugation action of $U^{2,q}_3(\fqq)$ is equal to $H_2(\fqq)$ (this fact relies on the assumption that $\psi$ has conductor $q^2$). Therefore if $\nu:H_2(\fqq)\rar{}\qls$ is any extension of $\psit$, then $\Ind_{H_2(\fqq)}^{U^{2,q}_3(\fqq)}(\nu)$ is irreducible and the map $\nu\mapsto\Ind_{H_2(\fqq)}^{U^{2,q}_3(\fqq)}(\nu)$ is a bijection between the set of all extensions of $\psit$ to $H_2(\fqq)$ and the set of isomorphism classes of irreducible representations of $\UU$ whose restriction to $H_3(\fqq)$ contains $\psit$. It remains to observe that a character of $H_2(\fqq)$ is an extension of $\psit$ if and only if it has the form $\chi^\sharp$ for some $\chi$ as in the statement of the corollary, and to apply Lemma \ref{l:analysis-irreps}.
\end{proof}

\subsubsection{Intertwiners} The following result will be proved in \S\ref{sss:proof-l:calculate-intertwiners}.

\begin{lem}\label{l:calculate-intertwiners}
Let $\psi:\fqq\rar{}\qls$ be a character whose conductor equals $q^2$, let $\psit:H_3(\fqq)\rar{}\qls$ be given by \eqref{e:psi-tilde} and put $V_\psit=\Ind_{H_3(\fqq)}^{U^{2,q}_3(\fqq)}(\psit)$. Then $\Hom_{U^{2,q}_3(\fqq)}\bigl(V_\psit,H^i_c(X_3,\ql)\bigr)=0$ for $i\neq 2$, $\dim\Hom_{U^{2,q}_3(\fqq)}\bigl(V_\psit,H^2_c(X_3,\ql)\bigr)=q^2$ and the Frobenius $\Fr_{q^2}$ acts on $\Hom_{U^{2,q}_3(\fqq)}\bigl(V_\psit,H^2_c(X_3,\ql)\bigr)$ via the scalar $q^2$.
\end{lem}

\begin{rem}\label{r:auxiliary}
Using Lemma \ref{l:analysis-irreps} and the proof of Corollary \ref{c:structure-irreps}, we see that the representation $V_\psit$ is the direct sum of all irreducible representations of $\UU$ on which $H_4(\fqq)$ acts via $\psi$.
\end{rem}

\begin{rem}\label{r:Y-3-is-complicated}
Recall from Remark \ref{r:X-h-is-a-preimage} that we can write $X_3=L_{q^2}^{-1}(Y_3)$ for some closed subscheme $Y_3\subset U^{2,q}_3$, where $L_{q^2}:U^{2,q}_3\to U^{2,q}_3$ is the Lang map $g\mapsto L_{q^2}(g)g^{-1}$. In principle, one can compute it explicitly and find that
\[
Y_3 = \bigl\{ 1+b_1\tau+b_2\tau^2+b_3\tau^3+b_4\tau^4\in U^{2,q}_3 \st b_2=0, \ b_4=-b_3 b_1^q\bigr\}.
\]
However, the calculations one needs to perform are rather tedious. Therefore in \S\ref{sss:proof-l:calculate-intertwiners} below we use an alternate idea, which allows us to apply Proposition \ref{p:DL-main} without calculating $Y_3$ explicitly.
\end{rem}

\subsubsection{Proof of Lemma \ref{l:calculate-intertwiners}}\label{sss:proof-l:calculate-intertwiners} We place ourselves in a situation in which Proposition \ref{p:DL-main} can be applied as follows:

\begin{itemize}
\item We replace $q$ with $q^2$ and take $G=U^{2,q}_3$ and $H=H_3\subset U^{2,q}_3$.
 \sbr
\item We identify the homogeneous space $U^{2,q}_3/H_3$ with the affine plane $\bA^2$ via the projection $U^{2,q}_3\rar{}\bA^2$ onto the first two coordinates and define the section $s:\bA^2\rar{}U^{2,q}_3$ by $s(a_1,a_2)=1+a_1\tau+a_2\tau^2$.
 \sbr
\item We define $f:H_3\rar{}\bG_a$ by $1+a_3\tau^3+a_4\tau^4\mapsto a_4$.
 \sbr
\item We let $\psi:\fqq\rar{}\qls$ be the character in the statement of Lemma \ref{l:calculate-intertwiners}.
 \sbr
\item We take $Y=Y_3\subset U^{2,q}_3$, cf.~Remark \ref{r:Y-3-is-complicated}.
\end{itemize}

The morphism $\be:\bA^2\times H_3\rar{}U^{2,q}_3$ defined in Proposition \ref{p:DL-main} is given by $\be(x,h)=s(F_{q^2}(x))hs(x)^{-1}$. In order to apply that proposition we must calculate the preimage $\be^{-1}(Y_3)\subset \bA^2\times H_3$. In fact, we claim that if we write \[\bA^2\times H_3=\bigl\{ (a_1,a_2,1+a_3\tau^3+a_4\tau^4)\bigr\}\] and use $a_1,a_2,a_3,a_4$ as coordinates on $\bA^2\times H_3$, then $\be^{-1}(Y_3)\subset \bA^2\times H_3$ is defined by the two equations
\begin{equation}\label{e:first}
a_2^{q^2}-a_2 = a_1^{q+q^2}-a_1^{1+q},
\end{equation}
\begin{equation}\label{e:second}
a_4=a_1^q a_3 - a_1^{q^2}a_3^q + a_2^{1+q}-a_2^{q+q^2}.
\end{equation}
Indeed, it suffices to verify this claim at the level of $\bfq$-points. So let $x=(a_1,a_2)\in\bA^2(\bfq)$ and $h=1+a_3\tau^3+a_4\tau^4\in H_3(\bfq)$. We can write $a_3=y_3^{q^2}-y_3$ and $a_4=y_4^{q^2}-y_4$ for some $y_3,y_4\in\bfq$, so that $h=L_{q^2}(1+y_3\tau^3+y_4\tau^4)$ and hence
\[
\be(x,h)=F_{q^2}(s(x))\cdot L_{q^2}(1+y_3\tau^3+y_4\tau^4)\cdot s(x) = L_{q^2}\bigl(s(x)\cdot(1+y_3\tau^3+y_4\tau^4)\bigr).
\]
We see that $\be(x,h)\in Y_3$ if and only if $s(x)\cdot(1+y_3\tau^3+y_4\tau^4)\in X_3$. Now
\[
s(x)\cdot(1+y_3\tau^3+y_4\tau^4) = 1+a_1\tau+a_2\tau^2+y_3\tau^3+(y_4+a_1y_3^q)\tau^4,
\]
so that $s(x)\cdot(1+y_3\tau^3+y_4\tau^4)\in X_3$ if and only if $a_2^q+a_2-a_1^{q+1}\in\bF_q$ and \[(y_4+a_1y_3^q)^q+(y_4+a_1y_3^q) +a_2^{q+1}-a_1y_3^q-y_3a_1^q\in\bF_q.\]
One checks by a direct calculation that these requirements are equivalent to the system \eqref{e:first}--\eqref{e:second} using the identities $a_3=y_3^{q^2}-y_3$ and $a_4=y_4^{q^2}-y_4$.

\mbr

Using Proposition \ref{p:DL-main}, the proof of Lemma \ref{l:calculate-intertwiners} reduces to the following calculation. Let $S\subset\bA^3$ be the subscheme defined by the equation \eqref{e:first}, let $P:\bA^3\rar{}\bG_a$ be given by $(a_1,a_2,a_3)\mapsto a_1^q a_3 - a_1^{q^2}a_3^q + a_2^{1+q}-a_2^{q+q^2}$ and let $\cL_\psi$ be the Artin-Schreier local system on $\bG_a$ over $\fqq$ defined by the character $\psi:\fqq\rar{}\qls$. We must prove that if $\psi$ has conductor $q^2$, then $H^i_c(S,P^*\cL_\psi)=0$ for $i\neq 2$ and $H^2_c(S,P^*\cL_\psi)$ is a $q^2$-dimensional vector space on which $\Fr_{q^2}$ acts by the scalar $q^2$. To this end, observe that $P(a_1,a_2,a_3)=a_1^q a_3 - a_1^{q^2}a_3^q+P_2(a_1,a_2)$, where $P_2(a_1,a_2)=a_2^{1+q}-a_2^{q+q^2}$ does not depend on $a_3$, and $S$ is the preimage of the subscheme $S_2\subset\bA^2$ defined by \eqref{e:first} under the projection $\bA^3\rar{}\bA^2$ along the third coordinate. Proposition \ref{p:inductive-idea} shows that $H^i_c(S,P^*\cL_\psi)\cong H^{i-2}_c(S_3,P_3^*\cL_\psi)(-1)$ for all $i$, where $S_3\subset S_2$ is the subscheme defined by $a_1=0$ and $P_3(a_1,a_2)=a_2^{1+q}-a_2^{q+q^2}$. But $S_3$ is the finite discrete set consisting of points of the form $(0,a_2)$ with $a_2\in\fqq$ and $P_3(0,a_2)=0$ for all $a_2\in\fqq$, so $H^0_c(S_3,P_3^*\cL_\psi)$ has dimension $q^2$ and the trivial action of $\Fr_{q^2}$, and $H^j_c(S_3,P_3^*\cL_\psi)=0$ for all $j\geq 1$. This completes the proof of Lemma \ref{l:calculate-intertwiners}. \qed

\subsubsection{Proof of Theorem \ref{t:main-example}(a)} Let $\chi:U^1_L/U^3_L\rar{}\qls$ be a character and let $\psi$ denote its restriction to $U^2_L/U^3_L\cong\fqq$. Assume that $\psi$ has conductor $q^2$ and consider, for each $i\geq 0$, the subspace $H^i_c(X_3,\ql)[\chi]\subset H^i_c(X_3,\ql)$ on which $U^1_L/U^3_L$ acts via $\chi$. By Remark \ref{rems:properties}(3), the subgroup $H_4(\fqq)\subset\UU$ acts on $H^i_c(X_3,\ql)[\chi]$ via $\psi$. By Lemma \ref{l:calculate-intertwiners} and Remark \ref{r:auxiliary}, we have $H^i_c(X_3,\ql)[\chi]=0$ for all $i\neq 2$, which implies the first assertion of Conjecture \ref{conj:reduction-to-X_h}.
The proof of the second assertion is based on the next lemma. Given characters $\chi:U^1_L/U^3_L\to\qls$ and $\nu:H_2(\fqq)\to\qls$ and an integer $i\geq 0$, we write $H^i_c(X_3,\ql)_{\chi,\nu}$ for the subspace of $H^i_c(X_3,\ql)$ on which $U^1_L/U^3_L$ and $H_2(\fqq)$ act via $\chi$ and $\nu$, respectively.

\begin{lem}\label{l:calculate-eigenspaces}
Let $\chi_1,\chi_2:U^1_L/U^3_L\rar{}\qls$ be characters whose restriction to $U^2_L/U^3_L\cong\fqq$ have conductor $q^2$. Then with the notation of Definition \ref{d:chi-sharp},
\[
\dim H^2_c(X_3,\ql)_{\chi_1,\chi_2^\sharp} =
\begin{cases}
1 & \text{if } \chi_1=\chi_2, \\
0 & \text{if } \chi_1\neq\chi_2.
\end{cases}
\]
\end{lem}

\begin{proof}
The observations in the previous paragraph imply the assertion of the lemma if $\chi_1\bigl\lvert_{U^2_L/U^3_L}\neq\chi_2\bigl\lvert_{U^2_L/U^3_L}$. Therefore we assume from now on that the two restrictions agree and denote them by $\psi:\fqq\to\qls$. The previous paragraph also shows that $H^i_c(X_3,\ql)_{\chi_1,\chi_2^\sharp}$ for all $i\neq 2$, and Lemma \ref{l:calculate-intertwiners} implies that $\Fr_{q^2}$ acts on $H^2_c(X_3,\ql)_{\chi_1,\chi_2^\sharp}$ by the scalar $q^2$. Lemma \ref{l:fixed-point-2} yields
\begin{equation}\label{e:proof-eigenspaces-1}
\dim H^2_c(X_3,\ql)_{\chi_1,\chi_2^\sharp} = q^{-12} \sum_{\ga,g} \chi_1(\ga)^{-1}\chi_2^\sharp(g) N(\ga,g),
\end{equation}
where the sum ranges over all $\ga\in U^1_L/U^3_L$, $g\in H_2(\fqq)$ and we write
\[
N(\ga,g)=\Bigl\lvert\bigl\{ x\in X_3(\bfq) \st \ga*F_{q^2}(x)=x\cdot g \bigr\} \Bigr\rvert.
\]
To calculate $N(\ga,g)$ we introduce the following notation. We will write $x\in X_3(\bfq)$ as $x=1+a_1\tau+a_2\tau^2+a_3\tau^3+a_4\tau^4$ and recall that $a_1,a_2,a_3,a_4$ are constrained by the conditions \eqref{e:equations-defining-X-3}. We also write $\ga=1+\la\pi+\mu\pi^2$ and $g=1+\la'\tau^2+\be\tau^3+\mu'\tau^4$, where $\la,\mu,\la',\be,\mu'\in\fqq$. A direct calculation shows that the equation $\ga*F_{q^2}(x)=x\cdot g$ amounts to the following system of four equations:
\begin{equation}\label{e:(1)}
a_1^{q^2}=a_1,
\end{equation}
\begin{equation}\label{e:(2)}
a_2+\la'=\la+a_2^{q^2},
\end{equation}
\begin{equation}\label{e:(3)}
a_3^{q^2}+\la a_1=a_3+\la'^q a_1+\be,
\end{equation}
\begin{equation}\label{e:(4)}
a_4^{q^2}+\mu+\la a_2^{q^2}=a_4+a_2\la'+\mu'+a_1\be^q.
\end{equation}
Hence $N(\ga,g)$ equals the number of quadruples $a_1,a_2,a_3,a_4\in\bfq$ satisfying the system \eqref{e:(1)}--\eqref{e:(4)} together with the conditions
\begin{equation}\label{e:(5)}
a_2^q+a_2-a_1^{q+1}\in\bF_q,
\end{equation}
\begin{equation}\label{e:(6)}
a_4^q+a_4 +a_2^{q+1}-a_1a_3^q-a_3a_1^q\in\bF_q.
\end{equation}
We now simplify this system of equations. Condition \eqref{e:(1)} amounts to $a_1\in\fqq$, in which case $a_1^{q+1}\in\bF_q$, so that \eqref{e:(5)} is equivalent to the condition that $a_2^q+a_2\in\bF_q$, which is the same as $a_2\in\fqq$. We see that \eqref{e:(1)} and \eqref{e:(5)} together are equivalent to requiring that $a_1,a_2\in\fqq$. We assume this from now on.

\mbr

Condition \eqref{e:(2)} becomes $\la'=\la$. We find that $N(\ga,g)=0$ when $\la'\neq\la$. From now on we suppose that $\la=\la'$, in which case \eqref{e:(2)} follows automatically from \eqref{e:(1)} and \eqref{e:(5)}. It remains to analyze conditions \eqref{e:(3)}, \eqref{e:(4)}, \eqref{e:(6)}. Since $a_2^{q+1}\in\bF_q$, \eqref{e:(6)} is equivalent to the condition that $a_4^q+a_4-a_1a_3^q-a_3a_1^q$ belongs to the kernel of the map $a\mapsto a^q-a$, which can be rewritten as
\begin{equation}\label{e:(6')}
a_4^{q^2}-a_4 = a_1^q(a_3^{q^2}-a_3)
\end{equation}
because $a_1\in\fqq$. We also rewrite \eqref{e:(3)}, \eqref{e:(4)} as
\begin{equation}\label{e:(3')}
a_3^{q^2}-a_3=(\la^q-\la)a_1+\be,
\end{equation}
\begin{equation}\label{e:(4')}
a_4^{q^2}-a_4 = \mu'-\mu+a_1\be^q,
\end{equation}
where we used the assumption that $\la=\la'$ and $a_2\in\fqq$. Substituting \eqref{e:(3')}, \eqref{e:(4')} into \eqref{e:(6')}, we find that the system \eqref{e:(6')}--\eqref{e:(4')} is equivalent to the system consisting of equations \eqref{e:(3')}, \eqref{e:(4')} and the equation
\begin{equation}\label{e:(6'')}
a_1^q\cdot\bigl((\la^q-\la)a_1+\be\bigr) = \mu'-\mu+a_1\be^q.
\end{equation}
Reversing the steps above, we see that if $\la'=\la$ and we choose $a_1\in\fqq$ satisfying \eqref{e:(6'')}, then we can solve \eqref{e:(3')}, \eqref{e:(4')} for $a_3,a_4\in\bfq$ (there are $q^2$ possibilities for each of $a_3,a_4$) and choose an arbitrary $a_2\in\fqq$, and obtain a solution of \eqref{e:(1)}--\eqref{e:(6)}. Moreover, all solutions of \eqref{e:(1)}--\eqref{e:(6)} are obtained in this way. Thus
\begin{equation}\label{e:proof-eigenspaces-2}
\begin{split}
& \sum_{\ga,g} \chi_1(\ga)^{-1}\chi_2^\sharp(g) N(\ga,g) = \\ & q^6\cdot \sum_{\la,\mu,\mu',\be\in\fqq} \chi_1(1+\la\pi+\mu\pi^2)^{-1}\chi_2(1+\la\pi+\mu'\pi^2) N'(\la,\mu,\mu',\be),
\end{split}
\end{equation}
where $N'(\la,\mu,\mu',\be)$ is the number of elements $a_1\in\fqq$ that satisfy \eqref{e:(6'')}.

\mbr

We can rewrite \eqref{e:(6'')} as follows:
\begin{equation}\label{e:(6''')}
a_1^{q+1} (\la^q-\la) + (\be a_1^q-\be^q a_1) = \mu'-\mu.
\end{equation}
If $a_1,\la,\be\in\fqq$, then $a_1^{q+1}$ and hence the left hand side of \eqref{e:(6''')} belongs to the kernel of the trace map $\Tr_{\fqq/\bF_q}:\fqq\rar{}\bF_q$. So if $\Tr_{\fqq/\bF_q}(\mu'-\mu)\neq 0$, then $N'(\la,\mu,\mu',\be)=0$. Recalling the definition of $\psi:\fqq\to\qls$, we see that the right hand side of \eqref{e:proof-eigenspaces-2} can be rewritten as
\[
q^6 \sum_{\la,\mu\in\fqq} \frac{\chi_2(1+\la\pi+\mu\pi^2)}{\chi_1(1+\la\pi+\mu\pi^2)} \cdot \left( \sum_{\de\in\Ker(\Tr_{\fqq/\bF_q})} \psi(\de) \cdot \sum_{\be\in\fqq} N''(\la,\de,\be),
\right)
\]
where $N''(\la,\de,\be)$ is the number of elements $a_1\in\fqq$ that satisfy
\begin{equation}\label{e:(6-4)}
a_1^{q+1} (\la^q-\la) + (\be a_1^q-\be^q a_1) = \de.
\end{equation}
Now given $\la,\de$ as above, $\sum_{\be\in\fqq} N''(\la,\de,\be)$ equals the number of pairs $(a_1,\be)\in\fqq$ satisfying \eqref{e:(6-4)}. If $a_1=0$, then \eqref{e:(6-4)} can only hold for $\de=0$, in which case $\be\in\fqq$ can be chosen arbitrarily. If $a_1\in\fqq^\times$, one easily checks that for each choice of $\la\in\fqq$ and $\de\in\Ker(\Tr_{\fqq/\bF_q})$, there are $q$ choices of $\be\in\fqq$ satisfying \eqref{e:(6-4)}, so
\begin{equation}\label{e:proof-eigenspaces-3}
\sum_{\be\in\fqq} N''(\la,\de,\be) = \begin{cases}
q^2+(q^2-1)q & \text{if }\de=0, \\
(q^2-1)q & \text{if }\de\neq 0.
\end{cases}
\end{equation}
Since $\psi:\fqq\to\qls$ has conductor $q^2$, its restriction to $\Ker(\Tr_{\fqq/\bF_q})$ is nontrivial, which implies that $\sum_{\de\in\Ker(\Tr_{\fqq/\bF_q})} \psi(\de)=0$, and therefore, by \eqref{e:proof-eigenspaces-3},
\[
\sum_{\de\in\Ker(\Tr_{\fqq/\bF_q})} \psi(\de) \cdot \sum_{\be\in\fqq} N''(\la,\de,\be) = q^2.
\]
Substituting the last identity into \eqref{e:proof-eigenspaces-2} yields
\[
\sum_{\ga,g} \chi_1(\ga)^{-1}\chi_2^\sharp(g) N(\ga,g) = q^2\cdot \sum_{\ga\in U^1_L/U^3_L} \chi_1(\ga)^{-1}\chi_2(\ga).
\]
Finally, since $U^1_L/U^3_L$ is a group of order $q^4$, we have
\[
\sum_{\ga\in U^1_L/U^3_L} \chi_1(\ga)^{-1}\chi_2(\ga) = \begin{cases}
q^4 & \text{if } \chi_1=\chi_2, \\
0 & \text{if } \chi_1\neq\chi_2,
\end{cases}
\]
which in view of \eqref{e:proof-eigenspaces-1} finishes the proof of Lemma \ref{l:calculate-eigenspaces}.
\end{proof}

Now we complete the proof of part (a) of Theorem \ref{t:main-example}. Lemma \ref{l:calculate-eigenspaces} implies that the restriction of $H^2_c(X_3,\ql)[\chi]$ to $H_2(\fqq)$ contains the character $\chi^\sharp$ (see Definition \ref{d:chi-sharp}) with multiplicity $1$ and does not contain any other character $\nu:H_2(\fqq)\to\qls$ that extends $\psit$. Corollary \ref{c:structure-irreps} implies that $H^2_c(X_3,\ql)[\chi]\cong\Ind_{H_2(\fqq)}^{U^{2,q}_3(\fqq)} (\chi^\sharp)$ is an irreducible representation of $\UU$. \qed

\subsubsection{Proof of Theorem \ref{t:main-example}(b)} We use Proposition \ref{p:main-reduction}, whose application is justified by part (a) of Theorem \ref{t:main-example}. We see that $R^i_{\bT,\te}=0$ for $i\neq 2$, and to calculate $R^2_{\bT,\te}$ we need to determine the representation $\eta_\te^\circ$ defined in Proposition \ref{p:main-reduction}(a). To this end, we identify $\cO_D^\times/U_D^5$ with the multiplicative group $\cR_{3,2,q}^\times(\fqq)$ via \eqref{e:identify-U-h} and let $\zeb\in\cR_{3,2,q}^\times(\fqq)$ be the element corresponding to the generator $\ze\in\fqq^\times\subset\cO_D^\times/U_D^5$. Thus $\cR_{3,2,q}^\times(\fqq)$ can be identified with the semidirect product $\langle\zeb\rangle\ltimes\UU$, where $\langle\zeb\rangle$ is the cyclic subgroup generated by $\zeb$.

\mbr

Recall that $\rho_\chi$ denotes the representation of $U^1_D/U^5_D\cong\UU$ in $H^2_c(X_3,\ql)[\chi]$, where $\chi:U^1_L/U^3_L\to\qls$ is the character induced by $\te$. The representation $\rho_\chi$ was calculated in the proof of part (a) of Theorem \ref{t:main-example}: with the notation of Definition \ref{d:chi-sharp}, we have $\rho_\chi\cong\Ind_{H_2(\fqq)}^{\UU}(\chi^\sharp)$ as representations of $\UU$. The character $\chi^\sharp:H_2(\fqq)\to\qls$ is invariant under conjugation by $\zeb$, so it extends to a character $\chi':\langle\zeb\rangle\cdot H_2(\fqq)\to\qls$ such that $\chi'(\zeb)=1$. Using the Frobenius formula for the character of an induced representation, one easily finds that the trace of $\zeb$ in $\Ind_{\langle\zeb\rangle\cdot H_2(\fqq)}^{\cR_{3,2,q}^\times(\fqq)}(\chi')$ is equal to $1$. Therefore, by uniqueness, the representation of $\cR_{3,2,q}^\times(\fqq)$ corresponding to the representation $\eta^\circ_\te$ described in Proposition \ref{p:main-reduction}(a) is given by $\Ind_{\langle\zeb\rangle\cdot H_2(\fqq)}^{\cR_{3,2,q}^\times(\fqq)}(\te^\sharp)$, where $\te^\sharp:\langle\zeb\rangle\cdot H_2(\fqq)\to\qls$ is the character
\[
\te^\sharp\bigl(\ze^k\cdot(1+a_2\tau^2+a_3\tau^3+a_4\tau^4)\bigr)=\te\bigl(\ze^k\cdot(1+a_2\pi+a_4\pi^2)\bigr).
\]
Theorem \ref{t:main-example}(b) follows easily from this observation and Proposition \ref{p:main-reduction}(b). \qed

\bibliographystyle{alpha}
\bibliography{DLtheory}

\end{document}